\documentclass[10pt,a4paper]{amsart}

\usepackage[pdftex]{graphicx}
\usepackage[pdftex]{xcolor}

\usepackage[colorlinks=true]{hyperref}

\usepackage{graphicx}

\usepackage{mathtools,  stmaryrd}
\usepackage{xparse} \DeclarePairedDelimiterX{\Iintv}[1]{\llbracket}{\rrbracket}{\iintvargs{#1}}
\NewDocumentCommand{\iintvargs}{>{\SplitArgument{1}{,}}m}
{\iintvargsaux#1} 
\NewDocumentCommand{\iintvargsaux}{mm} {#1\mkern1.5mu,\mkern1.5mu#2}

\usepackage{cancel}
\usepackage[normalem]{ulem}

\makeatletter
\newtheorem*{rep@theorem}{\rep@title}
\newcommand{\newreptheorem}[2]{
\newenvironment{rep#1}[1]{
 \def\rep@title{#2 \ref{##1}}
 \begin{rep@theorem}}
 {\end{rep@theorem}}}
\makeatother
\usepackage{verbatim}
\usepackage{amsmath}
\usepackage{amssymb, mathrsfs}
\usepackage{amsbsy}

\usepackage{amscd}
\usepackage{amsthm}
\usepackage{enumitem}
\usepackage[english]{babel}
\usepackage{todonotes}
\usepackage[T1]{fontenc}

\newcommand{\E}{\mathbb{E}}
\newcommand{\Z}{\mathbb{Z}}

\newcommand{\I}{\mathbb{I}}

\newcommand{\kC}{\mathcal{C}}

\newcommand{\rmP}{\mathrm{P}}

\newcommand{\rme}{{\mathrm{E}}}

\newcommand{\cvlaw}{\stackrel{{ (d)}}{\longrightarrow}}

\newcommand{\lin}{\left[\kern-0.15em\left[}
\newcommand{\rin} {\right]\kern-0.15em\right]}
\newcommand{\linf}{[\kern-0.15em [}
\newcommand{\rinf} {]\kern-0.15em ]}
\newcommand{\ilin}{\left]\kern-0.15em\left]}
\newcommand{\irin} {\right[\kern-0.15em\right[}

\def\al#1{\begin{align*}#1\end{align*}}
\def\aln#1{\begin{align}#1\end{align}}

\usepackage{constants}

\usepackage{bbold}

\newconstantfamily{c}{symbol=c}

\newconstantfamily{a}{symbol=\alpha}

\newcommand{\secno}[1]{\thesection.\arabic{#1}}
\newconstantfamily{kE}{
symbol=\mathcal{E},
format=\secno,
reset={section}
}

\renewcommand{\tilde}{\widetilde}

\newtheorem{lem}{Lemma}[section]
\newtheorem{remark}[lem]{Remark}
\newtheorem{prop}[lem]{Proposition}
\newtheorem{thm}[lem]{Theorem}

\newtheorem{cor}[lem]{Corollary}

\newtheorem {defin}[lem] {Definition}
\newtheorem {rem}[lem] {Remark}

\newcounter{assu}
\usepackage{color}
\definecolor{lilas}{RGB}{182, 102, 210}

\newcommand{\CO}{\color{red}}

\numberwithin{equation}{section}

\title[High moments of two dimensional directed polymers up to quasi-criticality]
{high moments of 2d directed polymers up to quasi-criticality}
\date{\today}
\author{Cl\'ement Cosco} 
\address[Cl\'ement Cosco]
{Weizmann Institute of Science, Israel.}
\email{clement.cosco@gmail.com}

\author{Shuta Nakajima} 
\address[Shuta Nakajima]
{Meiji University, Kanagawa, Japan}
\email{njima@meiji.ac.jp}

\begin{document}

\begin{abstract}
We consider two-dimensional directed polymers in random environment in the sub-critical regime and in the quasi-critical regime introduced recently by \textit{Caravenna, Cottini and Rossi,  arXiv:2307.02453v1}. 
For $q\leq q_N$ with $q_N\to\infty$ diverging at a suitable rate with the size of the system, we obtain upper bound estimates on the $q$-moment of the partition function for general environments. In the sub-critical regime, our results improve the $q_N$-threshold obtained for Gaussian environment in \textit{Cosco, Zeitouni, Comm. Math. Phys (2023)}.  As a corollary, we derive large deviation estimates with a Gaussian rate function. 
\end{abstract}
 
\maketitle

\section{Introduction}
\subsection{The model}
Let $\omega(n,x)$, $n \in \mathbb{N}$, $x \in \mathbb{Z}^2$, be independent and identically distributed (i.i.d.)\ with common distribution $\omega(n,x)\sim \mu$ of mean $0$, variance $1$, and finite exponential moments. Define $\Lambda(\beta) := \Lambda^{\mu}(\beta) := \log \mathbb{E}[e^{\beta \omega(n,x)}] \in \mathbb{R}_+$ for all $\beta \in [0,\infty)$, and $\Lambda_p(\beta) := \Lambda^{\mu}_p(\beta) := \Lambda(p\beta) - p\Lambda(\beta)$ for $p \geq 2$.
 
Let $(S_n)_{n\geq 0}$ be a simple random walk on $\mathbb Z^2$ and write $p_n(x) := \rmP_0(S_n=x)$, {where $\rmP_x$ and ${\rm E}_x$ stand for the probability and expectation of the simple random walk starting at $x\in\mathbb{Z}^2$}. Define the \emph{directed polymer partition function} by
\[
W_N(x,\beta) := {\rm E}_x\left[ e^{\sum_{n=1}^N (\beta \omega(n,x) - \Lambda(\beta))}\right], \quad W_N(\beta) := W_N(0,\beta).
\] 

For $q\in \mathbb N$, let 
$(S^1,\dots,S^q)$ be independent copies of the simple random walk and denote by ${\rm E}_X^{\otimes q}$ the expectation started at $X=(x_1,\dots,x_q) \in (\mathbb Z^2)^q$.  
Set (see \cite[Theorem 2.1.3]{LL10})
\[
R_N := \rme^{\otimes 2}_0\sum_{n=1}^N \mathbf 1_{S_n^1=S_n^2} = \pi^{-1} {\log N} + \mathcal O_N(1).\]
Rescaling the temperature as $\beta_N^2 := {\hat \beta^2}/{R_N} \to 0$ for $\hat \beta \in \mathbb R_+$, it is known from the seminal work of Caravenna, Sun and Zygouras \cite{CSZAAP} that for all $\hat \beta \geq 1$, $W_N(\beta_N)\to 0$ in probability while in the \textit{sub-critical regime} $\hat \beta <1$,
\begin{equation} \label{eq:CLT}
	W_N(\beta_N) \cvlaw e^{X}, \quad X\sim \mathcal N(-\lambda^2/2,\lambda^2),
\end{equation}
Here, $\mathcal N(\mu,\sigma^2)$ denotes the normal distribution and $\lambda^2 := \lambda^2(\hat \beta) :=  \log \frac{1}{1-\hat \beta^2}$. See also \cite{CD24,CaCo22} for simplified proofs of the result, \cite{DG22} in the non-linear setting. See \cite{Z24} for a recent review on the model of directed polymers.

It is possible to zoom on the transition between $\hat \beta < 1$ and $\hat \beta= 1$ 
in the so-called region of \textit{quasi-criticality} introduced by Caravenna, Cottini and Rossi in \cite{CCR23}. Setting $\sigma_N^2:=e^{\Lambda_2(\beta_N)}-1$, it is defined as the temperature region where  $\sigma_N^2 = (1- \frac{\vartheta_N}{\log N})/{R_N}$ with  $1\ll \vartheta_N\ll \log N$.  (Taking $\vartheta_N = (1-\hat \beta^2) \log N$ leads to the sub-critical regime, and $\vartheta_N = \mathcal O(1)$ to the \emph{critical} regime.  We refer to the lecture notes \cite{CSZ2dSHFnotes} regarding the critical point and mention its connection to the 2D stochastic heat flow \cite{CaSuZyCrit21}.)

In this paper, we derive upper bounds on the moments $\mathbb{E}[W_N(\beta_N)^q]$ both in the sub-critical and the quasi-critical regimes.  Namely, we show that up to some threshold $q_N\to\infty$, for all $q\leq q_N$, the $q$-moment of $W_N(\beta_N)$ is dominated (up to leading order) by the moment of the exponential of a Gaussian random variable. As a corollary, we obtain large deviation estimates on $\log W_N(\beta_N)$ with a Gaussian rate function. 
In the sub-critical regime, we improve by a constant multiplicative factor the maximal $q_N$-value derived in \cite{Cosco2021MomentsUB}. This refinement turns out to be crucial (and sharp) for application to extreme value statistics of the field $\sup_{x\in [-1,1]^2} \log W_N(\beta_N,x\sqrt N)$, see \cite{CNZ25}. Moreover, we extend the assumption of Gaussian weights of \cite{Cosco2021MomentsUB} to general weights with finite exponential moments. In the quasi-critical regime, the results are new. (See also the recent paper \cite{LZ24} about finite moment \emph{at criticality}.)

We refer to \cite{CNZ25} regarding the motivation of studying high moments at sub-criticality through the connection to log-correlated fields. Preliminary calculations suggest that quasi-critical moment estimates may be involved in extreme value statistics at criticality.

\section{Results}       
Define \textit{sub-criticality} (SC)  as the regime of parameters given by
\begin{equation} 
\tag{SC}
\label{eq:sub-criticality}
\beta_N^2 := \frac{\hat{\beta}^2}{R_N}, \quad  \hat \beta \in (0,1).
\end{equation}
With
\begin{equation} \label{eq:def_sigmaNandbarsigma}
	\sigma_N^2 := e^{{\Lambda^{\mu}_2(\beta_N)}}-1, \quad \bar{\sigma}_N^2 := \pi^{-1} \sigma_N^2,
\end{equation}
the system is said to be \textit{quasi-critical} (QC) if
\begin{equation} \tag{QC}
\label{eq:defNearCrit}
	\sigma_N^2 = \frac{1}{R_N}\left(1-\frac{\vartheta_N}{\log N}\right) \quad \text{with} \quad 1\ll \vartheta_N \ll \log N. 
\end{equation}
\begin{rem}
Note that \eqref{eq:defNearCrit} is equivalent to letting $\beta_N^2 := \frac{1}{R_N} \left(1-\frac{\vartheta_N'}{\log N}\right)$ with $1\ll \vartheta_N' \ll \log N$. However, working with $\sigma_N^2$ is more suitable for calculations.
\end{rem}

\subsection{Moment bounds}
\begin{thm} \label{th:SCnost}
Assume  \eqref{eq:sub-criticality}.  Let $\alpha \in (0,1)$. There exists $C_\alpha > 0$ such that for all $q\in \mathbb{N}_{\geq 3}$ satisfying
\begin{equation}\label{eq:q-conditionSubClassic}
	 \binom{q}{2}   \leq  \alpha \frac{(1-\hat \beta^2)}{\hat \beta^2}\log N,
\end{equation}
we have
\[\mathbb{E}[W_N(\beta_N)^{q}] \leq C_\alpha e^{\lambda^2(\hat \beta) \binom{q}{2}(1+o_N(1))},
\]
with $\lambda^2(\hat \beta):=\log \Big(\frac{1}{1-\hat \beta^2}\Big)$, where $o_N(1)\to 0$ as $N\to\infty$, uniformly over all $q\in \mathbb{N}_{\geq 3}$ satisfying \eqref{eq:q-conditionSubClassic}.
\end{thm} 
\begin{thm} \label{th:QCnost}
Assume \eqref{eq:defNearCrit}. Let $\alpha \in (0,1)$. There exists $C_\alpha>0$ such that for all $q\in \mathbb{N}_{\geq 3}$ satisfying
\begin{equation} \label{eq:QCtheta}
\binom{q}{2}  \leq \alpha\, \frac{(1-\bar \sigma_N^2 \log N)} {\bar \sigma_N^{2}},
\end{equation}
we have \[\mathbb{E}[W_N(\beta_N)^{q}] \leq C_\alpha e^{\lambda_N^2 \binom{q}{2}(1+o_N(1))},
\]
with $\lambda_N^2 := \log \Big(\frac{\log N}{\vartheta_N}\Big)$, where $o_N(1)\to 0$ as $N\to\infty$, uniformly over $q\in \mathbb{N}_{\geq 3}$ satisfying \eqref{eq:QCtheta}.
\end{thm}

\begin{rem}
\begin{enumerate}[label=(\roman*)] 
\item 
The condition \eqref{eq:QCtheta} is equivalent to $\binom q 2 \leq \alpha \vartheta_N$. In particular, we see that the maximal $q$ that we can take transitions from order $\sqrt{\log N}$ at sub-criticality  to order one at criticality.
\item Condition \eqref{eq:q-conditionSubClassic} is recovered from \eqref{eq:QCtheta} 
by setting $\vartheta_N \sim (1-\hat \beta^2)\log N$, i.e.,  $\bar \sigma_N^2 \sim \hat \beta^2/ \log N$. However, this case is not covered by the proof of Theorem \ref{th:QCnost} which requires $\vartheta_N\ll \log N$.
\end{enumerate}
\end{rem}

\begin{cor}[finite $q$] \label{cor:QCfiniteq} 
    Assume \eqref{eq:defNearCrit} and that $q\in \mathbb{N}_{\geq 3}$ is fixed. Let  $\lambda_N^2 = \log \frac{\log N}{\vartheta_N}$.  For all $\varepsilon>0$, there exists $N_0>0$ such that for all $n\geq N_0$,
\begin{equation} \label{eq:QCfiniteq}
\mathbb{E}[W_N^q] \leq (1+\varepsilon) \frac{\binom q 2}{\binom q 2 - 1} e^{\lambda_N^2 \binom q 2 (1+\varepsilon^{-1}\mathcal O(\vartheta_N^{-1}))}.
\end{equation}
Assume further that there is $C>0$ such that $\vartheta_N \geq C \log \log N$. Then, for all $q\in \mathbb N_{\geq 3}$,
    \begin{equation} \label{eq:qfixedQC}
        \limsup_{N\to\infty} \mathbb{E}[W_N(\beta_N)^{q}] e^{-\lambda_N^2 \binom{q}{2}} \leq \frac{\binom q 2}{\binom q 2 - 1}.
    \end{equation}
\end{cor}

\begin{remark}
In the sub-critical regime \eqref{eq:sub-criticality}, the limsup in \eqref{eq:qfixedQC} with $\lambda_N$ replaced by $\lambda(\hat \beta)$, is in fact a limit equal to 1, see \cite{LD23,Cosco2021MomentsUB}. (This can be seen from the convergence in law \eqref{eq:CLT} and the fact that $\sup_N \mathbb{E}[W_N^q] <\infty$ for all finite $q$.)
\end{remark}

\subsection{Estimates for varying starting and ending times}
For applications, it is of particular interest to estimate $\mathbb{E}[W_{s,t}^q(\beta_N)]$, where
\[
W_{s,t}(x,\beta_N) := {\rm E}_x\left[ e^{\sum_{n=s}^t (\beta_N \omega(n,x) - \Lambda^\mu(\beta_N))}\right], \quad s\leq t\leq N.
\]
For $t\leq N$, we consider the following relation on $q\geq 3$ and $t\leq N$: 
\begin{equation} 
\tag{$\mathcal H_{q,t,\alpha}$}
\label{eq:q,T-condition}
    \binom{q}{2}  b_{t}  \bar{\sigma}_N^2 \leq \alpha <1,
\end{equation}
where $b_t := ({1-\bar{\sigma}_N^2 \log t })^{-1}$ approximates the second moment of  $W_t(\beta_N)$ for large $t>0$, in both \eqref{eq:sub-criticality} and \eqref{eq:defNearCrit}.
For $s\leq t\leq N$, set
\begin{equation} \label{eq:def_lambda_T_Nbis}
	\lambda_{s,t,N}^2 := \log \left(\frac{1-\bar{\sigma}_N^2\log s}{1-\bar{\sigma}_N^2\log t }\right),\quad \lambda_{t,N}^2:=\lambda_{1,t,N}^2
\end{equation}
 
\begin{thm}\label{th:mainTheorem} Assume  either \eqref{eq:sub-criticality} or \eqref{eq:defNearCrit}, and \eqref{eq:q,T-condition}. Let $\alpha \in (0,1)$ and $\gamma >2$. There exist $N_{\alpha}$ and $\varepsilon_0(\alpha)$ 
such that for all $N\geq N_\alpha$, for all $1 \leq s\leq t\leq N$,
	\begin{equation} \label{eq:mainEstimateT}
		\sup_{(x_1,\dots,x_q) \in (\mathbb Z^2)^q} \mathbb{E}\left[\prod_{i=1}^q W_{s,t}(x_i,\beta_N)\right]  \leq  c_{\star,\alpha} b_t  e^{\lambda_{s,t,N}^2 (\binom{q}{2}-1)(1+\varepsilon_N)},
	\end{equation}
where $$c_{\star,\alpha} := (1+\mathcal O(q \bar \sigma_N) +q^c\, \mathbf{1}_{s\leq q^\gamma,\,\binom q 2 > \varepsilon_0 \log N}) c_{\star},$$ 
with $c_\star=c_{\star}(q,\alpha)$ from Proposition \ref{prop:mainProp},
 $c=c(\mu)>0$ and 
$\varepsilon_N = \varepsilon_N^{\hyperlink{deltaCondition}{\triangle}} + \varepsilon_N^{\hyperlink{symbol: star}{\diamondsuit}} + \varepsilon_N^{\hyperlink{symbol: clubsuit}{\clubsuit}}$ with 
\begin{enumerate}[label=(\roman*)]
\item \hypertarget{symbol: star}{} $\varepsilon_N^{\hyperlink{symbol: star}{\diamondsuit}} = \mathcal{O}((\log N)^{-1})$ in \eqref{eq:sub-criticality}; $\varepsilon_N^{\hyperlink{symbol: star}{\diamondsuit}} = \mathcal O(\vartheta_N^{-1})$ in \eqref{eq:defNearCrit};  \item {$\varepsilon^{\hyperlink{deltaCondition}{\triangle}}_N$ satisfies the property $(\hyperlink{deltaCondition}{\triangle})$ of Definition \ref{def:DeltaCondition};
\item \hypertarget{symbol: clubsuit}{}$\varepsilon_N^{ \hyperlink{symbol: clubsuit}{\clubsuit}} =0$ if $\mu = \mathcal N(0,1)$ or if {$\binom q 2 \leq \varepsilon_0 \log N$;} $\varepsilon_N^{\hyperlink{symbol: clubsuit}{\clubsuit}} = \mathcal O((\log N)^{-1/2})$ otherwise.}
\end{enumerate}
\end{thm}

\begin{remark}
Theorem \ref{th:mainTheorem} can be used to estimate joint moments for far away starting points $|x_i-x_j| \sim r_N\to\infty$. See e.g.\@ Lemma 3.6 in \cite{CNZ25}. 
\end{remark}

 \subsection{Large and moderate deviations estimates.}
We derive large and moderate deviation estimates on $\log W_N$ with Gaussian rate function.
\begin{cor} \label{cor:LDSC} Assume \eqref{eq:sub-criticality}. Let $x_N\to\infty$ with $\limsup_N \frac{x_N^2}{2\log N} < \lambda^2 \frac{1-\hat \beta^2}{\hat \beta^2}$, then
\begin{equation}\lim_{N\to\infty} \frac{1}{x_N^2} \log \mathbb{P}\left(\log W_N(\beta_N)  \geq \lambda x_N \right) = -\frac{1}{2},
\end{equation}
with $\lambda^2= \log \Big(\frac{1}{1-\hat \beta^2}\Big)$.
\end{cor}

\begin{cor} \label{cor:LDQC} Assume \eqref{eq:defNearCrit}. Let $x_N\to\infty$ with $\limsup_N \frac{x_N^2}{2\lambda_N^2 \vartheta_N} < 1$, then
\begin{equation} \label{eq:CorQC}
\lim_{N\to\infty} \frac{1}{x_N^2} \log \mathbb{P}\left(\log W_N(\beta_N)  \geq \lambda_N x_N \right) \leq -\frac{1}{2},
\end{equation}
with $\lambda_N^2 := \log \frac{\log N}{\vartheta_N}$.
\end{cor}
\begin{remark} 
What is missing to obtain an equality in \eqref{eq:CorQC} is a lower bound on the $q$-moment of the partition function, which was  derived so far only for sub-criticality  \cite{cosco2023momentsLB}. See also Section \ref{subsec:compar}.
\end{remark}

\subsection{Moment formula} 
If the weights are Gaussian $\mu=\mu_0= \mathcal N(0,1)$, writing $W^{\mu_0}_{s,t}(x,\beta)=W_{s,t}(x,\beta)$ for such weights, then
\begin{equation} \label{eq:identity0}
\mathbb{E}\left[\prod_{i=1}^q W^{\mu_0}_{s,t}(x_i,\beta_N)\right] =  {\rm E}_{x_1,\dots,x_q}^{\otimes q}\left[e^{\beta_N^2 \sum_{n=s}^t \sum_{(i,j)\in \mathcal C_q} \mathbf{1}_{\{S_n^i=S_n^j\}}}\right], 
\end{equation}
where $\mathcal C_q: = \{(i,j), 1\leq i<j\leq q\}$.
For general weights, one finds that
\begin{equation} \label{eq:identity}
\mathbb{E}\left[\prod_{i=1}^q W_{s,t} (x_i,\beta_N)\right]
	={\rm E}^{\otimes q}_{x_1,\dots,x_q}\left[e^{\sum_{n=s}^t \sum_{x\in\Z^2} \Lambda^\mu_{N_n^q(x)}(\beta_N)} \right],
\end{equation}
where $
	N_n^q(x) := \sum_{i=1}^q \mathbf{1}_{S_n^i = x}$ counts the number of particles at a point $x\in \mathbb Z^2$.
The proof of Theorem \ref{th:mainTheorem} proceeds by giving an upper bound on the random walk representations of the $q$-moment \eqref{eq:identity0}-\eqref{eq:identity}.

\subsection{Comparison with related works} \label{subsec:compar}
In the sub-critical regime, the upper bound in Theorem \ref{th:SCnost} matches, up to leading order, the lower bound on the same moment obtained in \cite{cosco2023momentsLB}.  We believe that the proof of \cite{cosco2023momentsLB} may also provide a matching lower bound for quasi-criticality, but we leave this question for further study. 
Still in the sub-critical regime, the fact that $W_N(\beta_N)$ is bounded in $L^p$ for all finite $p>0$ was derived previously in \cite{LD23,Cosco2021MomentsUB}.  
In \cite{Cosco2021MomentsUB}, the upper bound \(\mathbb{E}[W_N^q] \leq e^{\binom{q}{2} \lambda^2 (1+o_N(1))}\) was established, but with a limitation on the maximal range of \(q\) reduced by a factor of \(1/3\) compared to \eqref{eq:q-conditionSubClassic}, and a restricted range of small \(\hat{\beta}\). Overcoming these constraints, particularly the \(1/3\) discrepancy, is crucial for applications. 
Additionally, the results in \cite{Cosco2021MomentsUB} lacked uniform bounds on \(t\) and \(q\), which we believe are particularly useful for applications in the quasi-critical regime, as well as for investigating refined properties in the sub-critical regime. Finally, the analysis in \cite{Cosco2021MomentsUB} was limited to Gaussian weights.

The issue preventing \cite{Cosco2021MomentsUB} from reaching optimal thresholds in $q$ and $\hat \beta$ arises mainly from some crude use of the Cauchy-Schwarz inequality, in order to disregard, \textit{a priori}, the appearance of more than pairwise intersections in the moment formula \eqref{eq:defPhiX}. Instead, we proceed by analyzing more diagrams where a large (but finite) number of  particles may meet simultaneously, and prove \textit{a posteriori} that only those with mainly pairwise intersections matter. We refer to Section \ref{sec:descriptionOfProof} for more details on the proof.

In the quasi-critical regime, our results are new and the proof differs in some important aspects. In that case, we prove that intersections involving more than two particles can be neglected \textit{a priori}, by using an induction scheme whose relevance relies upon the fact that the maximal $q$ in  quasi-criticality satisfies some "weak energy" property $\binom q 2 \sigma_N^2 \to 0$  (which does not persist in the sub-critical regime). 

Finally, we refer the reader to Section~\ref{sec:discussion} for related questions and open problems.

\subsection{Notation and Gaussian weights.} For all $J\subset \kC_q$, we write $\bar{J}:=\cup_{(i,j)\in J}\{i,j\}$. For all $K\subset \Iintv{1,q}$, we set  
$\mathcal C_{K} := \left\{(i,j),  i,j\in K, i<j\right\}$ and $\mathcal C_q := \mathcal C_{\Iintv{1,q}}$.
Define: 
\begin{equation}  \label{eq:defPsiGauss}
    \psi_{s,t,q}^\star :=  \sum_{n=s}^t \sum_{(i,j)\in \mathcal C_q} \mathbf{1}_{\{S_n^i=S_n^j\}},\quad \psi_{t,K}^\star :=   \sum_{n=1}^t \sum_{(i,j)\in \mathcal C_{K}} \mathbf{1}_{\{S_n^i=S_n^j\}},
\end{equation}
$\psi_{k,q}^\star := \psi^\star_{1,k,q}$ and $\psi^\star_{0,q} :=0$. We also set, with $
	N_n^q(x) := \sum_{i=1}^q \mathbf{1}_{S_n^i = x}$,
\begin{equation} \label{eq:defPsikq}
	\psi_{s,t,q}(\beta) := \psi_{s,t,q}^{\mu}(\beta) :=
	\sum_{n=s}^t \sum_{x\in\Z^2} \Lambda^{\mu}_{N_n^q(x)}(\beta).
\end{equation}
Then, define for all $X\in (\mathbb Z^2)^q$,
\begin{equation}
    \label{eq:defPhiX}
    \Phi_{s,t,q}(X):= \Phi_{s,t,q,N}^{\mu}(X,\beta_N) := {\rm E}_{X}\left[e^{{\Lambda_2^{\mu}(\beta_N)}\psi^{\star}_{s,t,q}}\right],
\end{equation}
and $\Phi_{t,q}(X) := \Phi_{1,t,q}(X)$. 
Similarly, set 
\begin{equation}
\Psi_{s,t,q}(X) := \Psi^{\mu}_{s,t,q}(X) = {\rm E}_X\left[e^{\psi^{\mu}_{s,t,q}(\beta_N)}\right].
\end{equation}
Note that \eqref{eq:identity} writes $\Psi_{s,t,q}(X)=\mathbb{E}\left[\prod_{i=1}^q W_{s,t}(x_i,\beta_N)\right]$.

\begin{remark}[Gaussian weights] \label{rk:GaussianLambda}
	For $\mu_0 = \mathcal N(0,1)$, it holds that $\Lambda^{\mu_0}_2(\beta)=\beta^2$ so that $\sigma_N^2 = e^{\beta_N^2}-1$. Moreover, 
 $\Lambda_p^{\mu_0}(\beta) = \beta^2 \binom p 2$ for  all $p\geq 2$ which implies that
	$\psi^{\mu_0}_{k,q}(\beta)= \beta^2 \psi^{\star}_{k,q}$.  
In particular, \(\mathbb{E}\left[\prod_{i=1}^q W^{\mu_0}_{s,t}(x_i,\beta_N)\right] = \Phi^{\mu_0}_{s,t,q}(X)\) where  $\Lambda^{\mu_0}_2(\beta_N)=\beta_N^2$. 
\end{remark}

We will see that most of the study can be reduced to the case of Gaussian weights. For this reason, we discuss first how to deal with $\Phi_{s,t,q}$ and leave the treatment of $\Psi_{s,t,q}$ as a second step (cf.\@ Sections \ref{sec:quasiCritRemoveTriple+} and \ref{sec:proofMainThSubcrit}).

\subsection{Description of the proof}
\label{sec:descriptionOfProof}
The starting point is Equation \eqref{eq:expansion_with_triple+} which expresses $\Phi_{s,t,q}$ into a sum over diagrams of successive particle interactions. The proof then proceeds in two main steps. First, we estimate the contribution in this sum from pairwise interactions  (i.e.,\@ $|\bar J_r|\leq 2$, or equivalently $|J_r|= 1$), see point (ii) of  Proposition \ref{prop:mainProp}.  
This is done by adapting an induction argument from \cite{Cosco2021MomentsUB}. 
On the way, the condition
\eqref{eq:q,T-condition} appears to ensure the convergence of some geometric series.

Conversely, we argue that \eqref{eq:q,T-condition} is enough to neglect   the contribution in \eqref{eq:expansion_with_triple+} from 
sets $(J_r)$ that involve more than two particles.
These steps involve two different arguments depending on whether $q^2$ is of order $\log N$ (\emph{higher $q$ case}) or $q^2$ is of smaller order (\emph{lower $q$-case}). 

\begin{remark} \label{rk:quasiLowerq}
In quasi-criticality, we will be  only interested in the \emph{lower $q$-case}. Indeed, under the condition \eqref{eq:q,T-condition},  if $q^2$ is of  order $\log N$,  then  $t$ is less than $N^{c}$ with some $c\in(0,1)$. For such short times, the model behaves similarly to sub-criticality -- in particular, the second moment of the partition function is bounded from above (see \eqref{eq:borneTheta}).
\end{remark}

In the \emph{lower $q$-case}, the reduction to pairwise interactions 
is  done \textit{a priori}, see Section \ref{sec:quasiCritRemoveTriple+}. In the spirit of \cite{CSZAAP}, we show that  
 having one index $r$ such that  $|\bar J_r|\geq 3$
contributes by a factor strictly less than one. By iteration, this ensures that all occurrences contribute less than a constant factor.
However, some complications arise when $q$ approaches the upper threshold in \eqref{eq:q,T-condition}. Indeed, for such $q$'s, 
we need to neglect "bad" diagrams (namely, those with atypically few exchanges of pairs), which have a larger weight than typical ones but are less numerous.  
This leads us to introduce the concept of "small jumps," as described in \cite{Cosco2021MomentsUB}, referring to cases where an intersecting particle was recently involved in another intersection. (Typically, diagrams with numerous pair exchanges contain mostly long jumps.) To account for the error emerging from "bad" diagrams, we solve an induction scheme that  quantifies the excess weight contributed by "small jumps" in any given diagram. 

In the \emph{higher $q$-case}, we divide the reduction to pairwise interactions in two steps. 
First, we argue that one can reduce \textit{a priori} the situation to the case where all $J_r$ involve at most $p_0$ particles (i.e.,\@ $|\bar J_r| \leq p_0$), where $p_0$ is a large yet finite constant, see  Section \ref{sec:proofMainThSubcrit}. 
Next, we control truncated moments with $|\bar J_r| \leq p_0$ (Proposition \ref{prop:mainProp}-\ref{item:(ii)}) by relying again on the induction scheme from \cite{Cosco2021MomentsUB} 
--- the difference being that 
now one needs to deal with diagrams more than pairwise interactions.

Finally, we emphasize that the strategy applied to lower $q$ cannot be used for higher $q$ and vice versa. Indeed, to ensure that contributions from sets with $|\bar J_r| \geq 3$ are negligible for lower $q$, 
we need the energy of one such interaction 
$E_{\geq 3}=\sum_{J\subset \mathcal C_q, |\bar J|\geq 3} \sigma_N^{2|J|}$ to be arbitrarily small, where 
$\sigma_N^2 = e^{\Lambda_2(\beta_N)}-1$ is the energy gain of one couple meeting. This is only verified for lower $q$ 
(for higher $q$, 
the energy $E_{\geq 3}$ is bounded  uniformly away from $0$). Conversely, the argument for higher $q$ requires $\Phi_{t,p_0}$ to be bounded for all finite $p_0$, which does not hold in quasi-criticality as already for $p_0=2$, $t=N$, we have $\Phi_{N,2}\sim (\log N)/{\vartheta_N}\to\infty$.

\subsection{Structure of the paper.} Sections \ref{sec:withoutp_0} and \ref{sec:inductionSchemeSubcrit} provide preliminary steps for the proof of Proposition \ref{prop:mainProp} given in Section \ref{sec:mainProp}. Sections \ref{sec:quasiCritRemoveTriple+} and \ref{sec:mainTh,lowq} are dedicated to the lower $q$-case:  in Section \ref{sec:quasiCritRemoveTriple+}, we explain how to reduce the problem to diagrams with only pairwise interactions and in Section \ref{sec:mainTh,lowq}, we conclude the proof of Theorem \ref{th:mainTheorem} in this regime. The proof of Theorem \ref{th:mainTheorem} in the higher $q$ case is given in Section \ref{sec:proofMainThSubcrit}. Section~\ref{sec:discussion} gathers concluding remarks and open questions.


Throughout the paper, if not otherwise stated, $C>0$ denotes a universal constant that may change from line to line.

\section{Expansion into successive disjoint pairings.}
In Proposition \ref{prop:expansion} below, we give an expression of $\Phi_{s,t,q}$ in \eqref{eq:defPhiX} in terms of a sum over diagrams of successive, non-included sets of particle intersections. 
Define
\[
\mathcal D(m,q,\infty): = \{(J_i)_{i=1}^m:~\forall r \in \Iintv{1,m},\, \emptyset \neq J_r \subset \mathcal C_q\,, \bar J_1\nsupseteq \bar{J}_2,\dots,\bar J_{m-1} \nsupseteq  \bar J_m\}.
\]
We call $\mathbf J = (J_1,\dots,J_m) \in \mathcal D(m,q,\infty)$ a \textit{diagram}.

\begin{prop} \label{prop:expansion}
Let $X_0=(x_0^1,\dots,x_0^q)\in (\mathbb Z^2)^q$.
	{For all $1\leq s< t$}, 
we have:
	\begin{equation} \label{eq:expansion_with_triple+}
		\begin{aligned}
			 & \Phi_{s,t,q}(X_0) = 1+\sum_{m=1}^\infty
			\sum_{\mathbf J\in \mathcal D(m,q,\infty)} \sum_{I_1\subset  \mathcal C_{\bar J_1}, \dots ,  I_m \subset  \mathcal C_{\bar J_m}} \sum_{s \leq a_1\leq  b_1 < a_2 \leq b_2\dots \leq b_m\leq t} \sum_{\mathbf{X},\mathbf{Y}}  \prod_{r=1}^{m} \\
			 &  \Big\{ \sigma_N^{ 2|J_r|}  U_N(b_r-a_r,X_r,Y_r,  J_r, I_r) \prod_{(i,j)\in J_r} \mathbf{1}_{x_r^{i}=x_r^j}
			\prod_{(i,j)\in I_r} \mathbf{1}_{y_r^i=y_r^j} \prod_{i\in \bar{J}_r} p_{a_r-b_{k_r(i)}}\left(x_r^i-y_{k_r(i)}^i\right)\Big\},
		\end{aligned}
	\end{equation}
	where $b_0:= 0$ and
	\begin{itemize}
		\item the sum on $\mathbf{X},\mathbf{Y}$ is over all possible $\mathbf{X}=(X_1,\dots,X_m)$ with $X_r=X_r(\bar J_r)=(x_r^i)_{i\in \bar J_r}$ and $x_r^i\in \mathbb Z^2$, and all $\mathbf{Y} = (Y_1,\dots,Y_m)$ with $Y_r=Y_r(\bar J_r)=(y_r^i)_{i\in \bar J_r}$ and $y_r^i\in \mathbb Z^2$, 
		\item $k_r(i):=\sup\{ k\in\llbracket 1,r-1\rrbracket:i\in \bar{J}_k\}$  is the last index before $r$ where the particle $i$ has appeared in an intersection (we set $k_r(i):=0$ and $y_{k_r(i)}:=x_0^i$ if the set is empty),
		\item $U_N(b_r-a_r,X_r,Y_r, J_r, I_r)$ denotes the pairwise interactions within the  subset  of particles $\bar J_r$ (and only in this set) between time $a_r$ and $b_r$, starting at position $X_r$ at time $a_r$ and ending at $Y_r$ {at time $b_r$, with $I_r$ describing which particles intersect at time $b_r$,} given by (recall $\psi^\star_{t,K}$ from \eqref{eq:defPsiGauss})
\begin{equation} \label{eq:def_UNXYJI}
  U_N(t,X,Y,J,I) := \begin{cases}
      \sigma_N^{ 2|I|}{\rm E}_{X}^{\otimes \bar J} \left[e^{  {\Lambda_2(\beta_N)} \psi_{t-1,\bar{J}}^\star  } \prod_{i\in \bar J} \mathbf{1}_{S_t^i=y^i}\right] \mathbf{1}_{\bar I \neq \emptyset} & \text{if } t>0, \\
      \mathbf{1}_{J=I,\,X=Y}             & \text{if } t=0,
  \end{cases}
\end{equation}
for all $\emptyset \neq J\subset \mathcal C_q$, $I\subset \mathcal {C}_{\bar J}$, and $X,Y\in (\mathbb Z^2)^{|\bar J|}$ with $Y=(y^i)_{i\in \bar J}$.
\end{itemize}
\end{prop}
\begin{proof}
Writing $e^{\Lambda_2(\beta_N)\psi^\star_{s,t,q}} = \prod_{(i,j)\in  \mathcal C_q} \prod_{n\in \llbracket s,t\rrbracket}  (1+(e^{\Lambda_2(\beta_N)}-1)\mathbf{1}_{S_n^i=S_n^j})$, we obtain,  using that $\sigma_N^2= e^{\Lambda_2(\beta_N)}-1$,
	\begin{equation} \label{eq:ChaosOnqMoment}
		\begin{aligned}
			 & \Phi_{s,t,q}(X_0) =                                                                              \\
			 & {\rm E}_{X_0}^{\otimes q}\left[1+\sum_{k=1}^N \sum_{\substack{K_1,\dots,  K_k \subset \mathcal C_q \\K_r \neq \emptyset, r\leq k}} \sum_{s\leq n_1<\dots<n_k\leq t}  \prod_{r=1}^{k} \sigma_N^{ 2|K_r|} \prod_{(i,j)\in K_r} \mathbf{1}_{S_{n_r}^i=S_{n_r}^j}\right].
		\end{aligned}
	\end{equation}
Let $k\geq 1$, $s\leq n_1<n_2<\ldots<n_k\leq t$, and $K_1,\dots,K_k\subset \mathcal C_q$ with $K_r\neq \emptyset$ as in the above sum. Set $a_1:=n_1\in \llbracket s,t\rrbracket$. Let $k_1\in \Iintv{1,k}$ be the maximal number 
such that $\bar{K}_1,\dots \bar{K}_{k_1} \subset \bar{K}_1$
and let $b_1:=n_{k_1}$ be the associated time index. 
Write $J_1 := K_{1} $ and $I_1:=K_{k_1}\subset \mathcal C_{\bar J_1}$. If $k_1<k$, let $J_2 := K_{k_1+1}$ which satisfies $\bar J_1\nsupseteq \bar J_2$.  Then, continue by defining $k_2$ as the maximal number such that $\bar{K}_{k_1+i}\subset \bar{K}_{k_1+1}$ for all $i\in \llbracket 1,k_2\rrbracket$.
Finally, let $a_2:=n_{k_1+1}$ and $b_2:=n_{k_1+k_2}$.
Repeat this procedure (until $k_1+\dots +k_m=k$) to define $k_1,\dots,k_m$, \(s\leq a_1\leq  b_1 < a_2 \leq b_2 < \dots a_m \leq b_m\) (note that $a_r=b_r$ when $k_r=1$), \(\bar J_1\nsupseteq \bar J_2 \nsupseteq \dots \nsupseteq \bar J_{m-1} \nsupseteq \bar J_m\) and $I_1,\dots,I_m$ with $I_r\subset \mathcal C_{\bar J_r}$.

Rewriting the sum inside \eqref{eq:ChaosOnqMoment}, we obtain 
that 
	\begin{align}
		\Phi_{s,t,q}(X_0) = 1+{\rm E}_{X_0}^{\otimes q}\Bigg[   \sum_{m=1}^\infty & \sum_{\substack{J_1\nsupseteq J_2 \dots \nsupseteq  J_m \subset \mathcal C_q}}\sum_{I_r\subset \mathcal C_{\bar{J}_r},r\leq m} \,\sum_{s\leq a_1\leq  b_1 < a_2 \leq b_2\dots \leq b_m\leq t}
		\notag \\
      & \prod_{r=1}^m U^{(r)}_N(a_r,b_r,I_r,J_r) 
      \Bigg],\label{eq: discrete magical formula}
\end{align}
where, when $a_r < b_r$,  $U^{(r)}_N(a_r,b_r,I_r,J_r)$ is obtained by summing over all possible values of $k_r$ with $l:=k_r{-2}$, all configurations $K_{s}$, and $n_s$ between times $a_r$ and $b_r$,  that is,
	\al{
	&U^{(r)}_N(a_r,b_r,I_r,J_r)=\sigma_N^{ 2(|J_r|+|I_r|)}  \prod_{(i,j)\in J_r} \mathbf{1}_{S_{a_r}^i=S_{a_r}^j} \prod_{(i,j)\in I_r} \mathbf{1}_{S_{b_r}^i=S_{b_r}^j}\\
	&\times \left(1+ \sum_{l=1}^\infty \sum_{\substack{L_1,\dots,L_l \subset \mathcal C_q\\\emptyset \neq \bar{L}_s \subset{\bar{J}_r}}} \sigma_N^{ 2\sum_{s=1}^l |L_s|}\sum_{a_r < n_1 < \dots < n_l<  b_r} \prod_{s=1}^l \prod_{(i,j)\in L_s} \mathbf{1}_{S_{n_s}^i=S_{n_s}^j}\right)\\
	&=  \sigma_N^{ 2(|J_r|+|I_r|)}  \prod_{(i,j)\in J_r} \mathbf{1}_{S_{a_r}^i=S_{a_r}^j} \prod_{(i,j)\in I_r} \mathbf{1}_{S_{b_r}^i=S_{b_r}^j} e^{\Lambda_2(\beta_N) \sum_{k=a_r+1}^{b_r-1}\sum_{(i,j)\in   {\mathcal C_{\bar J_r}}
 } \mathbf{1}_{S_k^i=S_k^j}},
	}
and when $a_r = b_r$, $
		U^{(r)}_N(a_r,a_r,I_r,J_r)= \sigma_N^{ 2|J_r|}  \prod_{(i,j)\in J_r} \mathbf{1}_{S_{a_r}^i=S_{a_r}^j}$.
The statement of the proposition follows by Markov’s property.
\end{proof}
\begin{remark}
{The moment representation \eqref{eq:expansion_with_triple+} generalizes a decomposition of \cite[Section 5]{CaSuZyCrit18}, given there for $q=3$ and under the restriction in \eqref{eq:ChaosOnqMoment} that all $K_r$ contain only a single pairwise intersection.    In comparison, our formula does include the situation where more than pairwise intersections appear.}
\end{remark}
\begin{remark}
The identity \eqref{eq:expansion_with_triple+} can be seen as a discrete analogue of Eq.\@ (2.13) in \cite{Chen24}
which provides a diagram expansion for  the $q$-moment of polymers in continuous time and space.  In this continuous setting, with the particular choice of white noise environment, some cancellations appear and only pairwise interactions (i.e.,\@ $J_r=\{(i_r,j_r)\}\in \mathcal C_q$) remain. This contrasts with the discrete model, where one has to deal with the extra combinatorial constraint that the $J_r$ may involve more than two particles. 
\end{remark}

\begin{remark}
Because the expression of $\psi_{s,t,q}$ depends on a more precise description of the intersection clusters of the walks than $\psi^\star_{s,t,q}$, it is not possible to obtain an expression as simple as   \eqref{eq:expansion_with_triple+} when considering $\Psi_{s,t,q}$ instead of $\Phi_{s,t,q}$. 
For this reason, most of the paper deals only with $\Phi_{s,t,q}$. We defer the study of general weights to Section \ref{sec:quasiCritRemoveTriple+} and  \ref{sec:proofMainThSubcrit}. 
It is done by comparing $\Psi_{s,t,q}$ with $\Phi_{s,t,q}$.
\end{remark}

\section{Truncated moment estimates} \label{sec:withoutp_0}
\subsection{Truncation}
Let $p_0\in \mathbb N$.
We define $\Phi_{s,t,q}^{\leq p_0}(X_0)$ as the sum on the right-hand side of \eqref{eq:expansion_with_triple+}, but with $\mathcal D(m,q,\infty)$ replaced by
\begin{equation*} \label{eq:defDmqp} 
\begin{aligned}
&\mathcal D(m,q,p_0):=  \left\{(J_i)_{i=1}^m:~ \forall r \in \Iintv{1,m}, \,\emptyset \neq J_r \subset \mathcal C_q,  |\bar J_r|\leq p_0\,,\bar J_1\nsupseteq \bar{J}_2, \dots, \bar{J}_{m-1} \nsupseteq  \bar J_m\right\},
\end{aligned}
\end{equation*}
that is,
\begin{equation} \label{eq:Phip_0}
\begin{aligned}
      \Phi_{s,t,q}^{\leq p_0}(X_0) &:=1+\sum_{m=1}^\infty
    \sum_{\mathbf J\in \mathcal D(m,q,p_0)} \sum_{I_1\subset  \mathcal C_{\bar J_1}, \dots ,  I_m \subset  \mathcal C_{\bar J_m}} \sum_{s \leq a_1\leq  b_1 < a_2 \leq b_2\dots \leq b_m\leq t}\\ &  \sum_{\mathbf{X},\mathbf{Y}}  \prod_{r=1}^{m} \Big\{\sigma_N^{ 2|J_r|} U_N(b_r-a_r,X_r,Y_r,  J_r, I_r) \\
     &    \prod_{(i,j)\in J_r} \mathbf{1}_{x_r^{i}=x_r^j}
    \prod_{(i,j)\in I_r} \mathbf{1}_{y_r^i=y_r^j} \prod_{i\in \bar{J}_r} p_{a_r-b_{k_r(i)}}\left(x_r^i-y_{k_r(i)}^i\right)\Big\}.
\end{aligned}
\end{equation}

Considering $p_0=2$ allows for simplification. For all $\mathbf J \in \mathcal D(m,q,2)$, each $J_r$ contains only one couple that we call $(i_r,j_r)$, so we can write $$\mathcal D(m,q) := \mathcal D(m,q,2)= \{\mathbf J=((i_1,j_1),\dots,(i_m,j_m)), \forall r< m:  (i_r,j_r)\neq (i_{r+1},j_{r+1})\in \mathcal C_q\}.$$ 

Thus, we define $\Phi_{s,t,q}^{\leq 2}$ as
\begin{equation}\label{eq:Phi2}
    \begin{aligned}
         & \Phi_{s,t,q}^{\leq 2}(X_0) := 1+\sum_{m=1}^\infty
        \sum_{\mathbf{J}\in \mathcal D(m,q)} \sum_{s\leq a_1\leq  b_1 < a_2 \leq b_2\dots \leq b_m\leq t} \sum_{x_1,y_1,\dots,x_m,y_m\in \mathbb Z^2}    \\
         & \prod_{r=1}^{m} \Big\{ \sigma_N^{2}  U_N(b_r-a_r,y_r-x_r)  \prod_{i\in \{i_r,j_r\}} p_{a_r-b_{k_r(i)}}\left(x_r-y_{k_r(i)}\right)\Big\},
    \end{aligned}
\end{equation}
where
\begin{equation} \label{eq:defU_Nz}
    \begin{aligned}
        U_N(n,z) & :=
        \begin{cases}
            \sigma_N^2 {\rm E}_{0}^{\otimes 2}\left[ e^{ \beta_N^2\sum_{l=1}^{n-1} \mathbf 1_{S^1_l = S^2_l}} \mathbf{1}_{S^1_n= S^2_n=z} \right] & \text{if } n\geq 1, \\
            \mathbf{1}_{z=0}   & \text{if } n=0.
        \end{cases}
    \end{aligned}
\end{equation}
\begin{defin} \label{def:DeltaCondition}
We say that $\varepsilon_N^{\hyperlink{deltaCondition}{\triangle}}=\varepsilon_N^{\hyperlink{deltaCondition}{\triangle}}(q,t)>0$ satisfies condition \hypertarget{deltaCondition}{($\Delta$)} if there exists $C>0$ such that for all $(q,t)$ satisfying \eqref{eq:q,T-condition},
\begin{enumerate}
    \item 
for all $\binom q 2 \leq (b_t \bar \sigma_N^2)^{-4/5}$, $\varepsilon_N^{\hyperlink{deltaCondition}{\triangle}}(q,t) \leq C q^2 b_t \bar \sigma_N^2$,
\item \label{item:(22)}
for all $ \binom q 2 \geq (b_t \bar \sigma_N^2)^{-4/5}$, $\varepsilon_N^{\hyperlink{deltaCondition}{\triangle}}(q,t) \leq C q^{-1/4}$.
    \end{enumerate}
\end{defin}

\begin{rem}
    We emphasize that $b_t \bar \sigma_N^2 \leq b_N \bar \sigma_N^2\to 0$ in both \eqref{eq:sub-criticality} and \eqref{eq:defNearCrit}.
    
\end{rem}
\begin{prop} \label{prop:mainProp} Let $\varepsilon_0 >0$.
 Assume 
 that one of the two conditions holds:
\begin{enumerate}[label=(\roman*)]
    \item \label{item:(i)}
    \eqref{eq:sub-criticality},  $p_0\in \mathbb N\setminus \{1\}$ and {$q^2\geq \varepsilon_0 \log N$,}
    \item \label{item:(ii)} {\eqref{eq:sub-criticality}} or \eqref{eq:defNearCrit}, and $p_0=2$.
\end{enumerate}
Then, 
there exist $c(p_0)>0$ satisfying $c(2)=1$ and $C_\alpha,N_\alpha>0$ such that for all $N\geq N_\alpha$, for all $(q,t)$ satisfying \eqref{eq:q,T-condition},
\begin{equation} \label{eq:mainResultWithoutL2}
\forall s\leq t,\quad	\sup_{X_0 \in (\mathbb Z^2)^q} \Phi_{s,t,q}^{\leq p_0}(X_0) \leq c_\star  b_t e^{\lambda_{s,t,N}^2 (\binom{q}{2}-1)(1+\varepsilon_N)},
\end{equation}
where $c_\star = C_\alpha c(p_0) q^{1/4}$ and $\varepsilon_N$ satisfies condition ($\hyperlink{deltaCondition}{\triangle}$) in Definition \ref{def:DeltaCondition}.  Moreover, for all fixed $q\in \mathbb Z_{\geq 3}$, for any arbitrary $\varepsilon\in (0,1/4)$, one can take $c_{\star} = (1+\varepsilon)\frac{\binom q 2}{\binom q 2-1}$ at the price of multiplying $\varepsilon_N$ by $C\varepsilon^{-1}$.
\end{prop}
\begin{remark}
    The error $q^{-1/4}$ obtained as in  Definition \ref{def:DeltaCondition}--(\ref{item:(22)}) can be improved to $q^{-1/2}$. As this is computationally heavier, we keep $q^{-1/4}$ for simplicity.
\end{remark}

\subsection{The general $p_0$ case}
Recall the definition of $\Phi_{s,t,q}^{\leq p_0}$ in \eqref{eq:Phip_0}.
\begin{lem} \label{lem:ChangeOfVariables} For all $s\leq t$, we have
\begin{equation} \label{eq:expansion_first_bound}
    \begin{aligned}  & \Phi_{s,t,q}^{\leq p_0}(X_0) \leq
            \sum_{m=0}^\infty
            \sum_{\mathbf J\in \mathcal D(m,q,p_0)}  \sum_{I_1\subset  \mathcal C_{\bar J_1}, \dots ,  I_m \subset  \mathcal C_{\bar J_m}}  \sum_{\substack{u_r \in \llbracket 1,t\rrbracket,2\leq r \leq m \\u_1\in \llbracket s,t\rrbracket}}\sum_{v_r \in \llbracket 0,t\rrbracket,1\leq r \leq m} \sum_{X,Y}  \\
             & \prod_{r=1}^{m} \Big\{ \sigma_N^{ 2|J_r|}  U_N(v_r,X_r,Y_r,  J_r , I_r) \prod_{(i,j)\in J_r} \mathbf{1}_{x_r^{i}=x_r^j}
            \prod_{(i,j)\in I_r} \mathbf{1}_{y_r^i=y_r^j} \prod_{i\in \bar{J}_r} p_{w_r(i)}(x_r^i-y_{k_r(i)}^i)\Big\},
    \end{aligned}
\end{equation}
	where $w_r(i):=\sum_{k_r(i)+1}^r u_s + \sum_{k_r(i)+1}^{r-1} v_s$.
\end{lem}

\begin{proof}
 We make the change of variables $u_1=a_1\in \llbracket s,t\rrbracket$, $u_r=a_r-b_{r-1}\in\llbracket 1,t\rrbracket$ for $r\geq 2$ and $v_r=b_r-a_r\in \llbracket 0,t\rrbracket$ for $r\geq 1$. Then, we have $a_r-b_{k_r(i)}= w_r(i)$. This gives the lemma.
\end{proof}
\begin{defin}[$J$-clusters]\label{Defin: J-cluser} {For any $J\subset \mathcal C_q$ and \(i\in \bar J\), we call \emph{\(J\)-cluster of \(i\)} the set of indices \( j \in \bar{J} \) that are connected to \( i \) via  pairs in \( J\), i.e.,  if there exists $l\geq 2$ and $i=i_1,i_2,\dots, i_l=j$ such that $(i_r,i_{r+1})\in J$ or $(i_{r+1},i_r)\in J$ for all $r\leq l-1$. We call any such set a \emph{$J$-cluster}.}
\end{defin}
{Recall that $k_r(i):=\sup\{ k\in\llbracket 1,r-1\rrbracket:i\in \bar{J}_k\}$ and $k_r(i):=0$ if the set is empty.}
\begin{defin} \label{def:defi_1i_2}
Let $(J_1,\dots, J_m)\in \mathcal D(m,q,\infty)$. 
For all $r\leq m$, we call $r-k_{r}(i)$ the \textit{jump size of} $(i,r)$. We define $i^1_r\in \bar{J}_r$ as the index of $\bar{J}_r$ with the largest jump size (we call $i^1_r$ the largest jump of $J_r$). Let $i^2_r\in \bar{J}_r \setminus \{i^1_r\}$ be the index with the largest jump size in 
{the \(J_r\)-cluster of $i_r^1$.}
If $|J_r|\geq 2$, we define $i^\star_r$ as the index with the largest jump in $\bar{J}_r\setminus\{i^1_r,i^2_r\}$.
\end{defin}
Then, we define for $u\in[1,N]$,  
\begin{equation}
	\label{eq-F}
	F(u) := F(u,\bar \sigma_N) := 
     \frac{1}{u}  \frac{1}{1-\bar{\sigma}_N^2 \log u}{=u^{-1}b_u},\quad \bar \sigma_N^2 := \pi^{-1} \sigma_N^2 .
\end{equation}
\begin{thm} \label{thm:sumOverXYv} 
Assume \eqref{eq:sub-criticality} and $p_0\in \mathbb N_{\geq 3}$. 
{Let $\varepsilon_N^{\hyperlink{symbol: star}{\diamondsuit}}$ from Lemma \ref{eq:key_lemma}.} The following holds: for all $s\leq t\leq N$, 
	\begin{equation}
		\label{eq:startingPointTriple+}
		\begin{aligned}
			\sup_{X_0} \Phi_{s,t,q}^{\leq p_0}(X_0) &\leq 1+ c_0 \sum_{m=1}^\infty  (1+\varepsilon_N^{\hyperlink{symbol: star}{\diamondsuit}})^m  \sum_{\mathbf J\in \mathcal D(m,q,p_0)} \sum_{\substack{u_2,\dots, u_{m} \in \llbracket 1,t\rrbracket\\u_1\in \llbracket s,t\rrbracket}} \\
&\qquad  {\bar{\sigma}_N^{2|J_1|} {\pi} p^{\star}_{2u_1}
} \prod_{r={2}}^{m}
			\bar{\sigma}_N^{2|J_r|} F(u_r+\bar{u}_r)  \left(\frac{c_0}{u_r+\bar u_r(i^\star_r)}\right)^{ \mathbf{1}_{|J_r|\geq 2}} \mathbf{1}_{\sum_{i=1}^m u_i\leq t},
		\end{aligned}
	\end{equation}
	where  $p_n^\star := \sup_{x\in \mathbb{Z}^2} p_n(x)$, $c_0=c_0(\hat \beta,p_0)>1$, and {with $u_0:=0$,}
\begin{equation}                \label{eq:ubar}
    \bar u_r := \frac{1}{2}\left(\bar u_r(i^1_r)+\bar u_r(i^2_r)\right),   \text{ with }\ \bar u_r(i) := \sum_{k=k_r(i)+1}^{r-1} u_k.
\end{equation}
\end{thm}
\begin{proof}
	We consider the right-hand side of \eqref{eq:expansion_first_bound}.
	In the following, we fix $J_1,\dots,J_m$,  as in \eqref{eq:expansion_first_bound} and let $p_r:=|\bar J_r| \leq p_0$.
	We begin by summing over $Y_m$, 	which, by \eqref{eq:def_UNXYJI},  gives a factor of
	\begin{equation} \label{eq:sumOverY_m}
		\sum_{Y_m \in \mathbb Z^{\bar{J}_m}} U_N(v_m,X_m,Y_m, J_m, I_m) \prod_{(i,j)\in I_m} \mathbf{1}_{y_m^i=y_m^j} = U_N(v_m,X_m,J_m, I_m),
	\end{equation}
	where, recalling \eqref{eq:defPsiGauss}, we have set:	\begin{equation} \label{eq:def_UNXJI}
		U_N(n,X, J,I) := \begin{cases}    \sigma_N^{ 2|I|}{\rm E}_{X}^{\otimes \bar J} \left[e^{\Lambda_2(\beta_N)\psi^\star_{n-1,\bar J}} \prod_{(i,j)\in I} \mathbf{1}_{S_t^i=S_t^j} \right] \mathbf{1}_{\bar I \neq \emptyset} & \text{if } n>0, \\
                 \mathbf{1}_{J = I} & \text{if } n=0.
		\end{cases}
	\end{equation}
By \eqref{eq:sumOnU_N},  summing \eqref{eq:sumOverY_m} over $I_m$ and $v_m$ yields a factor
	\begin{equation} \label{eq:bound_sum_bm_Ym}
		\sum_{v_m \leq t} \sum_{I_m\subset \mathcal C_{\bar J_m}} U_N(v_m,X_m, J_m,I_m) \leq \sup_{X\in (\mathbb{Z}^2)^{\bar{J}_m}} \rme^{\otimes \bar J_m}_{X} \left[ e^{\Lambda_2(\beta_N)\psi^\star_{t,\bar J_m}}\right] \leq c(p_0)
	\end{equation}
with some constant $c(p_0)>0$, where the second inequality uses the assumption $|\bar{J}_m|\leq p_0$ and {Lemma \ref{lem:LastLemma}}.
Then, we sum on $X_m$ which leads to the factor (write $z_m(i):=y_{k_m(i)}^i$)
	\begin{equation} \label{eq:sumX_m}
		\sum_{x_m^i\in \mathbb Z^2,i\in \bar J_m} \prod_{(i,j)\in J_m} \mathbf{1}_{x_m^i=x_m^j} \prod_{i\in \bar J_m} p_{w_m(i)}(x_m^i-z_m(i)).
	\end{equation}
	Recall the quantities $i_r^{\#}$ with ${\#} \in \{1,2,\star\}$ defined in Definition~\ref{def:defi_1i_2}. If $|J_m|=1$, then \eqref{eq:sumX_m} reduces to 
	\begin{equation} \label{eq:Kolmogorov}
		\begin{split}
			\sum_{x \in \mathbb Z^2}  &p_{w_m(i^1_m)}(x-z_m({i^1_m})) p_{w_m(i^2_m)}(x-z_m(i^2_m))\\
   &=p_{w_m(i^1_m)+w_m(i^2_m)}(z_m(i^1_m)-z_m(i^2_m)) \leq p_{w_m(i^1_m)+w_m(i^2_m)}^\star,
		\end{split}
	\end{equation}
{with $p^\star_n =\sup_{x} p_n(x)$.} 
If $|J_m|\geq 2$, we first use that $p_{w_m(i^\star_m)}(x_m(i^\star_m)-z_m(i^\star_m))\leq p_{w_m(i^\star_m)}^\star$. 
Then, denoting by $\mathcal J_m \subset \bar J_m$ the {\(J_m\)-\emph{cluster  of $i^1_m$},} 
we bound by 1 the heat kernels in \eqref{eq:sumX_m} of the particles of $\mathcal J_m$ that are different from {$i^1_m, i^2_m$ {and $i_m^\star$}}. 
{Because of the indicator functions $\mathbf{1}_{x_m^i=x_m^j}$ in \eqref{eq:sumX_m}, we have that $x_m^i=x_m^{j}$ for all  $i,j\in \mathcal J_m$. Hence, if we denote by $x$ this common position, we get that the sum in \eqref{eq:sumX_m} over all the positions $x_m^i$ for $i\in \mathcal J_m$ is bounded by the first line of \eqref{eq:Kolmogorov}.} 
This gives again an upper bound of $p_{w_m(i^1_m)+w_m(i^2_m)}^\star$. The sum on the remaining positions of $X_m$ (if there are) can be bounded by $1$.
	In any case, we find that \eqref{eq:sumX_m} is upper bounded by
	\begin{equation} \label{eq:ResultSumOnX_m}
		p^\star_{w_m(i_m^1)+w_m(i_m^2)}(p^\star_{w_m(i_m^\star)})^{\mathbf{1}_{|J_m|\geq 2}}.
	\end{equation}
 Next, we sum on $Y_{m-1}$. As in \eqref{eq:sumOverY_m}, it leads to a factor of $U_N(v_{m-1},X_{m-1},J_{m-1},I_{m-1})$.
	\footnote{Note that we could not sum on $Y_{m-1}$ before summing on $X_{m}$, as quantities like $z_m$ may depend on $Y_{m-1}$}
Now, we have to be more careful when summing on $I_{m-1}$ and $v_{m-1}$, as by \eqref{eq:ResultSumOnX_m}, this amounts to handling
	\begin{equation} \label{eq:complicated_sum}
		\sum_{v_{m-1}\leq t} \sum_{I_{m-1}\subset \mathcal C_{\bar J_{m-1}}} U_N(v_{m-1},X_{m-1},J_{m-1},I_{m-1}) p^\star_{w_m(i_m^1)+w_m(i_m^2)}(p^\star_{w_m(i_m^\star)})^{\mathbf{1}_{|J_m|\geq 2}}.
	\end{equation}
Define $\bar u_r$ and $\bar u_r(i)$ as in \eqref{eq:ubar}. Observe that since $\bar J_{m-1}\nsupseteq \bar J_m$, we have $i_m^1\notin \bar{J}_{m-1}$ {(otherwise, since $i_m^1$ is the longest jump of $\bar J_m$, all the other particles of $\bar J_{m}$ would also belong to $\bar J_{m-1}$)} and thus 
$i_m^1$ must come from a time distance at least $u_m+v_{m-1}+\bar{u}_{m}(i_m^1)$. Therefore
	\begin{equation}
		\label{eq:sumwmwm}
		\begin{aligned}
        w_m(i^1_m)+w_m(i^2_m) & \geq (u_m+v_{m-1}+\bar{u}_{m}(i_m^1))+ (u_m+\bar{u}_{m}(i_m^2)) \\
          & = v_{m-1} + 2(u_m+\bar{u}_m).
		\end{aligned}
	\end{equation}  
Moreover $w_m(i_m^\star) \geq u_m+\bar u_m(i_m^\star)$,
so since $n\to p^\star_n$ is non-increasing, we obtain from \eqref{eq:pnstar} and Lemma \ref{lem:sum_on_v} (notice that $\bar u_m \geq 1/2$ und $u_m\geq 1$) that the sum in \eqref{eq:complicated_sum}  is less than
	\begin{equation}
		(1+\varepsilon_N^{\hyperlink{symbol: star}{\diamondsuit}}) {\frac{1}{\pi}}F(u_m+\bar{u}_m)  \left(\frac{c_0}{u_m+\bar u_m(i^\star_m)}\right)^{ \mathbf{1}_{|J_m|\geq 2}},
	\end{equation}
where  $\varepsilon_N^{\hyperlink{symbol: star}{\diamondsuit}}$ is as in  Lemma \ref{eq:key_lemma}.	Repeat the same steps to obtain the statement of the lemma.
\end{proof}
The following notion of \textit{long jump} is inspired from \cite[Section 3.5]{Cosco2021MomentsUB}. 
\begin{defin} \label{def:longJump}  Let $K\geq 2$.  Given $ \mathbf{J} \in \mathcal D(m,q,p_0)$,
  $r\in \llbracket 1,m\rrbracket$, and $i\in \bar J_r$, we say that $(r,i)$ is a $K$-\textbf{long jump} if $r-k_r^i > K$ or $k_r^i=0$. We say that $(r,i)$ is a $K$-\textbf{small jump} otherwise.
\end{defin}
\begin{rem}
The set of $K$-long jumps depends on the diagram $\mathbf{J}$.
\end{rem}
Let $L_0\in\mathbb N\setminus\{1,2\}$ be a (large) parameter, whose explicit value will be determined later on (see, e.g.,  \eqref{eq:conditionAlpha0}),
and set
\begin{equation}
\epsilon_r( \mathbf{J}):=\mathbf{1}\{|J_r|\geq 2 \text{ and } (r,i_r^\star) \text{ is an {$L_0$}-long jump}\}.
\end{equation}

 The following statement is an immediate consequence of the above definition and Theorem \ref{thm:sumOverXYv}.

\begin{cor}  \label{cor:1/L} 
  {Under {the same assumptions} and with   $\varepsilon_N^{\hyperlink{symbol: star}{\diamondsuit}}$,  $c_0$ as in Theorem \ref{thm:sumOverXYv}}, for all $s\leq t\leq N$,
	\begin{equation} \label{eq:middlePointTriple+}
		\begin{aligned}
			&\sup_X \Phi_{s,t,q}^{\leq p_0}(X) \leq  1+ c_0  \sum_{m=1}^\infty  (1+\varepsilon_N^{\hyperlink{symbol: star}{\diamondsuit}})^m \\
   &\left\{\sum_{\mathbf{J}\in \mathcal D(m,q,p_0)} \prod_{r=1}^m \bar{\sigma}_N^{2 |J_r|} \left\{ \left(\frac{c_0}{L_0}\right)^{\epsilon_r( \mathbf{J})} {c_0^{(1-\epsilon_r(\mathbf{J}))}} \right\}^{\mathbf{1}_{|J_r|\geq 2}} \mathcal F^{\star}_{\mathbf J}(s,t)\right\},
		\end{aligned}
	\end{equation}
where 
\begin{equation} \label{eq:defFstar}
\mathcal F^{\star}_{\mathbf J}(s,t): = \sum_{\substack{u_2,\dots, u_{m} \in \llbracket 1,t\rrbracket \\u_1\in \llbracket s,t\rrbracket}} {\pi} p_{2u_1}^\star\prod_{r={2}}^{m}
F(u_r+\bar{u}_r) \mathbf{1}_{\sum_{i=1}^m u_i\leq t}.
\end{equation}
\end{cor}
\begin{proof}
{Apply \eqref{eq:startingPointTriple+}. If $\epsilon_r(\mathbf J)= 1$, we have $\bar u_r(i_r^\star) \geq 1+ (L_0-1) = L_0$ by definition of $\bar u_r(i)$ in \eqref{eq:ubar}, since $u_k\geq 1$ for all $k$. Thus, $(\frac{c_0}{u_r+u_r(i^\star_r)})^{ \mathbf{1}_{|J_r|\geq 2}} \leq \frac{c_0}{L_0}$ in that case. Otherwise, 
we use that $(\frac{c_0}{u_r+u_r(i^\star_r)})^{ \mathbf{1}_{|J_r|\geq 2}}\leq c_0$.}
\end{proof}

\subsection{The case $p_0=2$. 
} 
The following proposition is the analogue of Theorem \ref{thm:sumOverXYv} when $p_0=2$.
Recall $\mathcal D(m,q)$ above \eqref{eq:Phi2}. For $\mathbf J\in \mathcal D(m,q)$, define $\mathcal F^{\star}_{\mathbf J}(s,t)$ by \eqref{eq:defFstar} with 
\begin{equation} \label{eq:ubarp_0=2}
\bar u_r:= \frac{1}{2}\left(\bar u_r(i_r)+\bar u_r(j_r)\right), \quad \bar u_r(i):= \sum_{k=k_r(i)+1}^{r-1} u_k,
\end{equation}
and $k_r(i):=\sup\{k<r:i\in \{i_k,j_k\}\}$ (let $k_r(i):=0$ if the set is empty).
  {
\begin{remark}
The definition of $\bar u_r$ in \eqref{eq:ubarp_0=2} coincides with the one from \eqref{eq:ubar} for $p_0=2$, as $i_r^1 = i_r$ and $i_r^2=j_r$ if $i_r$ has the largest jump size, and $i_r^1=j_r$, $i_r^2=i_r$ otherwise.
\end{remark}}
\begin{prop} \label{prop:UBFQcrit} Assume  \eqref{eq:sub-criticality} or \eqref{eq:defNearCrit}. \label{thm:sumOverXYvBis} 
{ Let $\varepsilon_N^{\hyperlink{symbol: star}{\diamondsuit}}$  from Lemma \ref{eq:key_lemma}.} Recall $b_t =(1-\bar{\sigma}_N^2 \log t )^{-1}$. The following holds: for all $s\leq t\leq N$,
	\begin{equation}
		\label{eq:startingPointTriple+QUASI}
		\sup_{X_0} \Phi_{s,t,q}^{\leq 2}(X_0) \leq 1+ b_t \sum_{m=1}^\infty  (1+\varepsilon_N^{\hyperlink{symbol: star}{\diamondsuit}})^m\bar{\sigma}_N^{2m}  \sum_{ \mathbf{J} \in \mathcal D(m,q)} \mathcal F^{\star}_{\mathbf J}(s,t).
	\end{equation}
\end{prop}
\begin{proof} Starting from \eqref{eq:Phi2},
 the result is entailed by the proof of Theorem \ref{thm:sumOverXYv} with $p_0=2$.  
Note that the bound \eqref{eq:bound_sum_bm_Ym} is replaced by \eqref{eq:key_bound}, which by the assumption of Lemma \ref{eq:key_lemma} holds under {both sub-criticality and} quasi-criticality.\footnote{We mention as well that Lemma \ref{lem:LastLemma} is \emph{not} used in that case. This is relevant for the paper to be self-contained, cf.\@ the proof of Lemma \ref{lem:LastLemma}}
\end{proof}

\section{The induction scheme for sub-critical and quasi-critical} \label{sec:inductionSchemeSubcrit}
The present section is dedicated to estimating $\mathcal F^{\star}_{\mathbf J}(s,t)$ that appeared in \eqref{cor:1/L} and  \eqref{eq:startingPointTriple+QUASI}. This is the purpose of Proposition \ref{prop:fibo1inverseMain}, from which we derive the main results of this section, that are Proposition \ref{prop:Amn1} and \ref{prop:Phi2UB} below.
\subsection{The induction scheme}
Let $L\in\mathbb N\setminus\{1,2\}$ be a (large) parameter  whose explicit value will be determined later on (see, e.g.,  \eqref{eq:conditionAlpha0} or Section~\ref{sec:proofOfQuasi(ii)}). 
Recall the notion of {\it long jump} from Definition \ref{def:longJump}.
 \begin{defin} \label{def:maindef}  Given $ \mathbf{J} \in \mathcal D(m,q,p_0)$, we define:
	\begin{enumerate}[label=(\roman*)]
		
		\item We say that $r\in \llbracket 1,m\rrbracket$ is a \textbf{long jump index} if $(r,i^1_r)$ and $(r,i^2_r)$ are both $L$-long jumps.  Otherwise, we call $r$ a \textbf{small jump index}. We denote by   $s_1<\dots<s_{l_0}$ the set of all small jump indices and set $s_{l_0+1} := m+1$. 
		\item For all $l\in \llbracket 1, l_0+1\rrbracket$ such that $s_l - s_{l-1} > L$,
		      we define the indices $\{s_{l} - kL:k\in \mathbb N, s_{l}-kL
			      >s_{l-1}\}$ as \textbf{stopping} indices.
		\item We say that $r\in \llbracket 1,m-1\rrbracket$ is a \textbf{fresh index} if it is a long jump index such that $r$ is a stopping index or such that $r+1$ is a small jump index.  Moreover, if $m$ is a long jump index, then we define it as a fresh index.
		\item We say that $r\in \llbracket 1,m\rrbracket$ is a \textbf{good index} if $r$ is a long jump index that is not a fresh index. We call $r$ a \textbf{bad index} otherwise. We also set the index $0$ as a \textbf{bad index}.
	\end{enumerate}
\end{defin}
\begin{remark} \label{eq:rkbad}
	{Any stopping index is a fresh index.} Bad indices correspond to fresh or small jump indices {(or the index 0)}. {The index $1$ is always a long jump index. } We emphasize again that (i)-(v) depend on the diagram $ \mathbf{J}\in \mathcal D(m,q,p_0)$.
\end{remark}
\begin{remark}
    {When $p_0=2$, for all $\mathbf J=((i_1,j_1),\dots,(i_m,j_m))\in \mathcal D(m,q)$, we keep the same definition by replacing $i_r^1,i_r^2$ 
    with $i_r,j_r$.}
\end{remark}
For all $ \mathbf{J}\in \mathcal D(m,q,p_0)$ and $0\leq r \leq m$, let
\begin{equation} \label{eq:defPsi}
    \psi(r):=\psi(r,\mathbf{J}): =
	\sup\{r' \leq r:~ r' \text{ is a bad index}\}.
\end{equation}
{By definition, $\psi(r)\geq 0$ for all $r\geq 0$ since $0$ is a bad index.}\

Define further 
$u_{(r)} := \sum_{i=\psi(r)}^{r} u_i$ {with the convention that $u_0=0$}. Recall $\bar u_r$ from \eqref{eq:ubar} (see also \eqref{eq:ubarp_0=2} when $p_0=2$).
\begin{lem} {Let $r\geq 1$.}
	\begin{enumerate}[label=(\roman*)]
		\item \label{lemmatata1} If $r$ is good, then $r+1$ is a long jump index.
		\item \label{lemmatata2} If $r$ is good, then $\psi(r-1) = \psi(r)$.
		\item \label{lemmatata5} If $r$ is a long jump index, then
		      $\bar u_{r}\geq u_{(r-1)}$.
		\item \label{lemmatata6} If $r$ is a small jump, then $\psi(r)=r$ and $\psi(r-1)=r-1$.
   \item \label{lemmatata7} For all $r\geq 2$, {we have $\bar u_r \geq \frac{u_{r-1}}{2}$.}
	\end{enumerate}
	\label{lemmat}
\end{lem}
\begin{proof} \ref{lemmatata1} If $r$ is good, then it is a long jump index. As it is not fresh, $r+1$ must be a long jump index by definition of fresh indices.  \ref{lemmatata2} is clear by definition.  
For \ref{lemmatata5}, first observe that $\psi(r-1)\geq r-L$. {Indeed, by definition of stopping times, each interval of length strictly larger than $L$ intersects at least a stopping time or a small jump (note that $s_{l_0+1}=m+1$ is used here). For $r>L$, this applies to $\{r-L,\dots,r-1,r\}$, but as $r$ is a long jump, it also applies to $\{r-L,\dots,r-1\}$, as if $r$ was a stopping time, then either $\{r-L,\dots,r-1\}$ would contain a small jump, or else $r-L$ would be a stopping time. Hence $\{r-L,\dots,r-1\}$ contains a bad index, which gives our claim.} Now, given that $r$ is a long jump index, {there are two possibilities. If $k_r(i_2)>0$, then by definition $ \bar u_{r}$ is larger or equal to $u_{r-1} +\dots + u_{r-L}$, which in turn is larger or equal to $u_{(r-1)}$ since $\psi(r-1)\geq r-L$. If $k_r(i_2)=0$, then $\bar u_r = \sum_{k=1}^{r-1} u_k \geq u_{(r-1)}$ since $\psi(r-1)\geq 0$ {and $u_0=0$}.} 
Regarding \ref{lemmatata6}, $\psi(r)=r$ by definition, and since $r-1$ is necessarily bad  (cf.\ \ref{lemmatata1}), we have $\psi(r-1)=r-1$. 
For \ref{lemmatata7}, since $\bar{u}_r(i^1_r) \geq u_{r-1}$ {as $k_r(i^1_r)<r-1$ for $r\geq 2$},
 the claim follows from the definition of $\bar u_r$ in \eqref{eq:ubar}. 
\end{proof}

For all
$v\in [1,t]$ and $t\leq N$, let
\begin{equation} \label{eq:def_fbisbis}
	f(v):=f_{t,N}(v):= \frac{1}{\bar{\sigma}_N^2}  \log \left(\frac{1-\bar{\sigma}_N^2\log v}{1-\bar{\sigma}_N^2\log t }\right),
\end{equation}
which is positive, non-increasing and satisfies $f'(v)=-F(v)$.

Lemma \ref{lem:fibo1Main}  below is an extension of \cite[Lemma 3.10, 3.12 and Proposition 3.14]{Cosco2021MomentsUB} to suit both the sub-critical and quasi-critical case -- in the sub-critical case, they are equivalent. (Note in particular that the definition of $f$ above  \cite[Lemma 3.10]{Cosco2021MomentsUB} is generalized here to fit 
both regimes).
Although the proof follows in spirit the one of \cite{Cosco2021MomentsUB}, several parts differ to take into account the change of function $\psi$ and to adapt to the near-critical case. Moreover, we introduce here some operators $\mathcal F^{\star}$, as similar ones will appear in Section \ref{sec:quasiCritRemoveTriple+}. We hope that this formalism will help the reader to easily see the analogy. 
Besides, a noteworthy improvement is our estimate \eqref{eq:lemmaInduction2}, which deals with  
the case of a later starting time $s$.

For any function $h$ and $v\leq t\leq N$, set
\begin{equation*}
	\mathcal F^{\rm fresh} h(v) :=   \sum_{u=1}^t
	F(u + v)  h(u) \mathbf{1}_{u+v \leq t}, \  
	\mathcal F^{\rm good} h(v): = \sum_{u=1}^t F(u+v) h(u+v) \mathbf{1}_{u+v\leq t},
\end{equation*}
and $\mathcal F^{\rm small}h(v) := \mathcal F^{\rm fresh}h(v/2)$. {Define the function $f^j$ by $f^j(s):= f(s)^j$.}
\begin{lem} \label{lem:fibo1Main}
Assume \eqref{eq:sub-criticality} or \eqref{eq:defNearCrit}. Then,
for all $j\geq 0$ and $v\in \llbracket 1,t\rrbracket$, {$t\leq N$,}
\begin{equation}\label{eq:indRedMain}
    \begin{aligned}
	\mathcal F^{\rm fresh} f^j(v)\leq        \frac{f(v)^{j+1}}{j+1} + \sum_{i=0}^{j} \frac{j!}{(j-i)!} (b_t)^{i+1} f(v)^{j-i},
    \end{aligned}
\end{equation}
	with $b_t= (1-\bar \sigma_N^2 \log t )^{-1}$,
{while
\begin{equation} \label{eq:calF(v/2)}
    \mathcal F^{\rm small} f^j(v)\leq        \frac{f(v)^{j+1}}{j+1} + 2\sum_{i=0}^{j} \frac{j!}{(j-i)!} (b_t)^{i+1} f(v)^{j-i}.
\end{equation}
} 
Furthermore, for $v\in \llbracket 1,t\rrbracket$, 
\begin{equation} \label{eq:w+vIndRed}
    \mathcal F^{\rm good} f^{j}(v) \leq \frac{f(v)^{j+1}}{j+1}  .
\end{equation}
\end{lem}
\begin{proof}
Since $F$ and $f$ are non-increasing { ($f'=-F\leq 0$ and see Lemma \ref{lem:F'})}, 
\begin{equation} \label{eq:ineqcalF'}
		\mathcal F^{\rm good} f^j(v)=\sum_{u=1}^t
		F(u+v) f(u+v)^{j} \mathbf{1}_{u+v \leq t} \leq \int_{v}^{t} F(x) f(x)^j {\rm d} x = \frac{f(v)^{j+1}}{j+1},
\end{equation}
as $f'(x)=-F(x)$. 
  Then, we decompose $\mathcal F^{\rm fresh} = \mathcal F^{\leq v} + \mathcal F^{>v}$, with
\begin{equation} \label{eq:calFgeqv}
  \mathcal F^{>v} f^j(v): = 
		\sum_{u=v+1}^t F(u+v)f(u)^j \mathbf{1}_{u+v \leq t} \,{\rm d} u\leq \int_{v}^{t} F(x) f(x)^j {\rm d} x = \frac{f(v)^{j+1}}{j+1},
\end{equation}
as in \eqref{eq:ineqcalF'}, and
\begin{equation} \label{eq:calFleqv}
    \mathcal F^{\leq v} f^j(v): = \sum_{u=1}^v F(u+v)f(u)^j \mathbf{1}_{u+v \leq t} \leq F(v) \sum_{u=1}^v f(u)^j,
\end{equation}
by monotonicity of $F$. Since $v\leq t$, $F(v)\leq v^{-1} b_t$, and we get
\[
\mathcal F^{\leq v} f^j(v) \leq v^{-1} b_t \left(f(1)^j+\int_{1}^v f(u)^j \mathrm{d} u\right).
\]
By integration by parts and using that $-f'(x)=F(x)= \frac{1}{x}\frac{1}{1-\bar{\sigma}_N^2 \log x}\leq x^{-1} b_t$ for any $x\leq t$, we see that for all $j\geq 1$ and $v\leq t$,
\begin{align*}
    f(1)^j + \int_{1}^v f(x)^j\mathrm{d} x & = vf(v)^j - j\int_{1}^v xf'(x) f(x)^{j-1}\mathrm{d} x \leq vf(v)^j + j b_t \int_{1}^v f(x)^{j-1}\mathrm{d} x.
\end{align*}
	If we iterate the integration by parts,  we obtain that
	\begin{align*}
		f(1)^j + \int_{1}^v f(x)^j\mathrm{d} x \leq v \sum_{i=0}^j \frac{j!}{(j-i)!} (b_t)^i  f(v)^{j-i},
	\end{align*}
	and so
\begin{equation} \label{eq:IPP}
		\mathcal F^{\leq v}  f^j(v) \leq \sum_{i=0}^j \frac{j!}{(j-i)!} (b_t)^{i+1}   f(v)^{j-i}.
\end{equation}
{Combining \eqref{eq:calFgeqv} and \eqref{eq:IPP} leads to  \eqref{eq:indRedMain}}. {Finally, by \eqref{eq:calFgeqv} we have
$\mathcal F^{>v} f^j(v/2) \leq  \frac{f(v)^{j+1}}{j+1}$; moreover, using that $F(u+v/2)\leq 2F(u+v)$ {in \eqref{eq:calFleqv}}, we get from \eqref{eq:calFleqv} and \eqref{eq:IPP} that 
$\mathcal F^{\leq v}  f^j(v/2)$ is at most $2$ times the right-hand side of \eqref{eq:IPP}, entailing \eqref{eq:calF(v/2)}.
}
\end{proof}

Define, for any function $g(u_1,\dots,u_r)$,
\begin{equation}
    \mathcal F_r g (u_1,\dots,u_{r-1}) := \sum_{u_r=1}^t  F(u_r+\bar u_r) g(u_1,\dots,u_r) \mathbf{1}_{\sum_{i=1}^r u_i \leq t}.
\end{equation}
Let $g_j(u_1,\dots,u_r) = f^j(u_{(r)})$.

\begin{lem} \label{lem:indReductionMain}
	For all $r\in \llbracket 2,m \rrbracket$, $j\geq 1$ and $\sum_{i=1}^{r-1} u_i \leq t$,
	\begin{equation}\label{eq:indRedMainRec}
		\begin{aligned}
			\mathcal F_r 
            {g_j} (u_1,\dots,u_{r-1})  \leq   \frac{f^{j+1}\left(u_{(r-1)}\right)}{j+1} + \mu_{r}^m   \sum_{i=0}^{j} \frac{j!}{i!} (b_t)^{j-i+1} f^{i}\left(u_{(r-1)}\right),
		\end{aligned}
	\end{equation}
	where $\mu_r^m := \mu_r^{m} (\mathbf{J}) :=  \mathbf{1}_{\{r \text{ is bad\}}}{+\mathbf{1}_{\{r \text{ is a small jump\}}}}$.
\end{lem}
\begin{proof}
	If $r$ is good, then  Lemma \ref{lemmat}--\ref{lemmatata2} implies that $\psi(r) = \psi(r-1)$, and hence $u_{(r)} = u_r + u_{(r-1)}$.
	As $F$ is non-increasing {(see Lemma \ref{lem:F'})},
	we obtain from Lemma \ref{lemmat}--\ref{lemmatata5} that $F(u_r+\bar u_r) \leq F(u_r+ u_{(r-1)})$, and therefore
	\begin{align*}
		\mathcal F_r g_j (u_1,\ldots,u_{r-1})&\leq \sum_{u_r=1}^t F\left(u_{r}+u_{(r-1)}\right)  f^{j}\left(u_{r} + u_{(r-1)}\right)\mathbf{1}_{u_{r} + u_{(r-1)} \leq t}\\
        & = \mathcal F^{\rm good}f^j(u_{(r-1)}),
	\end{align*} 
hence \eqref{eq:w+vIndRed} gives \eqref{eq:indRedMainRec}. If $r$ is fresh, then  by definition $\psi(r)=r$, and thus $u_{(r)} = u_r$. Using again that $F$ is non-decreasing and Lemma \ref{lemmat}--\ref{lemmatata5}, we obtain that
\begin{equation} \label{eq:badBoundMain}
	\begin{aligned} 
		\mathcal F_r g_j(u_1,\ldots,u_{r-1}) & \leq \sum_{u_r =1}^t  F\left(u_{r}+u_{(r-1)}\right)  f^j\left(u_{r}\right)\mathbf{1}_{u_{r} + u_{(r-1)} \leq t} \\
        & = \mathcal F^{\rm fresh} f^j(u_{(r-1)}).
	\end{aligned}
 \end{equation}
 So, we get \eqref{eq:indRedMainRec} from \eqref{eq:indRedMain}. Otherwise, $r$ is necessarily a small jump (see Remark \ref{eq:rkbad}), so by Lemma \ref{lemmat}--\ref{lemmatata6}, one has $\psi(r)=r$ and $\psi(r-1)=r-1$, giving $u_{(r)}=u_r$ and $u_{(r-1)}=u_{r-1}$. {By Lemma \ref{lemmat}--\ref{lemmatata7}, we get $F(u_r+\bar u_r) \leq F(u_r + u_{(r-1)}/2)$ and 
$\mathcal F_r g_j(u_1,\ldots,u_{r-1}) \leq \mathcal F^{\rm small} f^j(u_{(r-1)})$. 
We end the proof using \eqref{eq:calF(v/2)}.}
\end{proof}

 The following lemma is obtained by iterating Lemma \ref{lem:indReductionMain}. Let $\mathbb{1}{:\mathbb R^m \to \mathbb R}$ be the constant function with unit value. Recall the definition of $\mathcal F^{\star}_{\mathbf J}(s,t)$ from  \eqref{eq:defFstar}. We define
\[\gamma_{m-1}^m:=1,\quad \gamma_k^m := \mathbf{1}_{\{m-k \text{ is bad}\}},\quad {k\in \llbracket 0,m-2\rrbracket}.
\]
\begin{prop} 
	\label{prop:fibo1inverseMain}
	Assume \eqref{eq:sub-criticality} or \eqref{eq:defNearCrit}. Then, for all
	$m\geq 2$, $ \mathbf{J}\in \mathcal D(m,q,p_0)$,  $k\in \llbracket 1,m-1 \rrbracket$ and $\sum_{i=1}^{m-k} u_i \leq t$ with $u_i\in \llbracket 1,t\rrbracket$, for all $t\leq N$,
	\begin{equation} \label{eq:lemmaInduction0Main}
		\begin{aligned}
			 & (\mathcal F_{m-k  {+1}}\circ {\dots} \circ \mathcal F_{m} \mathbb{1} )(u_{1},\dots,u_{m-k})\leq \sum_{j=0}^k \frac{c_{k-j}^k}{j!}  (b_t)^{k-j} f^{j}\left(u_{(m-k)}\right),
		\end{aligned}
	\end{equation}
	with $c_0^0:=1$, 
 $c_i^{k+1}:= c_{i}^k +{2}\gamma_{k}^m
		\sum_{j=0}^{i-1} c_{j}^k$ for $i\leq k+1$  and $c_i^k:=0$ for $i>k$.
	
Furthermore, we have for all $s\leq t\leq N$,
	\begin{equation} \label{eq:lemmaInduction2}
		\mathcal F^{\star}_{\mathbf J}(s,t)
		\leq \sum_{j=0}^m \frac{c_{i}^m}{(m-j)!} (b_t)^j  f\left(s\right)^{m-j}.
	\end{equation}
\end{prop}
\begin{remark}
	$(c_i^k)_{i,k}$ depend on the choice of the diagram $ \mathbf{J}\in \mathcal D(m,q,p_0)$.
\end{remark}
\begin{proof}
{Observe that \eqref{eq:lemmaInduction0Main} holds trivially true for $k=0$ with the convention that an empty composition is the identity.} We then proceed by induction. Let $1<k<m-1$ be such that \eqref{eq:lemmaInduction0Main} holds with $k-1$, i.e., 
 \al{
 (\mathcal F_{m-k+2}\circ {\dots} \circ \mathcal F_{m} \mathbb{1} )(u_{1},\dots,u_{m-k+1})\leq \sum_{j=0}^{k-1} \frac{c_{k-1-j}^{k-1}}{j!}  (b_t)^{k-1-j} f^{j}\left(u_{(m-k+1)}\right).
 }
{This implies that the left-hand side of \eqref{eq:lemmaInduction0Main} is less than
\[
\sum_{j=0}^{k-1} \frac{c_{k-1-j}^{k-1}}{j!}  (b_t)^{k-1-j}  \sum_{u_{m-k+1}=1}^t \mathcal F_{m-k+1} g_j \left(u_1,\dots,u_{m-k}\right),
\]
}
By Lemma \ref{lem:indReductionMain} (use $\mu_{m-k}^m{\leq 2 \gamma_{k}^m}$), this is in turn smaller than
\begin{align*}
  &\sum_{j=0}^{k-1} \frac{c_{k-1-j}^{k-1}}{j!}  (b_t)^{k-1-j} \left(  \frac{f^{j+1}\left(u_{(m-k)}\right)}{j+1} + {2} \gamma^m_{k} \sum_{i=0}^{j} \frac{j!}{i!} (b_t)^{j-i+1} f^{i}\left(u_{(m-k)}\right)\right).
\end{align*}
Expanding the last parenthesis, the last display writes $A+B$, where
\[
A := \sum_{j=0}^{k-1} \frac{c_{k-1-j}^{k-1}}{(j+1)!}  (b_t)^{k-1-j}  f^{j+1}\left(u_{(m-k)}\right) = 
\sum_{i=0}^k \frac{c_{k-i}^{k-1}}{i!}  (b_t)^{k-i}  f^{i}\left(u_{(m-k)}\right),
\]
by setting $i:=j+1$ and using that $c_{k}^{k-1}=0$; exchanging the sums on $j$ and $i$,
\[
B := 2 \gamma^m_{k}\sum_{j=0}^{k-1} \sum_{i=0}^{j} \frac{1}{i!} (b_t)^{k-i} f^i(u_{(m-k)})  c_{k-1-j}^{k-1} = 2 \gamma^m_{k}\sum_{i=0}^{k} \frac{1}{i!} (b_t)^{k-i} f^i(u_{(m-k)}) \sum_{j=i}^{k-1} c_{k-1-j}^{k-1},
\]
where for $i=k$, the sum on $j$ of the right-hand side equals zero. Therefore, the left-hand side of \eqref{eq:lemmaInduction0Main} is bounded by
\al{
\sum_{i=0}^{k} \frac{1}{i!} (b_t)^{k-i} f^i(u_{(m-k)}) \left( c_{k-i}^{k-1} + 2 \gamma_{k}^m\sum_{j=i}^{k-1} c_{k-1-j}^{k-1} \right).
}
This yields \eqref{eq:lemmaInduction0Main} with $c_{k-i}^{k} = c_{k-i}^{k-1} + 2 \gamma_{k}^m\sum_{l=0}^{k-i-1} c_{l}^{k-1}$. 

Finally, we show \eqref{eq:lemmaInduction2}.   Observe that $\mathcal F^{\star}_{\mathbf J}(s,t) =\sum_{u_1=s}^t {\pi} p^{\star}_{2u_1} (\mathcal F_{2}\circ {\dots} \circ \mathcal F_{m} \mathbb{1} )(u_{1})$, 
which by \eqref{eq:lemmaInduction0Main} is bounded by 
	\[
		A'_{m,N}( \mathbf{J}):=   \sum_{u_1=s}^t {\pi} p^{\star}_{2u_1} \sum_{i=0}^{m-1} \frac{ c_i^{m-1}}{(m-1-i)!}
		f(u_1)^{m-1-i} (b_t)^i.
	\]
	Note that ${\pi} p^{\star}_{2u} \leq F(u)$ by \eqref{eq:pnstar}. Hence, as $F(u)f(u)$ is non-increasing {and $p^\star \leq 1$},
	\begin{align*}\sum_{u_1=s}^t  p_{2u_1}^\star f(u_1)^{m-1-i} & \leq f(s)^{m-1-i} + \int_{s}^t  F(u) f(u)^{m-1-i}\mathrm{d} u  = f(s)^{m-1-i} + \frac{f(s)^{m-i}}{m-i},
\end{align*}
and thus
	\begin{equation*}
		A'_{m,N}( \mathbf{J})\leq   \sum_{i=0}^{m} \frac{ c_i^{m-1}+c_{i-1}^{m-1}\mathbf{1}_{i\geq 1}}{(m-i)!}
		f(s)^{m-i}
		(b_t)^i. 
	\end{equation*}
	This yields  \eqref{eq:lemmaInduction2} since $c_i^{m-1}+c_{i-1}^{m-1}\mathbf{1}_{i\geq 1}\leq c_{i}^{m-1} +2\gamma_{m-1}^m
		\sum_{j=0}^{i-1} c_{j}^{m-1} = c_{i}^m$, {as $\gamma_{m-1}^m=1$.} 
\end{proof}

We continue with a simple yet important modification of \cite[Lemma 3.15]{Cosco2021MomentsUB}. This is key in the improvement of the $q$ threshold in \eqref{eq:q-conditionSubClassic} compared to \cite{Cosco2021MomentsUB}.\footnote{More precisely,  the factor $3$ in the left-hand side of Eq.\ (7) in \cite{Cosco2021MomentsUB} was coming from the factor $3^i$ in Lemma 3.15 therein. Here in Lemma \ref{lem:estimate_cim}, we replace it with $(1+\delta)^i$, at the cost of paying a greater constant $\delta^{-1}$ in the power of $\sum_k \gamma_k$. Hopefully, this can be neglected, as we will see later in Section \eqref{sec:proofMainProp}}

For all $\mathbf J \in\mathcal D(m,q,p_0)$,
let $n_{\mathrm{bad}}(\mathbf J) := {\sum_{k=0}^{m-1} \gamma_{k}^m = 1+\sum_{k=2}^m} \mathbf{1}_{k \text{ is bad}}$ and set $n_{\mathrm{small}}(\mathbf J):= \sum_{k=1}^m \mathbf{1}_{k \text{ is small jump}}$.

\begin{lem} \label{lem:estimate_cim}For all $\delta >0$, we have
	$
		c_{i}^m\leq (1+\delta)^i\left(1+2 \delta^{-1}\right)^{n_{\mathrm{bad}}(\mathbf J)}$
\end{lem}
\begin{proof}
	We claim that for any $\delta >0$ and $k\leq m$,
	$c_{i}^k\leq (1+\delta)^i (1+2\delta^{-1})^{\sum_{r=0}^{k-1} \gamma_r^m}$, which gives the lemma. 
	We prove the claim by induction. It is clear for $k=0$. Suppose that the result is true for some $k\in \mathbb N$. Then, we have
	\begin{align*}
		c_{i}^{k+1} = c_i^k + 2 \gamma_k^m \sum_{j<i} c_j^k & \leq  \left((1+\delta)^i + 2 \gamma_k^m \sum_{j<i} (1+\delta)^j \right) \left(1+2 \delta^{-1}\right)^{\sum_{r=0}^{k-1}\gamma_r^m}    \\
      & =  \left((1+\delta)^i+ 2 \gamma_k^m\frac{(1+\delta)^{i}-1 }{\delta}\right)\left(1+2 \delta^{-1}\right)^{\sum_{r=0}^{k-1} \gamma_r^m}.
\end{align*}
This is less than $ (1+\delta)^i \left(1+ 2 \gamma_k^m \delta^{-1}\right)\left(1+2 \delta^{-1}\right)^{\sum_{r=0}^{k-1} \gamma_r^m}                 = (1+\delta)^i\left(1+2 \delta^{-1}\right)^{\sum_{r=0}^{k}\gamma_r^m}.$
\end{proof}

We end with a useful observation. 
\begin{lem}\label{lem:gamma_n(I)} Let $p_0\geq 2$ and $\mathbf J \in\mathcal D(m,q,p_0)$. Then,
\[n_{\mathrm{bad}}(\mathbf J)\leq 2n_{\mathrm{small}}( \mathbf{J})+ m L^{-1}+{2}.
\]
\end{lem}
\begin{proof}
Note that if $k$ is bad, then either $k$ is a small jump, or $k+1$ is a small jump, or $k$ is stopping index. By definition, the number of stopping indices is at most $m L^{-1}$. The $+2$ comes from the fact that $m$ is always a bad index and that $\gamma_{m-1}^m=1$.
\end{proof}

\subsection{Conclusion} 
Recall $\lambda^2_{s,t,N}$ from \eqref{eq:def_lambda_T_Nbis}, $f$ from \eqref{eq:def_fbisbis}, $b_t=(1-\bar{\sigma}_N^2 \log t )^{-1}$,  and $c_i^m$ is defined in Proposition \ref{prop:fibo1inverseMain}.
\begin{prop} \label{prop:Amn1}
 Assume \eqref{eq:sub-criticality}, $p_0\geq 2$, 
 {let $c_0$ be  as defined in Theorem \ref{thm:sumOverXYv}} {and let $\varepsilon_N^{\hyperlink{symbol: star}{\diamondsuit}}$ from Lemma \ref{eq:key_lemma}}. The following holds:
    \begin{equation} \label{eq:endingPointTriple+} 
		\sup_X \Phi_{s,t,q}^{\leq p_0}(X) \leq 1+ c_0\sum_{m=1}^\infty  (1+\varepsilon_N^{\hyperlink{symbol: star}{\diamondsuit}})^m \sum_{\mathbf{J}\in \mathcal D(m,q,p_0)}  A(\mathbf J),
    \end{equation}
	where
\[
A(\mathbf J) = \prod_{r=1}^m \bar{\sigma}_N^{2 |J_r|}  
        \left\{ \left(\frac{c_0}{L_0}\right)^{\epsilon_r( \mathbf{J})} c_0^{(1-\epsilon_r(\mathbf{J}))} \right\}^{\mathbf{1}_{|J_r|\geq 2}}\sum_{i=0}^m \frac{c_{i}^m}{(m-i)!} (\lambda_{s,t,N}^2/\bar{\sigma}_N^2)^{m-i}(b_t)^{i}.
\]
\end{prop}
\begin{proof}
  {Observe that $f(s)=\lambda_{s,t,N}/\bar{\sigma}_N^2$.}	The result follows from Corollary \ref{cor:1/L} and  \eqref{eq:lemmaInduction2}.
\end{proof}
\begin{prop} \label{prop:Phi2UB} 
	Assume  \eqref{eq:sub-criticality} or \eqref{eq:defNearCrit}. {Let $\varepsilon_N^{\hyperlink{symbol: star}{\diamondsuit}}$ from Lemma \ref{eq:key_lemma}.} The following holds:
	\begin{equation*}
		\sup_X \Phi_{s,t,q}^{\leq 2}(X) \leq 1+  b_t\sum_{m=1}^\infty (1+\varepsilon_N^{\hyperlink{symbol: star}{\diamondsuit}})^m \sum_{\mathbf{J}\in \mathcal D(m,q)} \bar{\sigma}_N^{2m } \sum_{i=0}^m \frac{c_{i}^m}{(m-i)!} {(\lambda_{s,t,N}^2/\bar{\sigma}_N^2)^{m-i}}{\left(b_t \right)^{i}}.
	\end{equation*}
\end{prop}
\begin{proof}
As in  Proposition \ref{prop:Amn1}, this follows from Proposition \ref{prop:UBFQcrit} and \eqref{eq:lemmaInduction2}.
\end{proof}

\label{sec:inductionsection}
\section{Proof of Proposition \ref{prop:mainProp}}
\label{sec:mainProp}

\subsection{Large $q$ estimate for $p_0=2$}
\label{sec:largeqquasicrit}
In this section, we assume  \eqref{eq:sub-criticality} or \eqref{eq:defNearCrit}, and $p_0=2$.
Let $C_{L,\delta,q}:= \binom{q}{2} + 2(1+2\delta^{-1})^{{2}} Lq$.

\begin{lem} \label{lem:largeq_QC}  
{Let $\varepsilon_N^{\hyperlink{symbol: star}{\diamondsuit}}$ from Lemma \ref{eq:key_lemma}}.
Whenever
\begin{equation} \label{eq:conditionAlpha02}
    \alpha_0:=\alpha_0(L,\delta,q,T,N):= 
 C_{L,\delta,q} \left(1+2 \delta^{-1}\right)^{1/L} {\bar{\sigma}_N^{2} b_t (1+\delta) }(1+\varepsilon_N^{\hyperlink{symbol: star}{\diamondsuit}}) < 1,
\end{equation}
it holds that 
\begin{equation} \label{eq:finalBound02}
	\sup_X \Phi_{s,t,q}^{\leq 2}(X)\leq  \frac{b_t(1+2\delta^{-1})^{{2}}}{1-
 \alpha_0} {\frac{\binom q 2}{\binom q 2-1}} \exp\left\{ (1+\varepsilon_N^{\hyperlink{symbol: star}{\diamondsuit}}) \left(1+{2}\delta^{-1}\right)^{1/L} C_{L,\delta,q} \lambda_{s,t,N}^2 \right\}.
\end{equation}
\end{lem}
\begin{proof}
 By Lemma \ref{lem:estimate_cim} and Proposition \ref{prop:Phi2UB}, we obtain that
\begin{equation} \label{eq:boundSupXPhip0=2}
\begin{aligned}
	\sup_X \Phi_{s,t,q}^{\leq 2}(X) & \leq 1 + b_t\sum_{m=1}^\infty (1+\varepsilon_N^{\hyperlink{symbol: star}{\diamondsuit}})^{m} \bar{\sigma}_N^{2 m}   \\
 &\qquad    \sum_{ \mathbf{J}\in \mathcal D(m,q)}  \sum_{i=0}^{m} \frac{ (1+\delta)^i (1+2\delta^{-1})^{n_{\mathrm{bad}}(\mathbf J)}}{(m-i)!}
{(\lambda_{s,t,N}^2/\bar{\sigma}_N^2)^{m-i}}{(b_t)^{i}},
\end{aligned}
\end{equation}
with $\mathcal D(m,q)$ defined above \eqref{eq:Phi2}.  Let $\mathcal{D}^{(n)}(m,q)$ be the subset of all $\mathbf{J} \in \mathcal{D}(m,q)$ such that $n_{\mathrm{small}}(\mathbf{J}) = n$, where $n_{\mathrm{small}}(\mathbf{J})$ is the number of small jump indices in $\mathbf{J}$.
Then, we estimate
\[
{|\mathcal{D}^{(n)}(m,q)| \leq \binom{m-1}{n} \binom{q}{2} \left(\binom{q}{2} - 1\right)^{m-1-n} (2qL)^{n}.}
\]
Indeed, one first chooses where the small jumps occur, which gives a $\binom{m{-1}}{n}$ factor {(the first index is always a long jump)}. Then, for the first step, choose between $\binom{q}{2}$ pairs the one that meets. For each subsequent step, if a small jump occurs, the particle involved in the small jump must have appeared in a pair in the last $L$ indices. This gives at most $2L$ possibilities for this particle, hence at most $2qL$ possibilities in total. 
For each long jump index, we bound the choice by $\binom{q}{2} - 1$, as this corresponds to the total number of pairs except one (since a new pair has to be considered at each step).
Thus, it follows by Lemma \ref{lem:gamma_n(I)} that
\begin{equation} \label{eq:afterLemmacik2}
\sup_X \Phi_{s,t,q}^{\leq 2}(X)  \leq 1 + b_t\sum_{m=1}^\infty (1+\varepsilon_N^{\hyperlink{symbol: star}{\diamondsuit}})^{m} \bar{\sigma}_N^{2 m}  H_m \sum_{i=0}^{m} \frac{ (1+\delta)^i }{(m-i)!}
\frac{(\lambda_{s,t,N}^2/\bar{\sigma}_N^2)^{m-i}}{\left(1-\bar{\sigma}_N^2\log t \right)^{i}},
\end{equation}
where
\begin{align*}
  H_m &= (1+2\delta^{-1})^{m/L{+2}} \binom q 2  \sum_{n=0}^{m{-1}} \binom{m{-1}}{n}
 \left(\binom{q}{2}-1\right)^{m{-1}-n} (2qL)^{n} (1+2\delta^{-1})^{{2}n} \\
 &= (1+2\delta^{-1})^{m/L{+2}} \binom q 2 (C_{L,\delta,q})^{m{-1}},
\end{align*}
with {$C_{L,\delta,q}:= \binom{q}{2}-1 + 2(1+2\delta^{-1})^{{2}} Lq {\geq \binom q 2}$.} 
By exchanging the sum in $m$ and $i$ in \eqref{eq:afterLemmacik2}, we obtain for all $s\leq t,$
\begin{align}
	\sup_X \Phi_{s,t,q}^{\leq 2}(X) & \leq b_t (1+2\delta^{-1})^{{2}} \sum_{i=0}^\infty (1+\varepsilon_N^{\hyperlink{symbol: star}{\diamondsuit}})^{i} (C_{L,\delta,q})^i \left(1+2 \delta^{-1}\right)^{i/L} \left({\bar{\sigma}_N^{2}(1+\delta) }b_t\right)^i \label{eq:sumoni2}\\
          & \hspace{1.5cm} \times \sum_{m=i}^\infty (1+\varepsilon_N^{\hyperlink{symbol: star}{\diamondsuit}})^{m-i}(C_{L,\delta,q})^{m-i} \left(1+2 \delta^{-1}\right)^{(m-i)/L} \frac{\lambda_{s,t,N}^{2(m-i)}}{(m-i)!},\label{eq:sumonm2}
\end{align}
where the sum on $m$ in \eqref{eq:sumonm2} equals the exponential in the right-hand side of \eqref{eq:finalBound02}, and the remaining sum on $i$ in \eqref{eq:sumoni2} is geometric and converges under condition \eqref{eq:conditionAlpha02}, which ensures that the ratio of the geometric sum is strictly less than one. 
\end{proof}

\subsection{Small $q$ estimate for $p_0=2$}\label{sec:finiteq}
In this section, we assume  \eqref{eq:sub-criticality} or \eqref{eq:defNearCrit}, and $p_0=2$. Let $C_{\delta,q} := (1+2\delta^{-1}) (\binom q 2 - 1)$.
\begin{lem} \label{lem:smallq_QC}
Let {$\varepsilon_N^{\hyperlink{symbol: star}{\diamondsuit}}$ from Lemma \ref{eq:key_lemma}}. 
Whenever
\begin{equation} \label{eq:defalpha1}
\alpha_1:= 
\alpha_1(\delta,q,t,N)  := C_{\delta,q} b_t  \bar{\sigma}_N^{2}(1+\delta)(1+\varepsilon_N^{\hyperlink{symbol: star}{\diamondsuit}}) < 1,
\end{equation}
it holds that
\begin{equation}\label{eq:errorTermOptimizing}
\sup_X \Phi_{s,t,q}^{\leq 2}(X) \leq   \frac{b_t}{1-\alpha_1} \frac{\binom q 2}{\binom q 2 - 1} \exp\left\{\lambda_{s,t,N}^2 \left(\binom q 2 - 1\right) (1+2\delta^{-1})(1+\varepsilon_N^{\hyperlink{symbol: star}{\diamondsuit}}) \right\}.
\end{equation}
\end{lem}
\begin{proof}
We now focus on a fixed parameter $q_N=q\in \mathbb N\setminus\{1,2\}$.  Since $n_{\rm small}(\mathbf J)\leq m$, we obtain from \eqref{eq:boundSupXPhip0=2} that\footnote{Note that there is no dependence on $L$ in the upper bound of \eqref{eq:boundFiniteFixedq}. In fact, one can recover \eqref{eq:boundFiniteFixedq} by repeating Section \ref{sec:inductionsection} and defining every index as a small jump. Doing so, there is no need to introduce the parameter $L$}
\begin{equation} \label{eq:boundFiniteFixedq}
\begin{aligned}
	&\sup_X \Phi_{s,t,q}^{\leq 2}(X) \\
 & \leq 1 + b_t\sum_{m=1}^\infty (1+\varepsilon_N^{\hyperlink{symbol: star}{\diamondsuit}})^{m} \bar{\sigma}_N^{2 m}  |\mathcal D(m,q)| \sum_{i=0}^{m} \frac{ (1+2\delta^{-1})^{m}}{(m-i)!}
{(\lambda_{s,t,N}^2/\bar{\sigma}_N^2)^{m-i}}{((1+\delta)b_t)^{i}}.
\end{aligned}
\end{equation}
Now, similarly as above, we can bound the cardinality of $\mathcal D(m,q)$ by $\binom q 2 (\binom q 2 -1)^{m-1}$. Hence, by $b_t\geq 1$, we obtain by exchanging the sum on $i$ and $m$ that 
\begin{align}
	&\sup_X \Phi_{s,t,q}^{\leq 2}(X) \leq b_t \frac{ \binom{q}{2}}{\binom{q}{2}-1} \sum_{i=0}^\infty (1+\varepsilon_N^{\hyperlink{symbol: star}{\diamondsuit}})^{i} (C_{\delta,q})^i  \left({\bar{\sigma}_N^{2}(1+\delta) }b_t\right)^i \label{eq:sumoni3}\\
          & \hspace{1.5cm} \times \sum_{m=i}^\infty (1+\varepsilon_N^{\hyperlink{symbol: star}{\diamondsuit}})^{m-i}(C_{\delta,q})^{m-i} \frac{\lambda_{s,t,N}^{2(m-i)}}{(m-i)!},\nonumber
\end{align}
with $C_{\delta,q} = (1+2\delta^{-1}) (\binom q 2 - 1)$. 
If $\alpha_1<1$ as in the statement of the lemma, the geometric sum  in \eqref{eq:sumoni3} converges, yielding the result.
\end{proof}

\subsection{Proof of Proposition \ref{prop:mainProp}--\ref{item:(ii)} 
} 
\label{sec:proofOfQuasi(ii)}
 The proof goes by combining the estimates of Lemma \ref{lem:largeq_QC} and Lemma \ref{lem:smallq_QC}. 
In order to optimize in $\delta>0$, the error terms in \eqref{eq:finalBound02} and \eqref{eq:errorTermOptimizing} under the constraint $\alpha_0 < 1$ and $\alpha_1<1$ of \eqref{eq:conditionAlpha02} and \eqref{eq:defalpha1}, we need to consider different sets of values of $q$. {(We mention that $\delta$ will be taken large for Lemma \ref{lem:largeq_QC} and small for Lemma \ref{lem:smallq_QC}.)}

We fix an arbitrary constant $\varepsilon\in (0,1/4)$. We first consider $q\in \mathbb{N}_{\geq 3}$ such that $q^2 \leq \varepsilon/(b_t \bar \sigma_N^2)$.    We  
  take $\delta = \varepsilon (q^2 b_t \bar \sigma_N^2)^{-1}\geq 1$ in Lemma \ref{lem:smallq_QC} to obtain  for $N$ large enough, 
\[\alpha_1 \leq  2 (1+\delta) q^2 b_t \bar{\sigma}_N^2 \leq 4 \varepsilon<1,
\]

This entails
 \begin{equation*}
 \label{eq:2ndOptibisbis}
 \begin{aligned}
    &\sup_X \Phi_{s,t,q}^{\leq 2}(X) \leq     \frac{(1-4\varepsilon )^{-1}\binom q 2}{\binom q 2 - 1}{b_t} \exp\left\{\lambda_{s,t,N}^2 \left(\binom q 2 - 1\right) (1+2\varepsilon^{-1} q^2 b_t \bar \sigma_N^2))(1+\varepsilon_N^{\hyperlink{symbol: star}{\diamondsuit}}) \right\}.
\end{aligned}
\end{equation*}

Next, consider $q^2\geq (32\alpha/(1-\alpha))^{8}$. We now appeal to Lemma \ref{lem:largeq_QC}, where we set $L:=\lceil \sqrt q\rceil$ and $\delta := q^{-1/8}$. We find $(1+2\delta^{-1})^{1/L} \leq 1+4/(\delta L) \leq  1+8q^{-3/8}$,
\[
C_{L,\delta,q} = \binom q 2 \Big(1+4\Big((1+2\delta^{-1})^{{2}}\frac{L}{q-1}\Big)\Big) \leq { \left(\binom{q}{2}-1\right)}\left(1+ 8 q^{-1/4}\right),
\]
and with $\alpha_0$ from \eqref{eq:conditionAlpha02},
\begin{align*}
\alpha_0 =  b_t \bar \sigma_N^2 (1+\delta) C_{L,\delta,q}(1+2\delta^{-1})^{1/L}  \leq \binom q 2 b_t \bar \sigma_N^2 (1+16q^{-1/4})\leq \frac{1+\alpha}{2}<1,
\end{align*}
where the second inequality holds by \eqref{eq:q,T-condition} and $q^2 \geq  (32\alpha/(1-\alpha))^{8}$. Hence, \eqref{eq:finalBound02} yields
\begin{equation*}
\begin{aligned}
	\sup_X \Phi_{s,t,q}^{\leq 2}(X)\leq \frac{2q^{1/4}}{1-\alpha}  b_t
 \exp\left\{ \lambda_{s,t,N}^2\left(\binom q 2 -1\right)(1+\varepsilon_N^{\hyperlink{symbol: star}{\diamondsuit}}) (1+16q^{-1/4})  \right\}.
 \end{aligned}
\end{equation*}

\subsection{Large $q$ estimate for $p_0 > 2$ \eqref{eq:sub-criticality}}
We set,
\begin{equation} \label{eq:CLQ}
	C_{L,L_0,\delta,q} {:= C_{L,L_0,\delta,q} (c,p_0)} := \binom{q}{2} + 2(1+2\delta^{-1})^2\left( 2p_0 Lq+ c L_0^{-1} q^2 + cL_0q \right).
\end{equation}
\begin{lem} \label{lem:largeqSC}
Assume \eqref{eq:sub-criticality}.  Let $c_0$ from Theorem \ref{thm:sumOverXYv}, and {$\varepsilon_N^{\hyperlink{symbol: star}{\diamondsuit}}$ from Lemma \ref{eq:key_lemma}}. {There exists $c=c(\alpha,p_0,c_0)>0$, such that,} whenever
    \begin{equation} \label{eq:conditionAlpha0}
	\alpha_2:=
 C_{L,L_0,\delta,q} \left(1+2 \delta^{-1}\right)^{1/L} b_t \bar{\sigma}_N^{2}(1+\delta) (1+\varepsilon_N^{\hyperlink{symbol: star}{\diamondsuit}})< 1,
\end{equation}
it holds for $s\leq t$, {under \eqref{eq:q,T-condition},} 
\begin{equation} \label{eq:finalBound0}
	\sup_X \Phi_{s,t,q}^{\leq p_0}(X)\leq  \frac{2c_0(1+2\delta^{-1})^{{2}}}{1-\alpha_2} \exp\left\{ (1+\varepsilon_N^{\hyperlink{symbol: star}{\diamondsuit}}) \left(1+{2}\delta^{-1}\right)^{1/L} C_{L,L_0,\delta,q} \lambda_{s,t,N}^2 \right\}.
\end{equation}
\end{lem}
\begin{proof}
 By Proposition \ref{prop:Amn1} and Lemma \ref{lem:estimate_cim}, we obtain that
\begin{align*}
	&\sup_X \Phi_{s,t,q}^{\leq p_0}(X)  \leq 1+ c_0\sum_{m=1}^\infty (1+\varepsilon_N^{\hyperlink{symbol: star}{\diamondsuit}})^{m}  \sum_{ \mathbf{J}\in \mathcal D(m,q,p_0)}  \prod_{r=1}^m \bar{\sigma}_N^{2 |J_r|}  \\
  &\left\{ \left(\frac{c_0}{L_0}\right)^{\epsilon_r( \mathbf{J})} {c_0^{(1-\epsilon_r( \mathbf{J}))}} \right\}^{\mathbf{1}_{|J_r|\geq 2}} \sum_{i=0}^{m} \frac{ (1+\delta)^i (1+2 \delta^{-1})^{n_{\mathrm{bad}}(\mathbf J)}}{(m-i)!}
\frac{(\lambda_{s,t,N}^2/\bar{\sigma}_N^2)^{m-i}}{\left(1-\bar{\sigma}_N^2\log t \right)^{i}}.
\end{align*}
Let $\epsilon'_r( \mathbf{J}):=\mathbf{1}\{r \text{ is a small jump {index}}\}$ {(cf.\ Definition \ref{def:maindef})}  and for all $\epsilon' \in \{0,1\}^m$, define $n_{\mathrm{small}}(\epsilon'):=\sum_{r=1}^m \epsilon'_r$ and $\mathcal D(m,q,p_0,\epsilon') := \{ \mathbf{J}\in \mathcal D(m,q,p_0) : \epsilon_r'( \mathbf{J}) = \epsilon_r'\}$. 
Decomposing over the possible values of $\epsilon_r'( \mathbf{J})$, Lemma \ref{lem:gamma_n(I)} entails that
\begin{align}
	\sup_X \Phi_{s,t,q}^{\leq p_0}(X) & \leq 1+ c_0\sum_{m=1}^\infty  (1+\varepsilon_N^{\hyperlink{symbol: star}{\diamondsuit}})^{m}  D_m
	\sum_{n=0}^m \sum_{\substack{\epsilon' \in \{0,1\}^m \\ 
 n_{\mathrm{small}}(\epsilon') = n}}  
\left(1+2 \delta^{-1}\right)^{2n+\frac{m}{L}+2} B_{\epsilon'},\label{eq:1gyoume}
\end{align}
where 
\begin{equation} \label{eq:defDm} 
D_m:= \sum_{i=0}^m  \left(
{(1+\delta)b_t}\right)^i \frac{(\lambda_{s,t,N}^2)^{m-i} \bar \sigma_N^{-2(m-i)}}{(m-i)!},
\end{equation}
and
\begin{equation}
B_{\epsilon'} := \sum_{ \mathbf{J} \in \mathcal D(m,q,p_0,\epsilon')} \prod_{r=1}^m \bar{\sigma}_N^{2 |J_r|} \left( \left(\frac{c_0}{L_0}\right)^{\epsilon_r( \mathbf{J})}{c_0^{(1-\epsilon_r( \mathbf{J}))}}\right)^{\mathbf{1}_{|J_r|\geq 2}}. \label{eq:2gyoume}
\end{equation}
We first estimate $B_{\epsilon'}$ for any given $\epsilon'\in \{0,1\}^m$.  
{For $\mathbf J = (J_1,\dots,J_m)$, write $\mathbf{J}' = (J_1,\dots,J_{m-1})$ and let $\bar{\epsilon}':= (\epsilon_1',\dots,\epsilon_{m-1}')$. Note that 
\begin{align*}\mathcal D(m,q,p_0,\epsilon') &= \{\mathbf J = (J_1,\dots,J_m), J_r \subset \mathcal C_q,  :\mathbf J'\in \mathcal D(m-1,q,p_0,\bar{\epsilon}'),\\
& \qquad |\bar J_m|\leq p_0,\bar J_{m-1}\nsupseteq \bar J_m, \epsilon'_m((\mathbf J',J_m))= \epsilon'_m\}.
\end{align*}
Therefore, for a given $\mathbf J'\in \mathcal D(m-1,q,p_0,\bar{\epsilon}')$, we can estimate the sum over $J_{m}$ in \eqref{eq:2gyoume} by}
\begin{equation}
	\begin{aligned}
		\label{eq:SumOverJm}
		\Delta_m & := \sum_{\substack{J_m\subset \mathcal C_q: \bar  J_m\nsubseteq  \bar{J}_{m-1},  \\ {|\bar J_m| \leq p_0}  
        }
            } 
        \left(\left(\frac{c_0}{{L_0}}\right)^{\epsilon_m(  \mathbf{J})} {c_0^{1-\epsilon_m( \mathbf{J})}} \right)^{\mathbf{1}_{|J_m|\geq 2}} \bar{\sigma}_N^{2|J_m|}\\
		         & \leq \bar{\sigma}_N^2 |\mathcal J^{pair}_{\epsilon'_m,  \mathbf{J}'} | + {c_0}\sum_{k=2}^{p_0} \bar{\sigma}_N^{2k}| \mathcal  J^{k,1}_{  \mathbf{J}'}| + \frac{c_0}{L_0} \sum_{k=2}^{p_0} \bar{\sigma}_N^{2k}| \mathcal  J^{k,2}_{  \mathbf{J}'}| \\
		         & =: \Delta_m^0 + \Delta_m^1 + \Delta_m^2,
	\end{aligned}
\end{equation}
where we have set
\al{
\mathcal J^{pair}_{\epsilon'_m,  \mathbf{J}'}&:= \{{J_m=}(i_m^1,i_m^2)\in \mathcal C_q: \{i_m^1,i_m^2\} \nsubseteq \bar{J}_{m-1},
{\epsilon'_m((\mathbf J',J_m))}= \epsilon'_m\},\\
\mathcal J^{k,1}_{  \mathbf{J}'}&:= \{J_m\subset \mathcal C_q: |J_m|=k,\bar  J_m\nsubseteq  \bar{J}_{m-1}, {\epsilon_m((\mathbf J',J_m))=0} 
\},\\
\mathcal J^{k,2}_{  \mathbf{J}'}&:= \{J_m\subset \mathcal C_q: |J_m|=k,\bar  J_m\nsubseteq  \bar{J}_{m-1},{\epsilon_m((\mathbf J',J_m))=1}
\}.
} 
We start by estimating $\Delta_m^0$, which accounts for the main contribution in $\Delta_m$.
If $\epsilon'_m=0$, we simply bound $|\mathcal J^{pair}_{{\epsilon'_m,  \mathbf{J}'}}|$ by $|\mathcal C_q| = \binom{q}{2}$. 
If $\epsilon_m'=1$, then $m$ is a small jump index, so either $(m,i_m^1)$ or $(m,i_m^2)$ is a small jump, which by definition implies that $i_m^1$ or $i_m^2$ belongs to $\bar J_{m-1}\cup\dots\cup  \bar J_{m-L}$. Therefore, we can bound $|\mathcal J^{pair}_{{\epsilon',  \mathbf{J}'}}|$ by $2p_0 L q$  using  that $|\bar{J}_r|\leq p_0$. Hence, we have
\begin{equation} \label{eq:Delta_m^0}
	\Delta_m^0 \leq \bar{ \sigma}_N^2  \binom{q}{2} ^{1-\epsilon_m'}
	\left( 2p_0 L q \right)^{\epsilon_m'}.
\end{equation}
To estimate $\Delta_m^2$, {we use that $\bar{\sigma}_N^2 \binom q 2 \leq \alpha < 1$ under \eqref{eq:q,T-condition} (notice that $b_t\geq 1$)}, so that
\begin{equation} \label{eq:Delta_m^2}
	\Delta_m^2 \leq \frac{c_0}{L_0}  \sum_{k=2}^{p_{0}} \bar{\sigma}_N^{2k}
	{\binom q 2}^k  \leq \frac{c}{L_0}\bar{\sigma}_N^2 q^2,
\end{equation}
for some constant {$c=c(\alpha,p_0,c_0) > 0$.} 
We turn to $\Delta_m^1$. Since $(m,i_m^\star)$ is a ${L_0}$-small jump, similarly as above, we have $i_m^\star
	\in \bar J_{m-1}\cup\dots\cup \bar J_{m-{L_0}}$, so that  {$|\mathcal J^{k,1}_{\mathbf{J}'}| \leq  p_0{L_0} q {\binom q 2}^{k-1}$},
thus
\begin{equation} \label{eq:Delta_m^1}
	\Delta_m^1 \leq   {c_0}\sum_{k=2}^{p_{0}} \bar{\sigma}_N^{2k} {  p_0{L_0} q {\binom q 2}^{k-1}} \leq {c} L_0 q \bar{\sigma}_N^2.
\end{equation}
Putting \eqref{eq:Delta_m^0}, \eqref{eq:Delta_m^2} and \eqref{eq:Delta_m^1} together entails:
\al{
	\Delta_m 
 &\leq \bar \sigma_N^2 \left(\binom{q}{2} + c L_0^{-1} q^2 + {c} L_0 q\right)^{1-\epsilon_m'}\left(2 p_0 L q+ c L_0^{-1} q^2 +{c}L_0q\right)^{\epsilon_m'}.
}
Iterating this procedure, 
 it follows that $B_{\epsilon'}$ in \eqref{eq:2gyoume} satisfies 
\begin{equation} \label{eq:boundOn2gyoume}
	B_{\epsilon'}\leq \bar \sigma_N^{2m}  \left( \binom{q}{2} + c L_0^{-1} q^2+ {c} L_0 q\right)^{m-n} \left( 2p_0 Lq + c L_0^{-1} q^2+ {c} L_0 q\right)^{n}.
\end{equation}
{with $n=n_{\mathrm{small}}(\epsilon')$.}
We now plug \eqref{eq:boundOn2gyoume} in \eqref{eq:1gyoume}, then sum on $\epsilon'$ in \eqref{eq:1gyoume}, which gives a $\binom{m}{n}$ factor as \(\#\{\epsilon' \in \{0,1\}^m|~\sum_{i=1}^m \epsilon'_i =n\}=\binom{m}{n}\). We can then apply the binomial formula to the sum in $n$ in   \eqref{eq:1gyoume} and get:
\begin{equation} \label{eq:afterLemmacik}
	\begin{aligned}
\sup_X \Phi_{s,t,q}^{\leq p_0}(X) \leq & 1+  c_0(1+2\delta^{-1})^{{2}}\sum_{m=1}^\infty (1+\varepsilon_N^{\hyperlink{symbol: star}{\diamondsuit}})^{m} D_m (1+2 \delta^{-1})^{m/L}\bar{\sigma}_N^{2m} (C_{L,L_0,\delta,q})^{m},       
	\end{aligned}
\end{equation}
with $C_{L,L_0,\delta,q}$ as in \eqref{eq:CLQ}. 
 By exchanging the sum in $m$ in \eqref{eq:afterLemmacik} and the sum on $i$ in the definition of $D_m$ (see \eqref{eq:defDm}), we obtain
\begin{align}
	&\sup_X \Phi_{s,t,q}^{\leq p_0}(X)\nonumber\\
 & \leq  2 c_0 (1+2\delta^{-1})^{{2}}  \sum_{i=0}^\infty (1+\varepsilon_N^{\hyperlink{symbol: star}{\diamondsuit}})^{i} (C_{L,L_0,\delta,q})^i \left(1+2 \delta^{-1}\right)^{i/L} \left(b_t{\bar{\sigma}_N^{2}(1+\delta) }\right)^i \label{eq:sumoni}\\
          & \hspace{1.5cm} \times \sum_{m=i}^\infty (1+\varepsilon_N^{\hyperlink{symbol: star}{\diamondsuit}})^{m-i}(C_{L,L_0,\delta,q})^{m-i} \left(1+2 \delta^{-1}\right)^{(m-i)/L} \frac{\lambda_{s,t,N}^{2(m-i)}}{(m-i)!},\label{eq:sumonm}
\end{align}
where the sum on $m$ in \eqref{eq:sumonm} equals the exponential in the right-hand side of \eqref{eq:finalBound0}
below, and the remaining sum on $i$ in \eqref{eq:sumoni} is geometric, entailing the lemma, given that \eqref{eq:conditionAlpha0} implies the convergence of the geometric sum in \eqref{eq:sumoni}. 
\end{proof}

\subsection{Proof of Proposition \ref{prop:mainProp}--\ref{item:(i)} \eqref{eq:sub-criticality}} 
\label{sec:proofMainProp}
By using Lemma \ref{lem:largeqSC} with $L_0:=L:=\lceil \sqrt{q}\rceil$ and $\delta := q^{-1/8}$, for $q^2$ large enough depending on $\alpha,p_0,c$, the proof goes as in the large $q$ situation of Section \ref{sec:proofOfQuasi(ii)} (cf.\@  case $q^2 \geq  (32\alpha/(1-\alpha))^{8}$).

\section{Removing triple+ in the lower $q$ region}
\label{sec:quasiCritRemoveTriple+}
{This section assumes \eqref{eq:sub-criticality} or \eqref{eq:defNearCrit} and is dedicated to comparing $\Psi_{s,t,q}$ with $\Phi^{\leq 2}_{s,t,q}$} when $q^2$ of the order $\varepsilon_0 \log N$ for $\varepsilon_0=\varepsilon_0(\alpha)$ taken small enough. This condition will only be used at the very end of the argument, so unless stated otherwise, the statements of this section do not assume it.

We first explain how we deal with a general distribution $\mu$ of the weights $\omega(i,x)$. 

 \subsection{Comparison of general weights to Gaussian weights}\label{section: general weights in QC}  The proof of the next lemma may be skipped on first reading. 
\begin{lem} \label{lem:GeneralWeights} Assume \eqref{eq:sub-criticality} or \eqref{eq:defNearCrit}. There exist $c=c(\mu)>0$ and $\varepsilon_0=\varepsilon_0(\mu)$ such that for all $q^2 \sigma_N^2 \leq \varepsilon_0$, for all $s\leq t\leq N$, we have  $\Psi_{s,t,q} \leq \overline{\Phi}_{s,t,q}$, where $\overline{\Phi}_{s,t,q}$ is defined as the right-hand side of 
    {\eqref{eq:ChaosOnqMoment}}, 
    but with $\sigma_N^{2|K_r|}$ replaced by $\sigma_N^{2,K_r}:=(\sigma_N)^{2|K_r|-\mathbf{1}_{A_{K_r}}}
    $, where $A_K:=\{\exists (i,j),(k,l)\in K: \{i,j\}\cap \{k,l\}\neq \emptyset\}$.
\end{lem}
\begin{proof}
{Recall $N_n(x)$ defined above \eqref{eq:defPsikq} and $\mathcal C_{p_r}$ above \eqref{eq:defPsiGauss}. Fix a realization $(S_n^1,\dots,S_n^q)$ and suppose} that $x_1,...,x_k\in \Z^2$ and $p_1,...,p_k\geq 2$ are such that $N_n(x_r)=p_r$ with $\sum_{r\leq k} p_r = q$. 
Let $\mathcal I_r = \{i\in [q]: S_n^i=x_r\}$. By Lemma~\ref{lem:estimateLambda},   for $d_{p_r}=\mathbf{1}_{p_r>2} c(\mu) \beta_N p_r$, with $c(\mu)>1$,
\[
\Lambda_{p_r}(\beta_N) \leq (1+d_{p_r}) \Lambda_{2}(\beta_N) \binom {p_r}{2}  = (1+d_{p_r}) \Lambda_{2}(\beta_N) \sum_{(i,j)\in \mathcal{C}_{\mathcal{I}_r}} 1.
\]
and hence, 
\begin{equation} \label{eq:boundELAMBDA}
e^{\sum_{z\in \mathbb Z^2} \Lambda_{N_n(z)}(\beta_N)} \leq \prod_{r=1}^k \prod_{(i,j)\in \mathcal C_{\mathcal I_r}} e^{\Lambda_{2}(\beta_N) (1+d_{p_r}) 
}.
\end{equation}
Up to multiplying $d_{p_r}$ by a constant, this is less than
\begin{align*}
    \prod_{r=1}^k \prod_{(i,j)\in \mathcal C_{\mathcal I_r}} \left\{1 + \sigma_N^2 (1+d_{p_r}) 
    \right\} =  \prod_{r=1}^k \left(1+ \sum_{\emptyset\neq A_r\subset \mathcal C_{\mathcal I_r}} (\sigma_N^2 (1+d_{p_r}))^{|A_r|} 
    \right).
\end{align*}
We write the term inside the parenthesis as $1+D^1_r+ D^2_r$, where
\begin{gather*}
D^1_r = \sum_{(i,j)\in \mathcal C_{\mathcal I_r}} \sigma_N^2,  
\quad D^2_r = d_{p_r}D_r^1 + \sum_{A_r\subset \mathcal C_{\mathcal I_r}, |A_r|>1} (\sigma_N^2 (1+d_{p_r}))^{|A_r|}. 
\end{gather*}
Suppose $\varepsilon_0 = \varepsilon_0(\mu)\in (0,1/2)$ small enough and $q^2 \sigma_N^2\leq \varepsilon_0$ so that $(1+c(\mu)^2q^2)\beta_N^2 < 1/4$.
Then, if $p_r>2$, we have 
\(
D^2_r 
\leq \sigma_N^3 p_r^3\) as for all $p\leq q$, 
\[\sum_{\ell \ge 2} \binom{\binom{p}{2}}{\ell} \bigl( \sigma_N^2(1+d_p) \bigr)^\ell 
\le  \frac{1}{2}p^4 \sigma_N^4.\]
Therefore, by \eqref{eq:boundELAMBDA},
\(e^{\sum_{z\in \mathbb Z^2} \Lambda_{N_n(z)}(\beta_N)} \leq F(1),
\)
where
\[
F(t) = \prod_{r=1}^{k} (1 + D_r^1+ t\sigma_N^3 p_r^3 \mathbf{1}_{p_r>2} \Bigr), \quad t \in [0,1].
\] 
We have $F(0)  = \prod_{r=1}^k (1+\sum_{(i,j)\in \mathcal C_{\mathcal I_r}} \sigma_N^2)$ so that
\begin{align*}
F(0)  = 1+   \sum_{\substack{A_1,\dots,A_k\\ A_r \subset \mathcal{C}{p_r} \sum_{r=1}^k |A_r| \geq 1}} 
\sigma_N^{2\sum_r |A_r|}\leq  1+ \sum_{K\subset \mathcal C_q, |K|\geq 1} \sigma_N^{2|K|} \prod_{(i,j)\in K}\mathbf{1}_{S_n^i=S_n^j}.
\end{align*}
Besides, differentiating $ F(t) $ with respect to $ t $ gives
$$ F'(t) = \sum_{r=1}^{k} \sigma_N^3 p_r^3 \mathbf{1}_{p_r>2} \prod_{\substack{s=1 \\ s \neq r}}^{k} 
\Bigl( 1 + \sigma_N^2 p_s^2 + t\, \sigma_N^3 p_s^3\mathbf{1}_{p_r>2} \Bigr).$$
We use the inequality $ 1 + a \leq e^a $ to obtain
$$ \prod_{\substack{s=1 \\ s\neq r}}^{k} \left(1 + \sigma_N^2 p_s^2 + \xi\sigma_N^3 p_s^3 \right) \leq \exp\left( \sum_{\substack{s=1 \\ s\neq r}}^{k} (\sigma_N^2 p_s^2 + \sigma^3 p_s^3) \right), $$
Note that $ \sum_{s=1}^{k} p_s^2 \leq q\sum_{s=1}^k p_s = q^2$ and $\sum_{r=1}^{k} p_r^3 \leq q^3$. Since we assumed $q^2\sigma_N^2 < 1/2$, we obtain
that $\sup_{t\in [0,1]} F'(t)\leq e \sum_{r=1}^{k} \sigma_N^3 p_r^3 \mathbf{1}_{p_r>2} $.
Furthermore, for each $p_r>2$, one can construct at least $p_r^3$ couples of $\mathcal C_{\mathcal I_r}$ by choosing $3$ particles among $\mathcal I_r$ and choosing arbitrarily two distinct couples out of them. Hence,
\[\sum_{r=1}^k \sigma_N^3 p_r^3 \mathbf{1}_{p_r>2} \leq \sigma_N^3 \sum_{r=1}^k \sum_{\substack{K=\{(i,j),(k,l)\}\subset \mathcal C_{\mathcal I_r}\\ \{i,j\}\cap \{k,l\}\neq \emptyset}} \prod_{(i,j)\in K}\mathbf{1}_{S_n^i=S_n^j}\]
which is less than
\[
 \sigma_N^3 \sum_{\substack{K=\{(i,j),(k,l)\}\\ \{i,j\}\cap \{k,l\}\neq \emptyset}} \prod_{(i,j)\in K}\mathbf{1}_{S_n^i=S_n^j} \leq   \sum_{K\subset \mathcal C_q,|K|> 1} \sigma_N^{2|K|} \sigma_N^{-\mathbf{1}_{A_K}} \prod_{(i,j)\in K}\mathbf{1}_{S_n^i=S_n^j},
\]
where $A_K:=\{\exists (i,j),(k,l)\in K: \{i,j\}\cap \{k,l\}\neq \emptyset\}$.
\end{proof}
\begin{lem} \label{lem:yokkata}
Assume \eqref{eq:sub-criticality} or \eqref{eq:defNearCrit}.  There exists  $\varepsilon=\varepsilon(\mu)\in (0,1)$ such that,  for all $q^2\leq \varepsilon \log N$, 
\begin{equation}
    \label{eq:q^3} \sum_{K\subset \mathcal C_{q},|K|>1} \sigma_N^{2, K} \leq 8  q^3 \sigma_N^3.
\end{equation}
\end{lem}
\begin{proof}
For the sum over $K$ with $\mathbf{1}_{A_K}=0$, we have $$\sum_{K\subset \mathcal C_{q},|K|>1} \sigma_N^{2|K|} \leq  \sum_{k>1} q^{2k} \sigma_N^{2k} 2^{k}\leq 8 q^4 \sigma_N^4.$$
When $\mathbf{1}_{A_K}=1$, one particle must appear in two different pairs, thus 
\al{ 
\sum_{K\subset \mathcal C_{q},\mathbf{1}_{A_K}=1} \sigma_N^{2|K|-1} \leq  \sum_{k>1} q^{2k-1} \sigma_N^{2k-1} 2^{k}\leq 8 q^3 \sigma_N^3.
}
\end{proof}

\subsection{Presentation of the argument.} The strategy to remove triple+ intersections for $q^2 = \varepsilon\log N$ draws inspiration from \cite[Section 7]{CaSuZyCrit18}, although in our setup some important complications appear when $q^2$ reaches the order $ \varepsilon \log N$. For such $q$, the Markov trick that is used in \cite{CaSuZyCrit18} to control the convergence of some geometric series does not seem to be sufficient. Instead, to control these sums, we rely on an argument involving long jumps, similarly to Section \ref{sec:inductionSchemeSubcrit}. 

Although the idea that long jumps should dominate for combinatorial reasons is shared with Section \ref{sec:inductionSchemeSubcrit}, the technical analysis turns out to differ significantly. As for example, in Section \ref{sec:inductionSchemeSubcrit}, the nuisance terms come from the application of the integration by part, while here they constitute the "good" part of the estimates. More detail is given below.

Let $s<t$. By 
Lemma \ref{lem:GeneralWeights} 
and \eqref{eq:Phi2}, we can write: 
\begin{equation} \label{eq:dec>2<2}
\Psi_{s,t,q}(X_0) \leq \overline{\Phi}_{s,t,q} =  \Phi_{s,t,q}^{\leq 2} + \overline{\Phi}^{>2}_{s,t,q}(X_0),
\end{equation}
where
\begin{equation*} \label{eq:ChaosOnqMomentbis0}
\begin{split}
    	\overline{\Phi}^{>2}_{s,t,q}(X) &:=  \sum_{m=1}^\infty  \sum_{\substack{K_1,\dots,  K_m \in  2^{\mathcal C_q}\setminus\{\emptyset\}\\ \exists r,\,|K_r|\geq 2}} \sum_{s\leq n_1<\dots<n_m\leq t}   {\rm E}_{X}^{\otimes q}\left[\prod_{r=1}^m\sigma_N^{2,K_r} \prod_{(i,j)\in K_r} \mathbf{1}_{S_{n_r}^i=S_{n_r}^j}\right].
     \end{split}
\end{equation*}
(Indeed, when all $|K_r|=1$, we find  $\Phi_{s,t,q}^{\leq 2}$ by decomposing into successive exchanges of disjoint pairs as in the proof of Proposition \ref{prop:expansion}.) The following lemma is an extension to $q\geq 3$ of \cite[Eq.\@ (7.2)]{CaSuZyCrit18}.

{Let $\mathcal D^{> 2}(m+1,q):=\{\mathbf J\in \mathcal D(m+1,q,\infty): \forall r\leq m,\, |J_r|=1 \text{ and } |J_{m+1}|\geq 2\}$. For any element of $\mathcal D^{> 2}(m+1,q)$, we write $J_r=\{(i_r,j_r)\}$ for $r\leq m$.
\begin{lem} \label{lem:PhibarUpsPhi2}
Let $q\geq 3$. Then,
\begin{equation} \label{eq:upperboundPsiN}
	\sup_{X} \overline{\Phi}^{>2}_{1,t,q}(X) \leq \left(\sum_{k=1}^\infty \left(\sup_{X_0 \in (\mathbb Z^2)^q} \Upsilon_t(X_0)\right)^k\right) \sup_{X_0\in (\mathbb Z^2)^q} \Phi_{1,t,q}^{\leq 2}(X_0),
\end{equation}
where
\[
	\begin{aligned}
		 & \Upsilon_t(X_0):=\sum_{m=0}^\infty  \sum_{\substack{1\leq a_1\leq b_1 < \dots \leq b_m<a_{m+1} \leq t   \\ \mathbf J \in \mathcal D^{> 2}(m+1,q)}} \sum_{\mathbf X,\mathbf Y}  \sigma_N^{2,J_{m+1}}\times \\
		 & \prod_{r=1}^{m} \sigma_N^{2}  U_N(b_r-a_r,y_r^{i_r}-x_r^{i_r})  \mathbf{1}_{y_r^{i_r}=y_r^{j_r}} \prod_{r=1}^{m+1} \prod_{i\in \bar J_r}p_{a_r-b_{k_r(i)}}(x_r^i-y_{k_r(i)}^i) \prod_{(i,j)\in J_r} \mathbf{1}_{x_r^i=x_r^j} 
	\end{aligned}
\]
\end{lem}
}

\begin{proof} Set $H_{K,n}:=\sigma_N^{2,K} \prod_{(i,j)\in K} \mathbf{1}_{S_{n}^i=S_{n}^j}$. 
In the definition of $\overline{\Phi}^{>2}_{1,t,q}$, let $K_1,\ldots,K_m\in 2^{\mathcal C_q}\setminus\{\emptyset\}$ be such that $|K_r|\geq 2$  for some $r$. We collect and enumerate $r$ satisfying $|K_r|\geq 2$ as $\sigma_1<\sigma_2<\ldots<\sigma_k$, i.e.,  $\{\sigma_s|~s\in[
		k]\}=\{r\in \Iintv{1,m}|~|K_r|\geq 2\}$. By 
        assumption, we have $k\geq 1$ and 
\begin{align*}
	  \overline{\Phi}^{>2}_{s,t,q}(X)        & =\sum_{k=1}^\infty \sum_{m=1}^{\infty}  \sum_{\substack{K_1,\dots,  K_m \in  2^{\mathcal C_q}\setminus\{\emptyset\} \\|\{r\in \Iintv{1,m}|~|K_r|\geq 2\}|=k}} \sum_{s\leq n_1<\dots<n_m\leq t} {\rm E}_{X}^{\otimes q}\left[\prod_{r=1}^{m} H_{K_r,n_r}\right],
\end{align*}
so that $\overline{\Phi}^{>2}_{s,t,q}=\sum_{k=1}^\infty \Phi^{>2, k}_{s,t,q}$, where
\begin{align*}
 &\Phi^{>2, k}_{s,t,q}(X) :=\sum_{m=1}^{\infty} \sum_{1\leq \sigma_1<\cdots<\sigma_k\leq m} \sum_{\substack{K_1,\dots,  K_m \in  2^{\mathcal C_q}\setminus\{\emptyset\} \\ \{r\in \Iintv{1,m}|~|K_r|\geq 2\}=\{\sigma_r\}_{r=1}^k}} \sum_{s\leq n_1<\dots<n_m\leq t}  {\rm E}_{X}^{\otimes q}\left[\prod_{r=1}^{m} H_{K_r,n_r}\right].
\end{align*}
 Summing over $n_{\sigma_k+1},K_{\sigma_k +1},\ldots, n_m,K_m$  in $\Phi_{s,t,q}^{>2,k}$, we have
\begin{align*}
	 & \overline{\Phi}^{>2, k}_{s,t,q}(X_0)=     
  \sum_{1\leq \sigma_1<\cdots<\sigma_{k-1}\leq \sigma_k} \sum_{\substack{K_1,\dots,  K_{\sigma_k} \in  2^{\mathcal C_q}\setminus\{\emptyset\} \\\{r\in \Iintv{1,\sigma_k}|~|K_r|\geq 2\}=\{\sigma_r\}_{r=1}^k}} \sum_{s\leq n_1<\dots<n_{\sigma_k}\leq t}   \\
  &\quad\times \sum_{m=\sigma_k}^\infty\sum_{\substack{K_{\sigma_{k}+1},\dots,  K_m \in  2^{\mathcal C_q}\setminus\{\emptyset\} \\|K_r|=1\,\forall r\in \Iintv{\sigma_{k}+1,m}}} \sum_{n_{\sigma_k}< n_{n_{\sigma_k}+1}<\dots<n_m\leq t}  {\rm E}_{X}^{\otimes q}\left[\prod_{r=1}^{m} H_{K_r,n_r}\right],
\end{align*}
which, by Markov's property at time $n_{\sigma_k}$, is less than
\begin{align*}
\sup_{X} \Phi_{1,t,q}^{\leq 2}(X) \sum_{1\leq \sigma_1<\cdots<\sigma_{k-1}\leq \sigma_k}  \sum_{\substack{K_1,\dots,  K_{\sigma_k} \in  2^{\mathcal C_q}\setminus\{\emptyset\} \\\{r\in \Iintv{1,\sigma_k}|~|K_r|\geq 2\}=\{\sigma_r\}_{r=1}^k}} \sum_{s\leq n_1<\dots<n_{\sigma_k}\leq t} \times {\rm E}_{X_0}^{\otimes q}\left[ \prod_{r=1}^{\sigma_k} H_{K_r,n_r} \right].
\end{align*}
Next, summing over $n_{\sigma_{k-1}+1},K_{\sigma_{k-1} +1},\ldots, n_{\sigma_k},K_{\sigma_k}$, we have
\begin{align*}
	 & \Phi^{>2, k}_{1,t,q}(X_0)\leq \sup_{X \in (\mathbb Z^2)^q} \Upsilon_t(X)  \sup_{X\in (\mathbb Z^2)^q} \Phi_{1,t,q}^{\leq 2}(X)\\
     &\times\sum_{1\leq \sigma_1<\cdots<\sigma_{k-1}}  \sum_{\substack{K_1,\dots,  K_{\sigma_{k-1}} \in  2^{\mathcal C_q}\setminus\{\emptyset\} \\\{r\in \Iintv{1,\sigma_{k-1}}|~|K_r|\geq 2\}=\{\sigma_r\}_{r=1}^{k-1}}} \sum_{s\leq n_1<\dots<n_{\sigma_{k-1}}\leq t} {\rm E}_{X_0}^{\otimes q}\left[\prod_{r=1}^{m}  H_{K_r,n_r}\right].
\end{align*}
We obtain the lemma by induction.
\end{proof}
Proposition \ref{prop:mainProp}-\ref{item:(ii)} deals with $\sup_{X_0\in (\mathbb Z^2)^q} \Phi^{\leq 2 }_{1,t,q}$ in \eqref{eq:upperboundPsiN}. We then prove the following.
\begin{prop}
\label{prop:UpsT} 
Assume \eqref{eq:sub-criticality} or \eqref{eq:defNearCrit}. {There exists $\varepsilon_0=\varepsilon_0(\alpha)\in (0,1)$, such that uniformly for $(q,t)$ satisfying \eqref{eq:q,T-condition} and $q^2\leq \varepsilon_0 \log N$,}
\begin{equation} \label{eq:supUps0}
\limsup_{N\to\infty} \sup_{X_0} \Upsilon_t(X_0) < 1.
\end{equation}
{Moreover,
\begin{equation} \label{eq:Phi>2phi}
    \Psi_{1,t,q}(X_0) \leq (1+\mathcal O(q \bar \sigma_N))\Phi_{1,t,q}^{\leq 2}. 
\end{equation}
}
\end{prop}
Section \ref{sec:proofUpsT} is dedicated to the proof of Proposition \ref{prop:UpsT}. The rest of this section and Section \ref{sec:anotherLongJumpInduction} gather preliminary results.

Similarly to Lemma \ref{lem:ChangeOfVariables}, we make the change of variables $u_r=a_r-b_{r-1}$, $v_r=b_r-a_r$ for $r\leq m+1$, and obtain:
\begin{equation} \label{eq:FirstBoundUps}
	\begin{aligned}
		 & \Upsilon_t(X_0)\leq \sum_{m=0}^\infty \sigma_N^{2m} \sum_{
		\mathbf J \in \mathcal D^{>2}(m+1,q)} \sum_{\mathbf X, \mathbf Y} \sum_{\substack{ u_{m+1},u_r \in \llbracket 1,t\rrbracket\\
  v_r\in \llbracket 0,t\rrbracket,r\leq m}} \sigma_N^{2,J_{m+1}} \mathbf{1}_{\sum_{i=1}^{m+1} u_i \leq t} \, \\
  &
         \prod_{r=1}^{m}    U_N(v_r,y_r^{i_r}-x_r^{i_r}) 
		\mathbf{1}_{y_r^{i_r}=y_r^{j_r}} \prod_{r=1}^{m+1}  \prod_{i\in \bar J_r} p_{w_r(i)}(x_r^i-y_{k_r(i)}^i)\prod_{(i,j)\in J_r} \mathbf{1}_{x_r^{i}=x_r^{j}},
	\end{aligned}
\end{equation}
where $w_r(i):=\sum_{k_r(i)+1}^r u_s + \sum_{k_r(i)+1}^{r-1} v_s$. For $u\in(0,\infty)$, let
\begin{equation}
	\label{eq-FF}
	F(u) := 
 \begin{cases}
     \frac{1}{u}  \frac{1}{1-\bar{\sigma}_N^2 \log u},\,&\text{if $u\geq 1$},\\
    1 ,\,&\text{if $u< 1$},
      \end{cases} \quad \text{with} \quad \bar{\sigma}_N^2 := \pi^{-1} \sigma_N^2,
\end{equation} which extends the definition of $F$ in \eqref{eq-F} to $u<1$. Recall the notation $\bar{u}_r$ from \eqref{eq:ubar}, which depends on 
the diagram $\mathbf{J}$.  The next lemma uses techniques from Theorem \ref{thm:sumOverXYv} and Proposition \ref{prop:UBFQcrit}.
\begin{lem} \label{lem:FF}
	We have:
	\begin{equation}\label{eq:InitialboundDigammaNew}
		\sup_{X_0} \Upsilon_t(X_0)\leq {\frac{2}{\pi^2}}  \sum_{m=0}^\infty  (1+\varepsilon_N^{\hyperlink{symbol: star}{\diamondsuit}})^{m} \bar{\sigma}_N^{2m} 
        \sum_{
			\mathbf J  \in \mathcal D^{>2}(m+1,q)}  \sigma_N^{2,J_{m+1}} A_t(\mathbf{J}),
	\end{equation}
	where, letting $\Delta_{t}^{m+1}:=\{(u_i)_{i=1}^{m+1}:~u_1,\dots,u_{m+1}\geq 0, \sum_{i=1}^{m+1} u_i\leq t\}$,
	\begin{equation} \label{eq:defAbfI}
		A_t(\mathbf{J}):=\int_{\Delta_{t}^{m+1}}   F(u_1) \left(\prod_{r=2}^{m+1}
		F(u_r+\bar u_{r})\right)   F(u_{m+1})
		 {\rm d} u_1 \cdots {\rm d} u_{m+1}.
	\end{equation}
\end{lem}
\begin{proof}
	 Summing over $X_{m+1}$ in \eqref{eq:FirstBoundUps} gives a factor
 $p^\star_{w_{m+1}(i_{m+1}^1)+w_{m+1}(i_{m+1}^2)}p^\star_{w_m(i_{m+1}^\star)}$ as in \eqref{eq:ResultSumOnX_m}. Since $|J_m|=1$, at least one particle $i\in \bar J_{m+1}$ satisfies $k_{m+1}(i)<m$, so that 	\eqref{eq:sumwmwm} with $m$ replaced by $m+1$ holds.
	We then simply bound $w_{m+1}(i^\star_{m+1})$ from below by $u_{m+1}$.
	Then, summing over $Y_m$ gives a factor $U_N(v_{m})$. Using that $p^\star$ is non-increasing, we are left with bounding a contribution of the form
	\[
		p^\star_{u_{m+1}}\sum_{v_{m}} U_{N}(v_m) p^\star_{v_{m}+2(u_{m+1}+\bar u_{m+1})}\leq \frac{2}{\pi^2} (1+\varepsilon_N^{\hyperlink{symbol: star}{\diamondsuit}})  F(u_{m+1})
		F(u_{m+1}+\bar u_{m+1}),
	\]
where the inequality holds by Lemma \ref{eq:key_lemma} and \eqref{eq:pnstar}.
	
Since $F$ is non-increasing by Lemma \ref{lem:F'}, we have $F(x) \geq F(u)$ for $u=\lceil x\rceil$ so we can replace the summation over $u_{m+1}$ by an integral. 	The rest of the proof goes as in Proposition \ref{prop:UBFQcrit} together with that $p^{\star}_{2u_1} \leq {\frac{1}{\pi}F(u_1)}$ for all $u_1\geq 1$.
\end{proof}

\subsection{Another long jump induction} \label{sec:anotherLongJumpInduction}
We adopt the same notations, including good and bad indices, as in Definition~\ref{def:maindef},  but with $m+1$ in place of $m$.
Introduce:
$$
g(x):=
\begin{cases}
    \frac{1}{\bar{\sigma}_N^2}\log\left(\frac{1}{1-\bar \sigma_N^2 \log x}\right),&\text{ if }x\in [1,N],\\
    0,&\text{ if $x\in [0,1]$},
    \end{cases}	
$$
which satisfies $g'(x)=F(x)$ for $x> 1$. For any function $h$, set
\begin{equation*}
	\mathcal S^{\rm fresh} h(v): =   \int_{0}^t
	F(u + v)  h(u) \mathbf{1}_{u+v \leq t}\, \mathrm{d}u,\, 
	\mathcal S^{\rm good} h(v) := \int_{0}^t F(u+v) h(u+v) \mathbf{1}_{u+v\leq t}\,\mathrm{d} u,
\end{equation*}
and $\mathcal S^{\rm small} h(v) := \mathcal S^{\rm fresh}h(v/2)$.
Finally, let $h^{(j)}(v): = F(v)\,g(v)^j$. 
\begin{lem} \label{lem:indReductionReloaded}
	For all $j\geq 0$ and $v\in {[}0,t]$,
	\begin{equation}\label{eq:indReductionReloaded}
		\begin{aligned}
			\mathcal S^{\rm fresh} h^{(j)}(v)\leq  \sum_{i=0}^{j+1} \frac{j!}{i!} (\overline{b}_t)^{j+1-i}   h^{(i)}(v),
		\end{aligned}
	\end{equation}
with $\overline{b}_t:=2 + b_t (1+\varepsilon_N^\diamond)$, {where $\varepsilon_N^\diamond$ is from Lemma \ref{lem:F'}},
  {and
\begin{equation}\label{eq:indReductionReloaded/2}
    \mathcal S^{\rm small} h^{(j)}(v)\leq  \frac{2h^{(j+1)}(v)}{j+1}  + \sum_{i=0}^{j} \frac{j!}{i!} (\overline{b}_t)^{j+1-i}   h^{(i)}(v).
\end{equation}
}

	Furthermore, for $v\in {[}0,t]$,
	\begin{equation} \label{eq:w+v}
		\mathcal S^{\rm good} h^{(j)}(v) \leq \sum_{i=0}^{j} \frac{j!}{i!} (\overline{b}_t)^{j+1-i} h^{(i)}(v).
	\end{equation}
\end{lem}
\begin{proof}
We first suppose $j\geq 1$. We use the notation $a\vee b:=\max(a,b)$.	By monotonicity of $F$, $g'(x)=F(x)$, and $g(x)=0$ for $x\leq 1$, 
\begin{equation} \label{eq:boundF1veev}
    \int_{0}^{{1\vee v}}
    F(u+v) F(u) g(u)^{j} \mathbf{1}_{u+v \leq t} \leq F(v) \int_{1}^{{1\vee v}} F(u) g(u)^j {\rm d} u = F(v) \frac{g(v)^{j+1}}{j+1}.
\end{equation}
Moreover, we have
\begin{equation}	\label{eq:shortcut}
\begin{aligned} 
 \int_{{1\vee v}}^t
		F(u+v) F(u)  g(u)^{j} \mathbf{1}_{u+v \leq t} \,du
  &\leq 
		\int_{{1\vee v}}^{t}
		F(u)^2  g(u)^{j} \mathbf{1}_{u+v \leq t} \,{\rm d} u\\
&		  \leq
		\overline{b}_t \int_{{1\vee v}}^{t} (-F'(u)) g(u)^j\, {\rm d}u, 
\end{aligned}
\end{equation}
	where we have used \eqref{eq:F'} in the last inequality.
	By integration by parts, we have:
	\begin{align*}
		\int_{{1\vee v}}^{t} (-F'(u)) g(u)^j {\rm d} u \leq  F(v) g(v)^j + j\int_{{1\vee v}}^{t} F(u)^2 g(u)^{j-1} {\rm d} u,
	\end{align*}
where  we have used that $g(v)=g(1)=0$ for $v<1$. 
The last integral corresponds to the one on the top right of  \eqref{eq:shortcut} with $j$ replaced by $j-1$. Therefore, we find \eqref{eq:indReductionReloaded} by repeating this procedure. 
 For $j=0$, by \eqref{eq:F'} again, \eqref{eq:indReductionReloaded} follows from 
\aln{
& \int_{0}^{{1\vee v}}
		F(u+v) F(u) \mathbf{1}_{u+v \leq t} {\rm d} u\leq F(v) \left(1+ \int_{1}^{{1\vee v}} F(u)  \mathrm{d} u\right) = F(v) (1+g(v)),  \label{eq:bt1}\\
  &   \int_{{1\vee v}}^{t}
		F(u+v) F(u) \mathbf{1}_{u+v \leq t}{\rm d} u\leq (\overline{b}_t-2) \int_{{1\vee v}}^{t} (-F'(u)) \, {\rm d} u\leq (\overline{b}_t-2) F(v). \label{eq:bt2}
}

{For \eqref{eq:indReductionReloaded/2}, we can proceed exactly as the previous case, except for \eqref{eq:boundF1veev} and \eqref{eq:bt1} where we use that $F(u+v/2)\leq 2F(v)$.}
 
 Regarding \eqref{eq:w+v} with $j\geq 1$ or $j=0,\,v\geq 1$, 
 we bound $\mathcal S^{\rm good} h^{(j)}(v)$ similarly to \eqref{eq:shortcut}, {that is by
 \begin{equation*}
\begin{aligned} 
 \int_{{0}}^t
		F(u+v)^2 g(u+v)^{j} \mathbf{1}_{u+v \leq t} \,\mathrm{d}u\leq
		\overline{b}_t \int_{{1\vee v}}^{t} (-F'(u)) g(u)^j\, {\rm d}u, 
\end{aligned}
\end{equation*}
}
so we recover the result as in the previous case. For $j=0$ and $v\in {[}0,1)$, by \eqref{eq:F'},
 \al{
 	\mathcal S^{\rm good} h^{(0)}(v) &\leq 1+ \int_{1}^{t} F(u)^2 {\rm d} u\leq 1- (\overline{b}_t-1) \int_{1}^{t} F'(u) {\rm d} u\leq \overline{b}_t =\overline{b}_t h^{(0)}(v).
 }
\end{proof}

For all $ \mathbf{J}\in \mathcal D^{>2}(m+1,q)$  and $r\leq m+1$, let $\psi(r)$ be defined as in \eqref{eq:defPsi} and $u_{(r)} := \sum_{i=\psi(r)}^{r} u_i$. For all $r\geq 2$ and {all non-negative measurable function $h(u_1,\dots,u_r)$, set
\[
	\mathcal S_r h (u_1,\dots,u_{r-1}) :=
	\int_{0}^t
	F(u_{r} + \bar u_{r})  h\left(u_1,\dots,u_r\right) \mathbf{1}_{\sum_{i=1}^{r} u_i \leq t} {\rm d} u_{r},
\]
and for all $j\geq 0$, define $\bar h^{(j)}_r(u_1,\dots,u_r):=h^{(j)}(u_{(r)})$. Given that $h^{(0)}_{m+1}(u_1,\dots,u_{m+1})=F(u_{m+1})$ as $\psi(m+1)=m+1$, we can write 
\begin{equation} \label{eq:formulaA_t}
A_t( \mathbf{J}) = (\mathcal S^{\rm fresh} \circ \mathcal S_{2}\circ {\dots} \circ \mathcal S_{m+1}\bar h^{(0)}_{m+1})({0}).
\end{equation}
}

\begin{lem} \label{lem:indReduction}
	For all $r\in \llbracket 2,m+1 \rrbracket$, $j\geq 1$ and $\sum_{i=1}^{r-1} u_i \leq t$ with $u_i\in {[0,t]}$,
	\begin{equation}\label{eq:indRed}
		\begin{aligned}
			\mathcal S_r \bar h^{(j)}_r(u_1,\dots,u_{r-1})  \leq \mu_{r}^m  \frac{h^{(j+1)}\left(u_{(r-1)}\right)}{j+1} + \sum_{i=0}^{j} \frac{j!}{i!} (\overline{b}_t)^{j+1-i} {h^{(i)}\left(u_{(r-1)}\right)},
		\end{aligned}
	\end{equation}
{(there is now a bar on $h$ on the LHS)}
where $\mu_r^m :=  \mu_r^{m} (\mathbf{J}) := \mathbf{1}_{\{r \text{ is bad}\}} + \mathbf{1}_{\{r \text{ is a small jump}\}}.$ 
\end{lem}
\begin{proof}
Similarly to the proof of Lemma \ref{lem:indReductionMain}, we obtain that when
$r$ is good,
\begin{align*}
    \mathcal S_r \bar h^{(j)}_r (u_1,\ldots,u_{r-1})\leq\int_0^t  F\left(u_{r}+u_{(r-1)}\right)  h^{(j)}\left(u_{r} + u_{(r-1)}\right)\mathbf{1}_{u_{r} + u_{(r-1)} \leq t}\,{\rm d} u_r,
\end{align*}
	where the integral equals $\mathcal S^{\rm good} h^{(j)}(u_{(r-1)})$, so that \eqref{eq:w+v} gives \eqref{eq:indRed};
when $r$ is fresh, 
\begin{equation*} 
		\mathcal S_r \bar h^{(j)}_r(u_1,\ldots,u_{r-1}) \leq \int_{0}^t  F\left(u_{r}+u_{(r-1)}\right)  h^{(j)}\left(u_{r}\right)\mathbf{1}_{u_{r} + u_{(r-1)} \leq t}\,{\rm d} u_r,
\end{equation*}
where the integral is equal to $\mathcal S^{\rm fresh} h^{(j)}(u_{(r-1)})$, so we obtain \eqref{eq:indRed} from \eqref{eq:indReductionReloaded}; 
 if $r$ is a small jump, $\mathcal S_r \bar h^{(j)}_r(u_1,\ldots,u_{r-1})\leq \mathcal S^{\rm small}(u_{(r-1)})$ so we conclude with \eqref{eq:indReductionReloaded/2}.
\end{proof}

Similarly to Proposition \ref{prop:fibo1inverseMain}, we obtain:
\begin{prop}
	\label{prop:fibo1inverse}
Assume \eqref{eq:sub-criticality} or \eqref{eq:defNearCrit}. Then, for all
	$m\geq 1$, $ \mathbf{J}\in \mathcal D(m,q)$,  $k\in \llbracket {0},m-1 \rrbracket$ and $\sum_{i=1}^{m+1-k} u_i \leq t$ with $u_i\in {[0,t]}$, 
	\begin{equation} \label{eq:lemmaInduction0}
		\begin{aligned}
			 & (\mathcal S_{m+1-k}\circ {\dots} \circ \mathcal S_{m+1}\bar h^{(0)}_{{m+1}} )(u_{1},\dots,u_{m-k})\leq \sum_{j=0}^{k{+1}} \frac{a_{j}^{k{+1}}}{j!}  (\overline{b}_t)^{k{+1}-j} h^{(j)}\left(u_{(m-k)}\right),
		\end{aligned}
	\end{equation}
	with $a_0^0:=1$,
	$a_{l}^{k+1} := 2 \gamma_{k}^{m+1} a_{l-1}^k + \sum_{j=l}^k a_{j}^k$ for $l \in \Iintv{0,k}$, 
  $a_{l}^k:=0$ for $l<0$ or $l>k$, where $\gamma_k^{m+1} := \mathbf{1}_{\{m+1-k \text{ is bad}\}}$  {for all $k< m$} and $\gamma_{m}^{m+1}=1$.

Furthermore, recalling $A_t( \mathbf{J})$ from \eqref{eq:defAbfI}, we have:
	\begin{equation} \label{eq:AI}
		A_t( \mathbf{J}) \leq a_{0}^{m+1} (\overline{b}_t)^{m+1}.
	\end{equation}
\end{prop}
\begin{proof}
	The estimate \eqref{eq:lemmaInduction0} is obtained by iterating Lemma \ref{lem:indReduction}. Indeed, {\eqref{eq:lemmaInduction0} holds for $k=0$ by Lemma \ref{lem:indReduction}.} let $1\leq k<m-1$ be such that \eqref{eq:lemmaInduction0} holds with $k-1$, i.e.,
 \al{
 (\mathcal S_{m+1-k}\circ {\dots} \circ \mathcal S_{m+1}\bar h^{(0)}_{m+1} )(u_{1},\dots,u_{m+1-k})\leq \sum_{j=0}^{k} \frac{a_{j}^{k}}{j!}  (\overline{b}_t)^{k-j} h^{(j)}\left(u_{(m+1-k)}\right).
 }
 Then by Lemma \ref{lem:indReduction}, the left-hand side of \eqref{eq:lemmaInduction0} is at most 
	\al{
  \sum_{j=0}^{k} \frac{a_{j}^{k-1}}{j!}  (\overline{b}_t)^{k-j} \left(\mu_{m-k}^m  \frac{h^{(j+1)}\left(u_{(m-k)}\right)}{j+1} +   \sum_{i=0}^{j} \frac{j!}{i!} (\overline{b}_t)^{j+1-i} {h^{(i)}\left(u_{(m-k)}\right)}\right).
	}
	Exchanging the sum on $j$ and $i$, this is equal to
	\al{
 \sum_{i=0}^{k+1} \frac{1}{i!} (\overline{b}_t)^{k+1-i} h^{(i)}(u_{(m-k)}) \left(\mu_{m-k}^m a_{i-1}^{k} + \sum_{j=i}^{k} a_j^{k} \right),
	}
	which gives \eqref{eq:lemmaInduction0} by  {$\mu_{m-k}^m\leq 2\gamma_{k+1}^m$.} 

Finally, by \eqref{eq:formulaA_t}, \eqref{eq:lemmaInduction0} and \eqref{eq:indReductionReloaded}, $A_t( \mathbf{J})$ is less than
\[
\sum_{j=0}^{m} \frac{a_{j}^{m}}{j!}  (\overline{b}_t)^{m-j} \left(\sum_{i=0}^{j+1} \frac{j!}{i!} (\overline{b}_t)^{j+1-i}   h^{(i)}({0})\right) = \sum_{i=0}^{m{+1}} \frac{a_{i}^{m{+1}}}{i!}  (\overline{b}_t)^{m+1-i} h^{(i)}\left({0}\right).
\]
(We have used that 
    $\gamma_{m}^{m+1} = 1$ by definition.) Given that $h^{(i)}({0}) = 0$ for all $i\geq 1$ and $h^{(0)}({0})=1$, \eqref{eq:AI} follows.
\end{proof}
We end with a last technical lemma.
 We set $n_k:=\sum_{i=0}^{k-1} \gamma_{i}^{m+1}$. 
\begin{lem} \label{lem:aik}
	Let $k\geq 0$ and $\delta\in (0,1/2)$. We have:
 \begin{enumerate}
     \item[(i)]  $a_{l}^k=0$ for all $l > n_k$;
     \item[(ii)] for all $l \leq n_k$:
\begin{equation}
		a_{l}^k \leq  
  {2^{n_k}}\left(1+2\delta \right)^k \delta^{-(n_k -l)}.
\end{equation}
  \end{enumerate}
\end{lem}
\begin{proof}
	We prove both (i) and (ii) by induction.
We start with (i). The claim is clear for $k=0$. 
 Suppose that (i) holds for some $k-1\geq 0$ and let $l> n_{k}$. By the induction hypothesis we have\ $a_{j}^{k-1}=0$ for any $j\geq l$, so that
\[
a_{l}^{k} =2 \gamma_{k-1}^{m+1} a_{l-1}^{k-1} +\sum_{j=l}^{k-1} a_{j}^{k-1} = 2 \gamma_{k-1}^{m+1} a_{l-1}^{k-1}.
\]
If 
$\gamma_{k-1}^{m+1}=0$, then $a_{l}^{k} = 0$, while if 
$\gamma_{k-1}^{m+1}=1$, then $n_k=n_{k-1}+1$ and thus $l-1  > n_{k-1}$, so that $a_{l-1}^{k-1}=0$, which implies $a_{l}^{k} =0$.

	Next we prove (ii). For $k=0$, $n_k=0$ and $a_0^0=1$ yield the claim. Now assume (ii) for some $k-1\geq 0$, i.e.,  for all $l \leq n_{k-1}$, 
	$a_{l}^{k-1} \leq 2^{n_{k-1}} \left(1+2\delta \right)^{k-1} \delta^{-(n_{k-1} -l)}.$ 
We first suppose $\gamma_{k-1}^{m+1}=0$, i.e., $n_{k}=n_{k-1}$. By (i), $a_{i}^{k-1}=0$ for $i>n_{k-1}$,  hence by the induction hypothesis,  for all $l\leq n_k$
\al{
		a_{l}^{k} =
  \sum_{i=l}^{n_{k-1}} a_{i}^{k-1} & \leq 
  {2^{n_{k-1}}}\sum_{i=l}^{n_{k-1}} \left(1+2\delta \right)^{k-1} \delta^{-(n_{k-1} -i)}
\leq {2^{n_k}}
(1+2\delta)^{k}\delta^{-{(n_{k}-l)}},
}
where we have used that $n_{k-1}=n_{k}$ and  $\sum_{i\geq 0} \delta^{i} = (1-\delta)^{-1} \leq 1+2\delta$ for $\delta \in (0,1/2)$.
	Moreover, if 
 $\gamma_{k-1}^{m+1}=1$,
 then $n_k = n_{k-1}+1$. By (i),  $a_{i}^{k-1}=0$ for $i>n_{k-1}$, so that for all $l\leq n_{k}$, by the induction hypothesis,
\begin{align*}
    a_{l}^{k} =2  a_{l-1}^{k-1} +
    \sum_{j=l}^{n_{k-1}} a_j^{k-1}& \leq {2^{n_{k-1}+1}}
    \sum_{j=l-1}^{n_{k-1}} \left(1+2\delta \right)^{k-1} \delta^{-(n_{k-1} -j)} 
    \\
    &\leq {2^{n_k}} 
    (1+2\delta)^{k}\delta^{-{(n_{k}-l)}},
\end{align*}
 where we have used that $n_{k-1}=n_{k}-1$ and that $\sum_{j\geq 0} \delta^{j} = (1-\delta)^{-1} \leq 1+2\delta$. 
 \end{proof}

\subsection{Proof of Proposition \ref{prop:UpsT}}
\label{sec:proofUpsT}
By lemma \ref{lem:aik}, for all $\mathbf J\in \mathcal D^{>2}(m+1,q)$, we have $a_0^{m+1}\leq (1+2\delta)^{m+1} (\overline{b}_t)^{m+1} (2\delta^{-1})^{1+\sum_{i=2}^{m+1} \mathbf{1}_{i\textrm{ is bad}}} $. Therefore by \eqref{eq:InitialboundDigammaNew} and \eqref{eq:AI},
\begin{equation*}
\begin{aligned}
&\sup_{X_0} \Upsilon_t(X_0)  \leq\frac{2}{\pi^2}\sum_{m=0}^\infty  \sum_{\mathbf J \in \mathcal D^{>2}(m+1,q)} \frac{(1+2\delta)^{m+1}(1+\varepsilon_N^{\hyperlink{symbol: star}{\diamondsuit}})^{m}} {(\delta/2)^{1+\sum_{i=2}^{m+1} \mathbf{1}_{i\textrm{ is bad}}}} (\overline{b}_t)^{m+1} \bar{\sigma}_N^{2m}{\sigma}_N^{2,J_{m+1}}.
\end{aligned}
\end{equation*}
Next, we sum over $J_{m+1}$ above, which by Lemma \ref{lem:yokkata}, leads to the bound
	\begin{equation} \label{eq:boundwithq3sigma3}
		\sup_{X_0} \Upsilon_t(X_0) \leq 4 q^{3}\bar{\sigma}_N^3 \sum_{\mathbf J \in \mathcal D(m,q)} \frac{(1+2\delta)^{m+1}} {(\delta/2)^{1+\sum_{i=2}^{m+1} \mathbf{1}_{i\textrm{ is bad}}}} (1+\varepsilon_N^{\hyperlink{symbol: star}{\diamondsuit}})^{m}(\overline{b}_t)^{m+1} \bar{\sigma}_N^{2m}.
	\end{equation}
 Recall  $\mathcal D^{(n)} (m,q)$ from Section~\ref{sec:largeqquasicrit}, which is defined as the subset of all $\mathbf{J}\in \mathcal D(m,q)$ such that $n_{\mathrm{small}}(\mathbf{J})=n$ and satisfies $|\mathcal D^{(n)}(m,q)| \leq {\binom {m}{n}} \binom{q}{2}^{m-n} (2qL)^n$. By Lemma \ref{lem:gamma_n(I)}, summing over all possible values of $n=n_{\mathrm{small}}(\mathbf J)$ yields 
\begin{equation} \label{eq:supUps}
\sup_{X_0} \Upsilon_t(X_0) \leq 4 q^3 \bar \sigma_N^3 \sum_{m=0}^\infty  (1+2\delta)^{m+1}  (1+\varepsilon_N^{\hyperlink{symbol: star}{\diamondsuit}})^{m}(\overline{b}_t)^{m+1} \bar{\sigma}_N^{2m} H_{m},
\end{equation}
where
\begin{align}
	H_m &:= \sum_{n=0}^{m} \binom{m}{n} \left(2 \delta^{-1}\right)^{2n+\frac{m}{L}+2} \binom{q}{2}^{m-n} (2qL)^n = (2\delta^{-1})^2 (\mathfrak C_{L,\delta,q})^m,\label{eq:valueHm}
\end{align}
where $\mathfrak C_{L,\delta,q}:=\left(2 \delta^{-1}\right)^{\frac{1}{L}} \left(8\delta^{-2} qL + \binom q 2\right)$.
Therefore, by \eqref{eq:valueHm} the sum in \eqref{eq:supUps} converges as soon as 
\begin{equation} \label{eq:conditionQCLdeltaq}
(1+2\delta)  \bar b_t \bar \sigma_N^2 \mathfrak C_{L,\delta,q} (1+\varepsilon_N^{\hyperlink{symbol: star}{\diamondsuit}}) < 1.
\end{equation}
Note that $\overline{b}_t\leq 4 b_t$. Let $\varepsilon_0' \in (0,1/2)$. If (i): $\binom q 2 \in [3,\varepsilon_0' b_t^{-1} \bar \sigma_N^{-2}]$, setting $L = 3$ and $\delta = 1/4$, then $ \bar b_t\bar \sigma_N^2 \mathfrak C_{L,\delta,q}$ is smaller than $C' \varepsilon_0'$ for $N$ large with a universal constant $C'>0$. 
Letting $\varepsilon_0'$ small enough, we obtain that \eqref{eq:conditionQCLdeltaq} is satisfied and thus by \eqref{eq:supUps},
\begin{equation} \label{eq:Gamma_TBefore}
\sup_{X_0} \Upsilon_t(X_0) \leq C\frac{   b_t q^3  \bar \sigma_N^3}{1-C'\varepsilon_0'}  \leq \frac{C\varepsilon_0'}{1-C'\varepsilon_0'} q \bar \sigma_N < 1,
\end{equation}
for $N$ large, 
where the third inequality holds for $\varepsilon_0'$ small enough, which we fix now, since $q^2 \bar \sigma_N^2$ is uniformly bounded by \eqref{eq:q,T-condition}.

Keeping this $\varepsilon_0'$, we then consider (ii): $\binom q 2 \in [\varepsilon_0' b_t^{-1} \bar \sigma_N^{-2}, \varepsilon_0 \log N]$ 
where $\varepsilon_0=\varepsilon_0(\alpha)$ is determined below. We also note that the lower bound diverges). We set $L=\sqrt q$ and fix $\delta$ small enough so that $2\delta < (1-\alpha)/4$. We distinguish two cases. {If $b_t \geq \delta^{-1}$ which entails $\bar b_t \leq (1+2\delta) b_t$}, then for $N$ large enough (depending on $\alpha,\varepsilon_0'$), we have by \eqref{eq:q,T-condition} that $(1+2\delta) \bar b_t \bar \sigma_N^2 \mathfrak C_{L,\delta,q}$ is smaller than $ (1+\alpha)/2<1$. {If $b_t \leq \delta^{-1}$, then we choose $\varepsilon_0(\alpha) \in (0,1)$ small enough so that the condition
$\binom q 2\leq \varepsilon_0 \log N$ enforces $(1+2\delta) \bar b_t \bar \sigma_N^2 \mathfrak C_{L,\delta,q}\leq (1+\alpha)/2<1$. In both cases,}
Hence, \eqref{eq:Gamma_TBefore} holds and \eqref{eq:supUps} entails
\begin{equation} \label{eq:lastGammaT}
	\sup_{X_0} \Upsilon_t(X_0) \leq  \frac{8 \delta^{-2} b_t  q^3 \bar \sigma_N^3}{\frac{1}{2}(1-\alpha) }
    \leq  \frac{ 8 \delta^{-2}}{\frac{1}{2}(1-\alpha)}q \bar \sigma_N,
\end{equation}
using \eqref{eq:q,T-condition}. {Finally, by the assumption
$q^2\leq \varepsilon_0 \log N$, we have
$\sup_{X_0} \Upsilon_t(X_0)\leq C_{\alpha} \sqrt{\varepsilon_0}$, implying \eqref{eq:supUps0} for $\varepsilon_0=\varepsilon_0(\alpha)$ small enough.
 Equation \eqref{eq:Phi>2phi} then follows from \eqref{eq:dec>2<2}, \eqref{eq:upperboundPsiN} 
 and 
\eqref{eq:lastGammaT}.

\subsection{Case $s>1$} 
To conclude, we prove:
\begin{prop}
\label{prop:UpsTs} 
Assume \eqref{eq:sub-criticality} or \eqref{eq:defNearCrit}. With $\varepsilon_0$ as in Proposition \ref{prop:UpsT},  uniformly for $(q,t)$ satisfying \eqref{eq:q,T-condition} and $q^2\leq \varepsilon_0 \log N$, for all $s\leq t$,
\begin{equation} \label{eq:Phi>2phis}
    \Psi_{s,t,q}(X_0) \leq \Phi_{s,t,q}^{\leq 2}(X_0) + \gamma_{s}  \sup_{X_0 \in (\mathbb Z^2)^q}\Phi_{1,t,q}^{\leq 2}(X_0),  
\end{equation}
where for some $C=C_\alpha>0$ and $c=c_{\alpha}=(1+\alpha)/2 \in (0,1)$,
$\gamma_{s} = C q \bar \sigma_N  F(s) e^{c \binom{q}{2} \lambda_{1,s,N}^2}$.
\end{prop}
\begin{proof}
Similarly to \eqref{eq:upperboundPsiN}, using \eqref{eq:supUps0}, we have
\aln{ \label{eq:eqOnPsist}
\Psi_{s,t,q}(X) &\leq \Phi_{s,t,q}^{\leq 2}(X) +  \Upsilon_{s,t}(X) \left(\sum_{{k=0}}^\infty \left(\sup_{X_0 \in (\mathbb Z^2)^q} \Upsilon_t(X_0)\right)^k\right) \sup_{X_0\in (\mathbb Z^2)^q} \Phi_{1,t,q}^{\leq 2}(X_0)\\
&\leq \Phi_{s,t,q}^{\leq 2}(X) + {C} \Upsilon_{s,t}(X)  \sup_{X_0\in (\mathbb Z^2)^q} \Phi_{1,t,q}^{\leq 2}(X_0) ,\notag
}
where
\[
	\begin{aligned}
		 & \Upsilon_{s,t}(X_0):=\sum_{m=0}^\infty \sigma_N^{2m} \sum_{\substack{{s}\leq a_1\leq b_1 < a_2 \leq b_2 <  \dots < a_m \leq b_m < N                                                                                   \\ \mathbf J =(i_r,j_r)_{r\leq m} \in \mathcal D(m,q)}} \sum_{\mathbf X, \mathbf Y}  \sum_{\substack{b_m<n\leq N\\K\subset \mathcal C_q, |K|\geq 2\\X_{m+1}\in (\mathbb Z^2)^{|\bar K|}}} (1+\varepsilon_N^{\flat})^{|K|}\sigma_N^{2|K|}\\
		 & \prod_{r=1}^{m} \Big\{ \sigma_N^{2}  U_N(b_r-a_r,y_r^{i_r}-x_r^{i_r}) \mathbf{1}_{x_r^{i_r}=x_r^{j_r}}  \mathbf{1}_{y_r^{i_r}=y_r^{j_r}} \prod_{i\in \{i_r,j_r\}} p_{a_r-b_{k_r(i)}}(x_r^i-y_{k_r(i)}^i)\Big\} \\
		 & \prod_{i\in \bar K} p_{a_{m+1}-b_{k_{m+1}(i)}}(x_{m+1}^i-y_{k_{m+1}(i)}^i)\prod_{(i,j)\in K} \mathbf{1}_{x_{m+1}^{i}=x_{m+1}^j}.
	\end{aligned}
\]
The only new term we have to deal with in \eqref{eq:eqOnPsist} is $\Upsilon_{s,t}$, as the other terms are controlled by \eqref{eq:mainResultWithoutL2}. Similarly to Lemma \ref{lem:FF}, we have the following.
\begin{lem}
It holds:
\begin{equation}\label{eq:InitialboundDigammaNewst}
		\sup_{X} \Upsilon_{s,t}(X)\leq  \sum_{m=0}^\infty  (1+\varepsilon_N^{\hyperlink{symbol: star}{\diamondsuit}})^{m} \bar{\sigma}_N^{2m} \sum_{
			\mathbf J  \in \mathcal D^{>2}(m+1,q)}  \sigma_N^{2,J_{m+1}} A_{s,t}(\mathbf{J}),
\end{equation}
	where
\begin{equation} \label{eq:defAbfIst}
		A_{s,t}(\mathbf J)=\int_{s}^t\int_0^t\cdots \int_0^t   F(u_1) \left(\prod_{r=2}^{m+1}
		F(u_r+\bar u_{r})\right)   F(u_{m+1})
		\mathbf{1}_{\sum_{i=1}^{m+1} u_i \leq N} \, {\rm d} u_1 \cdots {\rm d} u_{m+1}.
\end{equation}
\end{lem}
Then, by the proof of \eqref{eq:AI}, we find
\begin{equation}
    A_{s,t}(\mathbf J) \leq \sum_{i=0}^{m{+1}} \frac{a_{i}^{m{+1}}}{i!}  (\overline{b}_t)^{m+1-i} h^{(i)}\left(s\right).
\end{equation}
Proceeding as in \eqref{eq:boundwithq3sigma3}--\eqref{eq:valueHm},  with $ \mathfrak C_{L,\delta,q} := (2\delta^{-1})^{1/L} (8qL\delta^{-2}+\binom q 2)$ that verifies $H_m= 4\delta^{-2}( \mathfrak C_{L,\delta,q})^m$, we get
\al{
&\sup_{X} \Upsilon_{s,t}(X) \leq C \delta^{-2} q^3 \bar \sigma_N^3 \sum_{m=0}^\infty ( \mathfrak C_{L,\delta,q})^m \bar \sigma_N^{2m} (1+\varepsilon_N^{\hyperlink{symbol: star}{\diamondsuit}})^m \sum_{i=0}^{m}   (\overline{b}_t)^{m+1-i} \frac{h^{(i)}\left(s\right)}{i!}\\
&\leq C  \delta^{-2} F(s) q \bar \sigma_N \sum_{i=0}^{\infty} \bar \sigma_N^{2i}( \mathfrak C_{L,\delta,q})^{i} \frac{(\lambda_{1,s,N})^{2i}}{i!}  \sum_{m=i}^\infty ( \mathfrak C_{L,\delta,q})^{m-i}  (1+\varepsilon_N^{\hyperlink{symbol: star}{\diamondsuit}})^m   \bar \sigma_N^{2(m-i)} (\overline{b}_t)^{m-i} ,
}
where in the second inequality we have used that $\bar b_t q^2 \bar \sigma_N^2 \leq \mathcal O(1)$ by \eqref{eq:q,T-condition}.
 Using again that $\bar b_t  \mathfrak C_{L,\delta,q} \bar \sigma_N^2 \leq (1+\alpha)/2<1$ for $N$ large with a suitable choice of $L$ and $\delta$, which implies in turn that $\mathfrak C_{L,\delta,q} \bar \sigma_N^2 \leq (1+\alpha)/2$, we find $\sup_{X} \Upsilon_{s,t}(X) \leq C q\sigma_N  F(s) e^{c \binom{q}{2} \lambda_{1,s,N}^2}$ for $c=(1+\alpha)/2$, leading to \eqref{eq:Phi>2phis}.

\section{Proof of Theorem
\ref{th:mainTheorem} in the lower $q$ case}
  \label{sec:mainTh,lowq}
Assume \eqref{eq:sub-criticality} or \eqref{eq:defNearCrit}. In this section, we prove Theorem
\ref{th:mainTheorem} when $q^2\leq \varepsilon_0 \log N$ with $\varepsilon_0$ small enough. To begin with, we assume $\varepsilon_0$ to be smaller than $\varepsilon_0$ from Proposition \ref{prop:UpsT}.  
Combining \eqref{eq:mainResultWithoutL2} and  \eqref{eq:Phi>2phis}   gives
\begin{equation} \label{eq:BoundPsistqlambda}
\begin{split}
  &  \sup_X \Psi_{s,t,q}(X)\\
  & \leq  c q^2 \bar\sigma_N^2 F(s) e^{ c \binom{q}{2} \lambda_{s,N}^2} c_\star b_t e^{(\binom{q}{2}-1)\lambda^2_{t,N}(1+\varepsilon^{\hyperlink{deltaCondition}{\triangle}}_N)} 
    + c_{\star} {b_t} e^{(\binom{q}{2}-1)\lambda^2_{s,t,N}(1+\varepsilon_N^{\hyperlink{deltaCondition}{\triangle}})},
    \end{split}
\end{equation}
where $\varepsilon_N^{\hyperlink{deltaCondition}{\triangle}}$ is independent of $s,t$.
By \eqref{eq:BoundPsistqlambda} and since $\lambda^2_{t,N} = \lambda^2_{s,t,N} + \lambda^2_{s,N}$,  
 it is enough to prove that for all $t\leq N$,
\begin{equation} \label{eq: goal s,t situation}
\sup_{s\leq t}  F(s)  e^{ c \binom{q}{2} \lambda_{s,N}^2}    e^{(\binom{q}{2}-1)\lambda^2_{s,N}(1+\varepsilon_N^{\hyperlink{deltaCondition}{\triangle}})} = \mathcal O(1).
\end{equation}
Since $F(s)= s^{-1} e^{\lambda_s^2} $ we have
$
F(s)  e^{ c \binom{q}{2} \lambda_{s,N}^2 }    e^{(\binom{q}{2}-1)\lambda^2_{s,N}(1+\varepsilon_N^{\hyperlink{deltaCondition}{\triangle}})} \leq s^{-1} e^{c' \binom q 2 \lambda_{s,N}^2}$ 
with $c'=c'(\alpha)>0$.  
Let $\varepsilon_1\in (0,1)$. If $\bar \sigma_N^2 \log s\leq 1-\varepsilon_1$, then 
by Taylor expansion,  we have 
\[
\binom q 2 \lambda_{s,N}^2 = \binom q 2 \log \frac{1}{1-\bar \sigma_N^2 \log s} \leq C({\varepsilon_1}) \binom q 2 \bar \sigma_N^2 \log s, 
\]
with some constant $C(\varepsilon_1)>0$, so that since $q^2 \bar \sigma_N^2\leq \mathcal O(\varepsilon_0)$, by choosing $\varepsilon_0$ small enough (depending on $\varepsilon_1)$, it holds that $\sup_{\bar \sigma_N^2 \log s\leq 1-\varepsilon_1} s^{-1} e^{c' \binom q 2 \lambda_s^2}$ is uniformly bounded from above in $N$. 
 Then, we note that in both \eqref{eq:sub-criticality} and \eqref{eq:defNearCrit}, there exists $c>1$ such that for all $N\geq 1$, $c^{-1}\log N \leq \bar \sigma_N^{-2} \leq c \log N$. 
Hence, if $\bar \sigma_N^2 \log s > 1-\varepsilon_1$,
  by  $\binom q 2\leq c\log N/b_t$ from \eqref{eq:q,T-condition}, then 
 we obtain for all $s\leq t$,
\begin{equation} \label{eq:sqbt}
s^{-1} e^{c' \binom q 2 \lambda_{s,N}^2} \leq e^{-c^{-1}(1-\varepsilon_1)\log N} e^{c'c \frac{\log b_t}{b_t} \log N }.
\end{equation}
Since $t\geq s$, we have $b_t \geq \log \varepsilon_1^{-1}$ which diverges as $\varepsilon_1\to 0$, hence if we let $\varepsilon_1$ small enough, by \eqref{eq:sqbt} we can make $\sup_{s\leq t, 1-\varepsilon_1\leq \bar \sigma_N^2 \log s} s^{-1} e^{c' \binom q 2 \lambda_{s,N}^2}$ arbitrarily small for large $N$, uniformly on \eqref{eq:q,T-condition}. 
In both cases, we get \eqref{eq: goal s,t situation}.
\end{proof}

\section{Proof of Theorem
  \ref{th:mainTheorem} in the higher $q$ regime}
  \label{sec:proofMainThSubcrit}
We complete the proof of Theorem
  \ref{th:mainTheorem} by dealing with the remaining case  
\eqref{eq:q,T-condition} with $q^2 \geq \varepsilon_0 \log N$, where $\varepsilon_0>0$ {is fixed throughout the section}. We assume these two conditions throughout the section.  

We focus solely on  \eqref{eq:sub-criticality}, as the higher $q$ regime is relevant only at sub-criticality, see Remark \ref{rk:quasiLowerq}}. The proof can be easily adapted to \eqref{eq:defNearCrit} and higher $q$.

{Again, we proceed by comparing $\Psi_{s,t,q} = \mathrm{E} [e^{\psi_{s,t,q}(\beta_N)}]$ with $\Phi_{s,t,q}^{\leq p_0}(X)$} 
and conclude using Proposition \ref{prop:mainProp}-\ref{item:(i)}. 
We use a strategy similar in spirit to \cite[Section 2.3]{Cosco2021MomentsUB}, where the case  $p_0=2$ is considered.

Let $\gamma_0>12$. We begin by proving that for all $(q,t)$ satisfying \eqref{eq:q,T-condition} and $q^2\geq \varepsilon_0 \log N$,
\begin{equation} \label{eq:boundOnM}
	 \Psi_{t,q}^{\star}:=\sup_X \Psi_{1,t,q}(X)\leq 2 c_{\star} b_t e^{\lambda_{t,N}^2 (\binom q 2 -1)(1 + \varepsilon_N^{\hyperlink{symbol: blacksquare}{\blacksquare}})+2q^2 c\gamma_0  \frac{\log q}{\log N}} ,
\end{equation}
with $c=c(\mu)$,  
{$c_\star$ as in {Proposition \ref{prop:mainProp}-\ref{item:(i)}} and $\varepsilon_N^{\hyperlink{symbol: blacksquare}{\blacksquare}}$ as in Lemma \ref{lem:gaussToG} below.}
In Section \ref{sec:removingp_0+_step1}, we show \eqref{eq:boundOnM}. 
 We conclude the proof of Theorem \ref{th:mainTheorem} in Section \ref{sec:removingp_0+}.

Let $F_n^i := \{\exists j\in \Iintv{1,q}\setminus\{i\}: S_n^i=S_n^j\}$ and $m_n^q := \sum_{i=1}^q \mathbf{1}_{F_n^i}$. We 
define the event "there are no more than $p$ particles involved in an intersection at time $n$" as
$A_{n,p} := \{m_n^q \leq p\}$.
Then, let $B_{s,t,p} := \cap_{n=s}^t A_{n,p}$ and $B_{k,p} := B_{1,t,p}$. 
\begin{lem} \label{lem:gaussToG} \hypertarget{symbol: blacksquare}
{Assume \eqref{eq:sub-criticality}.} Let $p_0\in \mathbb N\setminus \{1\}$ and $c_\star,\varepsilon_0$ from Proposition \ref{prop:mainProp}-\ref{item:(i)}. For all $u>0$ such that $\hat \beta u < 1$, for all $s\leq t $, {$q^2\geq \varepsilon_0 \log N$ and \eqref{eq:q,T-condition},}
\begin{equation} \label{eq:nop0+} 
	\sup_{X\in (\mathbb Z^2)^q}{\rm E}_{X}^{\otimes q}\left[e^{u^2 \psi_{s,t,q}(\beta_N)} \mathbf{1}_{B_{s,t,p_0}}\right] \leq c_{\star} b_t e^{\lambda_{s,t,N}^2(u {\beta_N})  \binom q 2(1+\varepsilon_N^{\hyperlink{symbol: blacksquare}{\blacksquare}})},
\end{equation}
where for all $u>0$, $\varepsilon_N^{\hyperlink{symbol: blacksquare}{\blacksquare}}=\varepsilon_N^{\hyperlink{deltaCondition}{\triangle}}+\varepsilon_N^{\hyperlink{symbol: clubsuit}{\clubsuit}}$, with   $\varepsilon_N^{\hyperlink{deltaCondition}{\triangle}}$ satisfying the property $(\Delta)$ of Definition \ref{def:DeltaCondition}, {$\varepsilon_N^{\hyperlink{symbol: clubsuit}{\clubsuit}}=0$ if $\mu= \mathcal N(0,1)$, $\varepsilon_N^{\hyperlink{symbol: clubsuit}{\clubsuit}}=\mathcal{O}((\log N)^{-1/2})$ {uniformly} otherwise}, and   $\lambda_{s,t,N}^2(u \beta_N)$ is defined as $\lambda_{s,t,N}^2$ with $u{\beta_N}$ in place of  ${\beta_N}$.
\end{lem}
\begin{proof} 
We get by \eqref{eq:boundLambda} that there exists $c=c(p_0,\hat \beta,\mu)$
such that for all $u>0$, for all $X\in (\mathbb Z^2)^q$ and $s\leq t$,
\begin{equation} \label{eq:compToGauss}
	{\rm E}_{X}^{\otimes q}\left[e^{{u^2}\psi_{s,t,q}(\beta_N)} \mathbf{1}_{B_{s,t,p_0}}\right]\leq
	{\rm E}_{X}^{\otimes q} \left[e^{(1+c\beta_N)u^2\beta_N^2\psi^{\star}_{s,t,q}}\mathbf{1}_{B_{s,t,p_0}}\right].
\end{equation}
  Indeed, by \eqref{eq:boundLambda}, there exists $c_0=c_0({\mu},p_0)>0$ such that under $B_{s,t,p_0}$, 
\begin{align*}
    \psi_{s,t,q}(\beta_N) \leq (1+c_0\beta_N) \psi_{s,t,q}^{\mu_0}(\beta_N) =(1+c_0\beta_N) \beta_N^2 \psi_{s,t,q}^{\star},
\end{align*}
where we refer to Remark \ref{rk:GaussianLambda} for the equality. {The constant $c_0$ can be set to $0$ for Gaussian weights.}
{Now, set $\tilde \beta_N := \sqrt{(1+c_0\beta_N)}u \beta_N$.} It follows from the proof of Proposition \ref{prop:expansion}, in particular \eqref{eq: discrete magical formula} with the indicator function $\mathbf{1}_{B_{s,t,p_0}}$, that
${\rm E}_{X}^{\otimes q}[e^{
{\tilde{\beta}_N^2}\psi^{\star}_{s,t,q}} \mathbf{1}_{B_{s,t,p_0}}]\leq \Phi_{s,t,q}^{\leq p_0}(\tilde \beta_N)$.
We thus get from Proposition \ref{prop:mainProp}-\ref{item:(i)}, where we set  
$\varepsilon_N^{\triangle,u}:=\varepsilon_N^{\hyperlink{deltaCondition}{\triangle}}(\tilde \beta_N)$ {(which satisfies $(\Delta)$ for all $u>0$)},
\begin{equation} \label{eq:boundPsi}
	\sup_{X_0 \in (\mathbb Z^2)^q} {\rm E}_{X_0}^{\otimes q}\left[e^{
    {\tilde{\beta}_N^2}\psi^{\star}_{s,t,q}} \mathbf{1}_{B_{s,t,p_0}}\right] \leq c_\star {b_t} e^{\lambda_{s,t,N}^2({\tilde \beta_N}) (\binom{q}{2}-1)(1+\varepsilon_N^{\triangle,u})},
\end{equation}
Let $\tilde \sigma_N^2 := e^{\tilde \beta_N^2}-1=(e^{u^2 \beta_N^2}-1) + \mathcal O((\log N)^{-3/2})$. We have
\[
\lambda_{t,N}^2(\tilde \beta_N)-\lambda_{t,N}^2( u\beta_N) = \log \left(1+\frac{1}{\pi} \frac{(e^{u^2\beta_N^2}-1) - \tilde \sigma_N^2}{1-\bar \sigma_N^2 \log t}
\log t\right) = \mathcal O((\log N)^{-1/2}),
\]
uniformly in $t\leq N$, since $\limsup_N \sup_{t\leq N} \frac{1}{1-\bar \sigma_N^2 \log t}\leq \frac{1}{1-\hat \beta^2}$.
Combining \eqref{eq:boundPsi} and  \eqref{eq:compToGauss},   
leads to the statement of the lemma.
\end{proof}
\begin{remark} \label{rk:EstimateMsta}
As this will be useful later, we mention that the proof of \eqref{eq:boundOnM} given below also applies to $\Psi^{\mu,u}_{t,q}(X) := {\rm E}_X[e^{u^2\psi_{t,q}(\beta_N)}]$ for $u \hat \beta < 1$ and $\lambda_{t,N}^2(u {\beta_N})$. In the proof, we keep $u=1$ for simplicity of notation.
\end{remark}

\subsection{Proof of \eqref{eq:boundOnM}}
\label{sec:removingp_0+_step1}
In this section, we always assume \eqref{eq:sub-criticality}, \eqref{eq:q,T-condition} {and $q^2\geq \varepsilon_0 \log N$}. In particular, those imply that $q^2 = \mathcal O(\log N)$, which we will use repeatedly. 
 Let $\gamma_0>{12}$ and $r_0=\lceil q^{\gamma_0} \rceil$. 
 We set
$\tau_{r_0}:=\inf \left\{n\geq r_0 : m_n^q > p_0\right\}$. We start by decomposing
\begin{align}
	 \Psi_{t,q}(X)
	 & = {\rm E}_X\left[e^{\psi_{t,q}(\beta_N)} \mathbf {1}_{\tau_{r_0} >t}\right] + \sum_{k=r_0}^{t}{\rm E}_X\left[e^{\psi_{t,q}(\beta_N)} \mathbf {1}_{\tau_{r_0}=k}\right]\nonumber \\
	 & =: Q(X) + R(X).\label{eq:decAB}
\end{align}
Set $\lambda_{t,N} := \lambda_{1,t,N}$. By Markov's property,
\begin{align*}
	Q(X) & \leq {\rm E}_X^{\otimes q}\left[e^{\psi_{r_0,q}(\beta_N)}\right] \sup_{Y} {\rm E}_{Y}^{\otimes q} \left[e^{\psi_{t-r_0,q}(\beta_N)} \mathbf{1}_{B_{t-r_0,p_0}} \right],
\end{align*}
so that by \eqref{eq:nop0+} and \eqref{eq:AprioriLemma2},
using that $\lambda^2_{t-r_0,N}\leq \lambda_{t,N}^2$
(recall \eqref{eq:def_lambda_T_Nbis}),
\begin{equation} \label{eq:boundOnA(X)}
\begin{split}
    	\sup_{X} Q(X) &\leq  e^{c q^2 \frac{\log r_0 }{\log N}} c_\star b_t e^{\lambda_{t,N}^2 (\binom q 2-1)(1+\varepsilon_N^{\hyperlink{symbol: blacksquare}{\blacksquare}})} \\
        &\leq  c_\star b_t e^{\lambda_{t,N}^2 (\binom q 2-1) (1+\varepsilon_N^{\hyperlink{symbol: blacksquare}{\blacksquare}}) +  2 q^2 c\gamma_0  \frac{\log q}{\log N}},
        \end{split}
\end{equation}
for some $c=c(\mu)>0$. 
For $R(X)$, we use again Markov’s property to find that
\begin{equation} \label{eq:decR}
	\begin{aligned}
		R(X)  \leq  \Psi_{t,q}^{\star}\sum_{k=r_0}^{t}   {\rm E}_{X}^{\otimes q} \left[e^{\psi_{k,q}(\beta_N)}  \mathbf {1}_{\tau_{r_0}=k} \right] = \Psi_{t,q}^{\star}(R_1(X) + R_2(X)),
	\end{aligned}
\end{equation}
where for  $r_1=2r_0$,
we set
\begin{equation} \label{eq:defR2}
	R_1(X) := \sum_{k=r_0}^{r_1}   {\rm E}_{X}^{\otimes q} \left[e^{\psi_{k,q}(\beta_N)}  \mathbf {1}_{\tau_{r_0}=k} \right], R_2(X) := \sum_{k=r_1+1}^{t}    {\rm E}_{X}^{\otimes q} \left[e^{\psi_{k,q}(\beta_N)}  \mathbf {1}_{\tau_{r_0}=k} \right].
\end{equation}
To deal with $R_1$ and $R_2$, we rely on \eqref{eq:AprioriLemma2} which ensures that
for some $c=  c(\hat \beta)>1$,
\begin{equation} \label{eq:useOfAprioriBound}
	\forall k \in \llbracket 1, 
	r_1\rrbracket: \quad  \sup_{X\in (\mathbb Z^2)^q} {\rm E}_X^{\otimes q} \left[e^{2\psi_{k,q}(\beta_N)}\right] \leq  c k^{c}.
\end{equation}
We start with $R_1(X)$. By the Cauchy-Schwarz inequality,
\[
	R_1(X)   \leq \sum_{k=r_0}^{r_1}   {\rm E}_{X}^{\otimes q} \left[e^{\psi_{k,q}(\beta_N)}  \mathbf {1}_{A_{k,p_0}^c} \right]   \leq \sum_{k=r_0}^{r_1}   {\rm E}_{X}^{\otimes q}  \left[e^{2\psi_{k,q}(\beta_N)} \right]^{1/2} \rmP_{X}^{\otimes q} \left(A_{k,p_0}^c\right)^{1/2}.
\]
By \eqref{eq:useOfAprioriBound} and \eqref{eq:boundPAc},
 for $p_0=p_0(\hat \beta)$ large enough so that {$1+(c-p_0/3)/2 < -p_0/12$},
\begin{equation} \label{eq:boundR1}
	\sup_{X\in (\mathbb Z^2)^q} R_1(X)\leq C q^{p_0}\sum_{k=r_0}^{r_1} k^{{(c-p_0/3)/2}} \leq C q^{p_0} r_0^{1 + {(c-p_0/3)/2}} \leq C q^{-c'p_0},
\end{equation}
where $c'>0$, where we have used that $r_0\geq q^{\gamma_0}$ with $\gamma_0 > {12}$ in the last inequality.

We next focus on $R_2(X)$. Observe that for all $k\geq r_0$, we have 
$\{\tau_{r_0}=k\} = B_{r_0,k-1,p_0} \cap A_{k,p_0}^c$,
hence by Markov’s property at time $r_0$,
\[
{\rm E}_{X} \left[e^{\psi_{k,q}(\beta_N)}  \mathbf {1}_{\tau_{r_0}=k} \right] = {\rm E}_X\left[e^{\psi_{r_0,q}(\beta_N)}{\rm E}_{S_{r_0}} \left[e^{\psi_{k-r_0,q}(\beta_N)} \mathbf{1}_{B_{k-r_0-1,p_0}} \mathbf{1}_{A_{k-r_0,p_0}^c}\right]\right].
\]
By \eqref{eq:useOfAprioriBound} with $k=r_0$, using $e^{\psi_{k,q}} \leq e^{\beta_N^2 q^2} e^{\psi_{k-1,q}}$ and $\beta_N^2 q^2 = \mathcal{O}(1)$ under \eqref{eq:q,T-condition},
\begin{align*}
	R_2(X) & \leq  {\rm E}_X\left[e^{\psi_{r_0,q}(\beta_N)}\right] \sum_{k=r_1+1}^{t}   \sup_{Y} {\rm E}_{Y} \left[e^{\psi_{k-r_0,q}(\beta_N)} \mathbf{1}_{B_{k-r_0-1,p_0}} \mathbf{1}_{A_{k-r_0,p_0}^c}\right]\\
    & \leq C r_0^c \sum_{k=r_1+1}^{t}   \sup_{Y} {\rm E}_{Y} \left[e^{\psi_{k-r_0-1,q}(\beta_N)} \mathbf{1}_{B_{k-r_0-1,p_0}} \mathbf{1}_{A_{k-r_0,p_0}^c}\right],
\end{align*}
Hence, by H\"older's inequality (with $a^{-1}+b^{-1}=1$, $\sqrt{a}\hat \beta <1$), 
\begin{align*}
    R_2(X) & \leq  Cr_0^c  \sum_{k=r_1+1}^{t}   \sup_{Y} {\rm E}_{Y} \left[e^{a\psi_{k-r_0-1,q}(\beta_N)} \mathbf{1}_{B_{k-r_0-1,p_0}}\right]^{1/a} \sup_{Y}\rmP_{Y} \left(A_{k-r_0,p_0}^c\right)^{1/b}  \\
	 & \leq Cr_0^c q^{p_0/b} \sum_{k=r_1+1}^{t}  e^{a^{-1}(1+o_N(1)) \binom{q}{2}\lambda_{k,N}^2(\sqrt{a}\beta_N)} (k-r_0)^{-p_0/(3b)},
\end{align*}
where in the third inequality, we have used \eqref{eq:nop0+} and \eqref{eq:boundPAc}.
Observing that $\binom{q}{2}\lambda_{k,N}^2{(\sqrt{a}\beta_N)}  \leq    c' \log k$  
with some $c'=c'(\hat \beta)$  since $\sqrt{a}\hat \beta <1$ {and $q^2=\mathcal O(\log N)$}.  Then, as $r_1=2r_0$, we have $\frac{k}{k-r_0} \leq 2$ for $k\geq r_1$, and hence
\begin{align}
	\sup_{X} R_2(X)&\leq C r_0^c q^{p_0/b}\sum_{k=r_1+1}^{N} 2^{p_0/(2b)} k^{  c'-p_0/(3b)} \nonumber\\
 & \leq C  (2q)^{p_0} r_0^{c} \,r_1^{c'-p_0/(3b)+1} \leq  C q^{-c''p_0}, \label{eq:boundR2}
\end{align}
with $c''=c''(\hat \beta)>0$, where the last inequality holds by letting $p_0=p_0(\hat \beta)$ large enough depending $c,c'$.
Putting together \eqref{eq:decAB}, \eqref{eq:boundOnA(X)}, \eqref{eq:decR}, \eqref{eq:boundR1} and \eqref{eq:boundR2}, we find that
\begin{equation*}
	\Psi_{t,q}^\star \leq \sup_X Q(X) + \sup_X R(X)\leq 2c_\star e^{\lambda^2_{t,N} \binom q 2(1 + \varepsilon_N^{\hyperlink{symbol: blacksquare}{\blacksquare}}) +  q^2 c\gamma_0  \frac{\log q}{\log N}} + C\Psi_{t,q}^\star q^{-c'''p_0},
\end{equation*}
for $c'''=c'''(\hat \beta)>0$, which  implies \eqref{eq:boundOnM} as $Cq^{-c'''p_0} \leq 1/2$ for $p_0$ large enough (or $N$ large enough since we assume $q^2\geq \varepsilon_0 \log N$).

\subsection{Conclusion} \label{sec:removingp_0+}
{First}, assume 
$s\geq q^{\gamma}$ 
with {$\gamma > 3$}.
Set $\tau_{s}:=\inf \left\{n\geq s : m_n^q > p_0\right\}$. We start by decomposing
\begin{align}
	\Psi_{s,t,q}(X)
	 & = {\rm E}_X\left[e^{\psi_{s,t,q}(\beta_N)} \mathbf {1}_{\tau_{s} >t}\right] + \sum_{k=s}^{t}{\rm E}_X\left[e^{\psi_{s,t,q}(\beta_N)} \mathbf {1}_{\tau_{s}=k}\right]\nonumber \\
	 & =: Q(X) + R(X).\label{eq:decAB'}
\end{align} 
By \eqref{eq:nop0+},
$\sup_{X} Q(X) \leq {c_\star} b_t e^{\lambda_{s,t,N}^2 (\binom q 2-1)(1+\varepsilon_N^{\hyperlink{symbol: blacksquare}{\blacksquare}})}$. 
Then by Markov’s property,
\begin{equation} \label{eq:decR'}
\begin{aligned}
R(X)  \leq  \Psi_{t,q}^\star \sum_{k=s}^{t}   {\rm E}_{X} \left[e^{\psi_{k,q}(\beta_N)}  \mathbf {1}_{\tau_{s}=k} \right].
\end{aligned}
\end{equation}
Using again that $\{\tau_{s}=k\} = B_{s,k-1,p_0} \cap A_{k,p_0}^c$ 
and $e^{\psi_{s,k,q}} \leq C e^{\psi_{s,k-1,q}}$,
we get by H\"older's inequality (with $a^{-1}+b^{-1}=1$, $\sqrt{a}\hat \beta <1$), 
\begin{align*}
    R(X) & \leq  C\Psi_{t,q}^\star  \sum_{k=s}^{t}   \sup_{Y} {\rm E}_{Y} \left[e^{a\psi_{s,k-1,q}(\beta_N)} \mathbf{1}_{B_{s,k-1,p_0}}\right]^{1/a} \sup_{Y}\rmP_{Y} \left(A_{k,p_0}^c\right)^{1/b}  \\
	 & \leq C\Psi_{t,q}^\star q^{p_0/b} \sum_{k=s}^{t}  e^{a^{-1} \binom{q}{2}\lambda_{k,N}^2(\sqrt{a}\beta_N)(1+\varepsilon_N^{\hyperlink{symbol: blacksquare}{\blacksquare}})} k^{-p_0/(3b)}.
\end{align*}
where in the third inequality, we have used \eqref{eq:nop0+} and \eqref{eq:boundPAc}.
Now, we use again that $\binom{q}{2}\lambda_{k,N}^2{(\sqrt{a}\beta_N)}  \leq   c' \log k$, so that by \eqref{eq:boundOnM},  
by letting $p_0=p_0(\hat \beta)$ large enough,
\begin{align*}
	\sup_{X} R(X)&\leq C q^{p_0/b} \Psi_{t,q}^\star \sum_{k=s}^{N}  k^{  c'-p_0/(3b)} \\
    &\leq C q^{p_0/b} b_t e^{\lambda_{t,N}^2 (\binom q 2-1) (1+\varepsilon_N^{\hyperlink{symbol: blacksquare}{\blacksquare}}) + 2  q^2 c \gamma_0 \frac{\log q}{\log N}}   s^{c' - p_0/(3b)+1}.
\end{align*}
Since $\lambda_{t,N}^2= \lambda_{s,t,N}^2 +\lambda_{s,N}^2$ and  $e^{\binom{q}{2}\lambda_{s,N}^2(1+\varepsilon_N^{\hyperlink{symbol: blacksquare}{\blacksquare}})} e^{2q^2 c\gamma_0 \frac{\log q}{\log N}} \leq s^{c''}q^{c''}$ with some $c''=c''(\hat{\beta})>0$, 
using the assumption $s\geq q^{\gamma}$ with {$\gamma > 3$}, we obtain that $\sup_X R(X) \leq \varepsilon b_t e^{\lambda_{s,t,N}^2 (\binom q 2-1)(1+\varepsilon_N^{\hyperlink{symbol: blacksquare}{\blacksquare}})}$ for any arbitrary $\varepsilon>0$ by letting $p_0$ large enough. 
This 
concludes the proof in the case $s\geq q^{\gamma}$ by  \eqref{eq:decAB'}-\eqref{eq:decR'}.
For $s\leq q^{\gamma}$, we bound $\Psi_{s,t,q}\leq \Psi_{t,q}$ and use \eqref{eq:boundOnM} with $e^{2q^2 c\gamma_0  \frac{\log q}{\log N}}\leq q^{c_1}$ with $c_1=c_1(\mu)>0$. 

\section{Proof of the main results}
\subsection{Proofs of Theorem \ref{th:SCnost}, Theorem \ref{th:QCnost} and Corollary \ref{cor:QCfiniteq}}
The results follow from Theorem \ref{th:mainTheorem} by setting $s=1,t=N$. Indeed, for Theorem \ref{th:SCnost}, 
we use that with 
$\lambda^2 = \lambda^2(\hat \beta) = \frac{1}{1-\hat \beta^2}$,
we have 
$b_N \leq  e^{\lambda^2(\hat \beta)(1+o_N(1))}$ as well as $\lambda_{1,1,N}^2 = \lambda^2(\hat \beta) +  \mathcal O( (\log N)^{-1})$. 
In quasi-criticality, setting $\lambda_N^2 = \log \frac{\log N}{\vartheta_N}$, it holds that 
 $b_N \leq e^{\lambda_N^2(1+o_N(1))}$ and   $\lambda_{1,1,N}^2 = \lambda_N^2(\hat \beta)(1 +  \mathcal O(\vartheta_N^{-1}))$. 
We turn to Corollary \ref{cor:QCfiniteq}.
Let $M>0$ and assume $q\leq M$. Again by Theorem \ref{th:mainTheorem}, referring to Proposition \ref{prop:mainProp} regarding the error terms,
$\mathbb{E}[W_N^q] \leq (1+o_N(1)) \frac{\binom q 2}{\binom q 2 - 1} e^{\lambda_N^2 \binom q 2 (1+o_N(1))}$, where $|o_N(1) \lambda_N^2|$ is bounded from above by  $C (\vartheta_N^{-1} + b_N \bar \sigma_N^2) \log \frac{\log N}{\vartheta_N}$. Observe that $b_N\bar \sigma_N^2 = \mathcal O (\vartheta_N^{-1})$ under \eqref{eq:defNearCrit}, which leads to \eqref{eq:QCfiniteq}. Finally, if $\vartheta_N\geq C\log \log N$, the error $o_N(1) \lambda_N^2$ is bounded.

\subsection{Proof of Corollary \ref{cor:LDSC} and Corollary \ref{cor:LDQC}.}
\subsubsection{Proof of Corollary \ref{cor:LDSC}} 
For any integer \( q_N > 0 \), using Markov's inequality, we have
\begin{align*}
\mathbb{P}\left( \log W_N(\beta_N)  \geq \lambda x_N \right)
&= \mathbb{P}\left( W_N(\beta_N) \geq e^{ \lambda x_N} \right) \leq e^{- q_N  \lambda x_N} \mathbb{E}[ W_N(\beta_N)^{q_N} ].
\end{align*}
From Theorem \ref{th:SCnost}, under the condition
\[
\limsup_{N\to\infty} \binom{q_N}{2} \frac{1}{\log N} < \frac{1 - \hat{\beta}^2}{\hat{\beta}^2},
\]
we have
$
\mathbb{E}[ W_N(\beta_N)^{q_N} ] \leq e^{ \lambda^2 \binom{q_N}{2} (1 + o_N(1)) }.$
Let us choose \( q_N = \frac{x_N}{\lambda} \). We can verify that this choice satisfies the condition of Theorem \ref{th:SCnost} from our assumption on \( x_N \):
\[
\limsup_{N\to\infty} \binom{q_N}{2} \frac{1}{\log N} = \limsup_{N\to\infty} \frac{q_N^2}{2\log N} = \limsup_{N\to\infty} \frac{x_N^2}{2\lambda^2 \log N} < \frac{1 - \hat{\beta}^2}{\hat{\beta}^2}.
\]
Substituting back into the inequality, we get
\begin{align*}
\mathbb{P}\left( \log W_N(\beta_N)  \geq \lambda x_N \right)
&\leq e^{ - q_N  \lambda x_N + \lambda^2 \binom{q_N}{2} (1 + o_N(1)) }.
\end{align*}
With \( q_N = \frac{x_N}{\lambda} \),
  the exponent simplifies to
\begin{align*}
-x_N^2 + \left( \frac{x_N^2}{2}\right)(1+o_N(1)) = -\frac{x_N^2}{2}(1-o_N(1)).
\end{align*}
Thus,
$
\mathbb{P}\left( \log W_N(\beta_N)  \geq \lambda x_N \right) \leq e^{ -\frac{x_N^2}{2}  (1 - o_N(1))}.
$

Next, we consider the lower bound. We take $\varepsilon>0$ arbitrary and set $q_N= (1+\varepsilon )x_N/\lambda$. 
We have
\begin{align}
\label{eq: lower tail estimate}
\mathbb{P}\left( \log W_N(\beta_N)  \geq \lambda x_N \right)&\geq \frac{\mathbb{E}\Big[ e^{- q_N \log W_{N}} e^{ q_N \log W_{N}} \mathbf{1}_{\{ \log W_{N}\in [\lambda x_N, (1+2\varepsilon)\lambda x_N ]\}}\Big]}{\mathbb{E}\Big[ W_{N}^{q_N}\Big]} \mathbb{E}\Big[ W_{N}^{q_N}\Big]\\
&\geq e^{- q_N  (1+2\varepsilon)\lambda x_N}  \frac{\mathbb{E}\Big[  W_{N}^{q_N}\,\mathbf{1}_{\{ \log W_{N}\in [\lambda x_N, (1+2\varepsilon)\lambda x_N] \}}\Big]}{\mathbb{E}\Big[ W_{N}^{q_N}\Big]} \mathbb{E}\Big[ W_{N}^{q_N}\Big].\notag
\end{align}
Moreover, we have
\al{
\mathbb{E}\Big[  e^{ q_N \log W_{N}} \mathbf{1}_{ \log W_{N}>  (1+2\varepsilon)\lambda x_N }\Big]& = \mathbb{E}\Big[  e^{(1+\varepsilon)q_N\log W_{N}}\, e^{-\varepsilon q_N\log W_{N}} \mathbf{1}_{\log  W_{N}>  (1+2\varepsilon)\lambda x_N}\Big]\\
&\leq  e^{-\varepsilon q_N  (1+2\varepsilon)\lambda x_N} \mathbb{E}\Big[  e^{(1+\varepsilon)q_N\log W_{N}}\Big].
}
For any $\varepsilon'>0$, by Theorem~\ref{th:SCnost},  this can be further bounded from above by 
\al{
&    \exp{\Big(- \varepsilon q_N  (1+2\varepsilon)\lambda x_N+(2\varepsilon+\varepsilon^2) q_N^2  \lambda^2/2+ q_N^2  \lambda^2/2 +\varepsilon' q_N^2 \Big)}.
}
We note that 
\al{
 & x_N^{-2} \left( -\varepsilon q_N  (1+2\varepsilon)\lambda x_N + (2\varepsilon  +\varepsilon^2) q_N^2  \lambda^2 /2 {+\varepsilon' q_N^2}\right)\\
 &= -2 \varepsilon (1+2\varepsilon) (1+\varepsilon) + (2\varepsilon  +\varepsilon^2) (1+\varepsilon)^2 {+\varepsilon'(1+\varepsilon)} =-\varepsilon^2+\varepsilon^4 {+\varepsilon'(1+\varepsilon)},
}
which is negative if $\varepsilon$ is less than $1$ {and $\varepsilon'$ is taken small enough}.  Similarly,  we estimate
\al{
&\mathbb{E}\Big[  e^{ q_N \log W_{N}} \mathbf{1}_{ \log W_{N}\leq \lambda x_N}\Big]\leq e^{\varepsilon \lambda x_N q_N} \mathbb{E}\Big[  e^{ (1-\varepsilon)q_N \log W_{N}}\Big]\\
&\leq \exp\Big(\varepsilon \lambda x_N q_N+((1-\varepsilon)q_N)^2  \lambda^2 /2 +\varepsilon' {q_N^2}\Big)\\
& =   \exp\Big(\varepsilon \lambda x_N q_N-(q_N^2-((1-\varepsilon)q_N)^2)  \lambda^2 /2+ q_N^2  \lambda^2 /2  +\varepsilon' {q_N^2}\Big).
}
We note that 
\al{
&x_N^{-2}    \left(\varepsilon\lambda x_N q_N-(q^2_k - ((1-\varepsilon)q_N)^2) \lambda^2/2 {+\varepsilon' q_N^2}\right) \\
& = \varepsilon (1+\varepsilon) - (2\varepsilon -\varepsilon^2) (1+\varepsilon)^2 + {+\varepsilon' (1+\varepsilon)} = 
-\varepsilon - 2\varepsilon^2  + \varepsilon^4 {+\varepsilon' (1+\varepsilon)},
}
which is also negative for $\varepsilon<1$ {and $\varepsilon'$ small enough}. 

Finally, we use \cite[Theorem~1.1]{cosco2023momentsLB} which gives $\mathbb{E}[W_N^{q_N}] \geq \exp\{\lambda^2 \binom {q_N} 2 (1-o_N(1))\}$. Putting things together, by taking $\varepsilon'$ small enough depending on $\varepsilon$, we have
\al{
\lim_{n\to\infty} \frac{\mathbb{E}\Big[  e^{ q_N \log W_{N}} \mathbf{1}_{\{ \log W_{N}\not\in [\lambda x_N, (1+2\varepsilon)\lambda x_N] \}}\Big]}{\mathbb{E}\Big[ W_{N}^{q_N}\Big]} =0.
}
Together with \eqref{eq: lower tail estimate} and \cite[Theorem~1.1]{cosco2023momentsLB}, this completes the proof.
\subsubsection{Proof of Corollary \ref{cor:LDQC}}
Using the same approach as before, for any integer \( q_N > 0 \),
\[
\mathbb{P}\left( \log W_N(\beta_N)  \geq \lambda_N x_N \right) \leq e^{- q_N \lambda_N x_N} \mathbb{E}[ W_N(\beta_N)^{q_N} ].
\]
From Theorem \ref{th:QCnost}, under the condition
$
\limsup_{N\to\infty} \binom{q_N}{2} \vartheta_N^{-1} < 1,
$
we have
$
\mathbb{E}[ W_N(\beta_N)^{q_N} ] \leq e^{ \lambda_N^2 \binom{q_N}{2} (1 + o_N(1)) }$. 
Let us choose \( q_N = \frac{x_N}{\lambda_N} \). The condition from Theorem \ref{th:QCnost} becomes
\[
\limsup_{N\to\infty} \binom{q_N}{2} \vartheta_N^{-1} = \limsup_{N\to\infty} \frac{q_N^2}{2\vartheta_N} = \limsup_{N\to\infty} \frac{x_N^2}{2\lambda_N^2 \vartheta_N} < 1,
\]
which holds by assumption. Substituting, we have
\begin{align*}
&- q_N   \lambda_N x_N + \lambda_N^2 \binom{q_N}{2} (1 + o_N(1)) = - x_N^2+ \lambda_N^2 \frac{ \left( \frac{x_N}{\lambda_N} \right) \left( \frac{x_N}{\lambda_N} - 1 \right) }{2} (1 + o_N(1)),
\end{align*}
which equals $-\frac{x_N^2}{2} (1 - o_N(1))$. Thus, 
$\mathbb{P}\left( \log W_N(\beta_N)  \geq \lambda_N x_N \right) \leq e^{ -\frac{x_N^2}{2}  (1 - o_N(1))}$.

\section{Discussion and open questions} \label{sec:discussion}
\begin{enumerate}
    \item In the quasi-critical regime, it is an open question whether $c_N (\log W_N(\beta_N)-d_N) \to \mathcal N(0,1)$ for some proper choice of scaling $c_N,d_N$.  We emphasize that our moment bounds (even for a finite $q$) do not lead to such a result (and rather yield large deviation estimates). This is due to the fact that $W_N(\beta_N)\to 0$ in this regime.
    \item In the quasi-critical regime, we expect that the limit supremum in \eqref{eq:qfixedQC} should be a limit equal to 1, in analogy to the sub-critical regime. 
    \item What should be the optimal $q$ threshold until which $q$-moments continue to behave as $e^{\lambda_N^2 \binom q 2}$? In \cite[Theorem 1.3]{cosco2023momentsLB}, it is proven that moments start to grow much faster (at least $e^{c\binom q 2 N/\log N}$) as soon as $q^2\geq c (\log N)^2$.

    \item 
Theorem \ref{th:mainTheorem} can be interpreted as an asymptotic independence property between the pairwise interactions of $q$ simple random walks. Indeed, one expects that for $q$ up to some $q_N\to\infty$ threshold -- in particular under \eqref{eq:q,T-condition} -- the family of  overlaps $ I_{s,t,N}^{i,j} = \beta_N^2 \sum_{n=s}^t \mathbf{1}_{S_n^i=S_n^j}$ for $(i,j)\in \mathcal C_q$ are asymptotically independent, that is, for any test functions $g^{i,j}$,
\begin{equation} \label{eq:asymInde}
    {\rm E}^{\otimes q}\left[e^{\sum_{(i,j)\in \mathcal C_q} g^{i,j}(I_{s,t,N}^{i,j})}  \}\right] \sim  \prod_{i,j\in \mathcal C_q} {\rm E}^{\otimes q}\left[e^{g^{i,j}(I_{s,t,N}^{i,j})}\right],
\end{equation}
where the right-hand side is asymptotically given by $e^{\lambda^2_{s,t,N}\binom q 2}$. Up to first order, Theorem \ref{th:mainTheorem} entails \eqref{eq:asymInde} with $g^{i,j} = \mathrm{Id}$. In \cite{LD23}, \eqref{eq:asymInde} is derived for finite $q\in \mathbb N\setminus\{3\}$ in the sub-critical regime. Thus, one can wonder about the extension of this result to the quasi-critical regime and more generally, to a system of $q=q_N\to\infty$ walks.

\end{enumerate}

\appendix
\section{Some technical lemmas}
 \begin{lem}\label{lem: comparison for general points}
      Given   $X=(x^i)\in (\mathbb{Z}^2)^p$ with $p\in \mathbb{N}$, it holds
\al{
{\rm E}^{\otimes p}_X[e^{\sum_{n=s}^t \sum_{x\in\Z^2} \Lambda_{N_n^q(x)}(\beta_N)}] &\leq {\rm E}^{\otimes p}_{0,\ldots,0}[e^{\sum_{n=s}^t \sum_{x\in\Z^2} \Lambda_{N_n^q(x)}(\beta_N)}].
} 
In particular, the same holds for $\mu_0$, i.e., 
\al{
{\rm E}^{\otimes p}_X[e^{\beta_N^2 \sum_{n=s}^t \sum_{(i,j)\in \mathcal C_q} \mathbf{1}_{\{S_n^i=S_n^j\}}}] &\leq {\rm E}^{\otimes p}_{0,\ldots,0}[e^{\beta_N^2 \sum_{n=s}^t \sum_{(i,j)\in \mathcal C_q} \mathbf{1}_{\{S_n^i=S_n^j\}}}].
}
 \end{lem}
\begin{proof}
Apply $\mathbb{E}[\prod_{i=1}^p Z_i]\leq \prod_{i=1}^p \mathbb{E}[|Z_i|^p]^{1/p}$ with $Z_i= {\rm E}_{x^i}[e^{\beta_N\sum_{n=1}^{N}(\omega(n,S_n)-\beta_N^2/2)}]$ to get
\al{
{\rm E}^{\otimes p}_X\Big[e^{\sum_{n=s}^t \sum_{x\in\Z^2} \Lambda_{N_n^q(x)}(\beta_N)}\Big] &= \mathbb{E}\left[\prod_{i=1}^p Z_i \right]\leq  \mathbb{E}[Z_1^p] = {\rm E}^{\otimes p}_{0}\Big[e^{\sum_{n=s}^t \sum_{x\in\Z^2} \Lambda_{N_n^q(x)}(\beta_N)}\Big].
}   
\end{proof}

Define:
\begin{equation} \label{eq:defU_N}
	\begin{aligned}
		U_N(n) & :=
		\begin{cases}
			\sigma_N^2 {\rm E}_{0}^{\otimes 2}\left[ e^{\Lambda_2(\beta_N)\sum_{l=1}^{n-1} \mathbf 1_{S^1_l = S^2_l}} \mathbf{1}_{S^1_n= S^2_n} \right] & \text{if } n\geq 1, \\                                             1   & \text{if } n=0.
		\end{cases}
	\end{aligned}
\end{equation}
\begin{prop} \label{prop:factor2ndMoment} {Assume  \eqref{eq:sub-criticality} or \eqref{eq:defNearCrit}}.
\begin{enumerate}[label=(\roman*)]
\item	For $N$ large enough, for all $m\leq N$,
\begin{equation} \label{eq:borneTheta}
		\sum_{n=0}^m U_N(n) = {\rm E}_0^{\otimes 2}\left[e^{\Lambda_2(\beta_N) \sum_{n=1}^t \mathbf{1}_{S_n^1=S_n^2}}\right] \leq \frac{1}{1-\sigma_N^2 R_m}. 
\end{equation}
\item There is $C>0$ such that, as $N\to\infty$ and for all $n\leq N$,
	\begin{equation} \label{eq:boundP2P2ndMomentC}
		U_N(n) \leq C  p_{2n}(0)  \frac{\sigma_N^2}{\left(1-\sigma_N^2 R_n \right)^2}.
	\end{equation} 
\end{enumerate}
\end{prop}
\begin{proof} 
The equality in \eqref{eq:borneTheta} is obtained by the proof of  \cite[Eq.\ (58)]{Cosco2021MomentsUB}. Then, observing that for $N$ large enough,  \eqref{eq:sub-criticality} or \eqref{eq:defNearCrit} implies $R_T \sigma_N^2 <1$ for all $T\leq N$, we obtain the upper bound in \eqref{eq:borneTheta} from the proof of \cite[Eq.\ (55)]{Cosco2021MomentsUB} and \eqref{eq:boundP2P2ndMomentC} from the proof of  \cite[Eq.\ (59)]{Cosco2021MomentsUB}.
\end{proof}
Let $p_n^{\star}:=\sup_x p_n(x)$. Recall $U_N(n)$ in \eqref{eq:defU_N} and $F(u)= \frac{1}{u}  \frac{1}{1-\bar{\sigma}_N^2 \log u}$ in \eqref{eq-F}.
\begin{lem} \label{eq:key_lemma}
	Assume \eqref{eq:sub-criticality} or \eqref{eq:defNearCrit}. For all $w\in \frac{1}{2} (\mathbb N\setminus\{1\})$,
	\begin{equation} \label{eq:key_bound}
		\sum_{v=0}^N U_N(v) p_{v+2w}^\star \leq  (1+\varepsilon_N^{\hyperlink{symbol: star}{\diamondsuit}}) \frac{1}\pi F(w),
	\end{equation}
	where $\varepsilon_N^{\hyperlink{symbol: star}{\diamondsuit}} = \mathcal O((\log N)^{-1})$ in \eqref{eq:sub-criticality} and $\varepsilon_N^{\hyperlink{symbol: star}{\diamondsuit}} = \mathcal O(\vartheta_N^{-1})$ in \eqref{eq:defNearCrit}, is uniform in $w$.
\end{lem}
\begin{proof}
	In the sub-critical case $\hat \beta<1$, this appeared in the proof of \cite[Proposition 3.7]{Cosco2021MomentsUB}. The argument needs to be adapted in near-critical case \eqref{eq:defNearCrit} to account for diverging constants.
 
 We write:
	\begin{equation} \label{eq:sumOnvraw}
		\sum_{v=0}^N U_N(v) p^\star_{v+2w}
		=: S_{\leq w} + S_{> w} ,
	\end{equation}
	where $S_{\leq w}$ is the sum on the left-hand side of \eqref{eq:sumOnvraw} restricted to $v\leq \lfloor w \rfloor$ and $S_{> w}$ is the sum for $v> \lfloor w \rfloor$.
	Using \eqref{eq:borneTheta} and that $p_n^\star$ is non-increasing, we have:
	\begin{equation}
    \label{eq:Sleqw}
		S_{\leq w} \leq p^\star_{2w} \sum_{v=0}^{\lfloor w \rfloor} U_N(v)
		\leq p^\star_{2w}  \frac{1}{1- \sigma_N^2 R_{\lfloor w \rfloor}},
	\end{equation}
where by \eqref{eq:pnstar}, 
we have
	\begin{equation} \label{eq:R_n_UB}
		R_n = \sum_{k=1}^n p_{2k}(0) 
        \leq \frac{1}{\pi} (1+\log {n}),
	\end{equation}
{so that for all $n\leq N$,
\begin{equation} \label{eq:Rnlogn}
\frac{1}{1-\sigma_N^2 R_{n}} - \frac{1}{1-\bar{\sigma}_N^2 \log n} \leq \frac{\bar \sigma_N^2}{1-\sigma_N^2 R_N} \frac{1}{1-\bar{\sigma}_N^2 \log n}. 
\end{equation}
Hence, coming back to \eqref{eq:Sleqw}, and since $p^\star_{2w}\leq \frac{1}{\pi w}$ again by \eqref{eq:pnstar}, this implies with $\varepsilon_N^{o}:=\frac{\bar \sigma_N^2}{1-\sigma_N^2 R_N}$, that
\begin{equation}
\label{eq:initialestimateSleqU}
S_{\leq w}\leq \frac{1}{\pi}(1+{\varepsilon_N^{o}}) F(w),
\end{equation}
where $ \varepsilon_N^{o}= \mathcal O((\log N)^{-1})$ in \eqref{eq:sub-criticality} and $\mathcal O(\vartheta_N^{-1})$ in \eqref{eq:defNearCrit}.}
On the other hand, by \eqref{eq:boundP2P2ndMomentC} and {\eqref{eq:Rnlogn}},
	\begin{equation} \label{bound_SbiggerU}
		\begin{split}
			S_{>w}&\leq C\sum_{v=\lfloor w \rfloor +1}^N \frac{1}{v^2} \frac{\sigma_N^2}{{(}1-\sigma_N^2 R_v{)^2}} \\
   		 &\leq C{(1+\varepsilon_N^{o})^2}\int_{w}^N \frac{1}{v^2} \frac{\sigma_N^2}{(1-\bar{\sigma}_N^2 \log{v})^2} dv\\
      &\leq  C  \int_{w}^N \frac{1}{v^2} \frac{\bar{\sigma}_N^2}{(1-\bar{\sigma}_N^2 \log{v})^2} dv =      C\int_{\log w}^{\log N} \frac{1}{e^x} \frac{\bar{\sigma}_N^2}{(1-\bar{\sigma}_N^2 x)^2} dx.
		\end{split}
	\end{equation}
Under \eqref{eq:sub-criticality},  since $1\leq (1-\bar{\sigma}_N^2 x)^{-1}\leq  (1-\bar{\sigma}_N^2 \log N)^{-1}=\mathcal{O}(1)$ for any $x\leq \log N$, the integral is bounded from above by 
\[
C (w\log N)^{-1} \leq C(\log N)^{-1} F(w).
\]
Under \eqref{eq:defNearCrit}, we consider two cases separately: (1) $\log{N}-\log{w} \leq \vartheta_N$ and (2) $\log{N}-\log{w} \geq \vartheta_N$. In the first case, i.e.,  $\log{N}-\log{w} \leq \vartheta_N$, since 
 \begin{equation}\label{eq: inverse of F}
 1-\bar{\sigma}_N^2 \log w \leq 1{\CO +}\bar{\sigma}_N^2 (\vartheta_N-\log N) \leq \frac{ C\vartheta_N}{\log N}, 
 \end{equation}
 we have
\al{
\int_{\log w}^{\log N} \frac{1}{e^x} \frac{\bar{\sigma}_N^2}{(1-\bar{\sigma}_N^2 x)^2} dx& \leq \frac{\bar{\sigma}_N^2}{(1-\bar{\sigma}_N^2 \log N)^2}  \int_{\log w}^{\log N} \frac{1}{e^x} dx\\
&\leq C\frac{\log N}{w \vartheta_N^2}\leq C \vartheta_N^{-1}F(w).
}

For the second case, i.e.,  $\log{N}-\log{w} \geq \vartheta_N$, we further split the integral in two parts corresponding to $x$ larger than or equal to 
 $\log N - \vartheta_N$ or smaller.  By change of variables $v = \log N -x$ and  $s=v-2\log v$, and $\log{(N/w)}=\log N-\log w$, we first estimate
	\begin{equation*} \label{eq:firstapprx}
		\begin{aligned}
			 \int_{\log w}^{\log N - \vartheta_N} \frac{1}{e^x} \frac{\bar{\sigma}_N^2}{(1-\bar{\sigma}_N^2 x)^2} dx &= N^{-1} \int_{\vartheta_N}^{\log(N/w)}  \frac{e^v\bar{\sigma}_N^2}{(1-\bar{\sigma}_N^2 (\log N- v))^2} dv\\
         &\leq 	C N^{-1} \log N\int_{\vartheta_N}^{\log(N/w)}  \frac{e^v}{v^2} dv\\
    &\leq 	C N^{-1} \log N \int_{\vartheta_N-2\log \vartheta_N}^{\log(N/w)-2\log \log(N/w)} e^{s} ds,
    		\end{aligned}
	\end{equation*}
 where we have used  $1-\bar{\sigma}_N^2 (\log N- v)\geq v/(C\log N)$ for $v\leq \log N$ and $ ds = (1-2v^{-1})dv\geq dv/2$ for $v\geq \vartheta_N$ and $N$ large enough.  This is further bounded from above by 
  	\begin{equation*} 
   	\begin{aligned}
     & \frac{C \log N}{w(\log (N/w))^2} \leq  C \vartheta_N^{-1} F(w),
		\end{aligned}
	\end{equation*}
 where we have used $\log(N/w)\geq\vartheta_N$ and $C (1-(\log w/\log N))\geq 1-\bar{\sigma}_N^2 \log w$ whenever $\log N -\log w\geq \vartheta_N$.  Moreover, using \eqref{eq: inverse of F}, $N w^{-1}\geq e^{\vartheta_N}$ and $\vartheta_N\ll \log N$, we have
	\begin{equation*}		\begin{aligned}\label{eq:secondapprx}
		  \int^{\log N}_{\log N - \vartheta_N} \frac{1}{e^x} \frac{\bar{\sigma}_N^2}{(1-\bar{\sigma}_N^2 x)^2} dx \leq \frac{ C\log N}{\vartheta_N^2}  \int^{\log N}_{\log N - \vartheta_N} \frac{1}{e^x}   dx &\leq \frac{ C e^{\vartheta_N} \log N}{N \vartheta_N^2}   \leq C \vartheta_N^{-1} F(w).
		\end{aligned}
	\end{equation*} 
	By \eqref{eq:initialestimateSleqU} and \eqref{bound_SbiggerU}, the sum in \eqref{eq:sumOnvraw} is smaller than $F(w)(1+C\vartheta_N^{-1})$.
\end{proof}

Recall that for $J\subset \kC_q$, we have set $\bar{J}:=\cup_{(i,j)\in J}\{i,j\}$. Recall also \eqref{eq:def_UNXJI}.
\begin{lem}  \label{lem:sum_b_I} Assume \eqref{eq:sub-criticality}.
	Let $J \subset \kC_q$ and $p=|\bar J|$.
	For all $X\in \mathbb Z^{p}$ and $t\geq 0$,
	\begin{equation} \label{eq:sumOnU_N}
		\rme^{\otimes p}_{X} \left[ e^{{\Lambda_2(\beta_N)}\psi^\star_{t,\bar J}}\right] = \sum_{n=0}^t \sum_{I\subset \mathcal C_{\bar J}} U_N(n,X,  J, I).
	\end{equation}
\end{lem}
\begin{proof} If $n>0$, we obtain similarly to \eqref{eq:ChaosOnqMoment}: (let  below $n_k=n$ and $I_k=I$)
	\[
		U_N(n,X,{J},I) = \rme^{\otimes p}_X \big[\sum_{k=1}^\infty \sum_{1\leq n_1 < \dots n_{k-1} < n} \sum_{\substack{I_1,\dots,I_{k-1} \subset C_{\bar J}\\I_r\neq \emptyset}} \prod_{r=1}^k \sigma_N^{2|I_r|} \prod_{(i,j) \in I_r} \mathbf{1}_{S_{n_r}^i = S_{n_r}^j}\big].
	\]
	By expanding the left-hand side of \eqref{eq:sumOnU_N} as in \eqref{eq:ChaosOnqMoment}, we see that \eqref{eq:sumOnU_N} holds.
\end{proof}

\begin{lem} Assume \eqref{eq:sub-criticality}. {Let $\varepsilon_N^{\hyperlink{symbol: star}{\diamondsuit}}$ from Lemma \ref{eq:key_lemma}.}
\label{lem:sum_on_v} Let $\emptyset\neq J\subset \mathcal C_q$ such that $|\bar J|\leq p_0$. Then, there exists $c_0=c(p_0)>0$ such that for any $X=(x^i)\in (\mathbb Z^2)^{|\bar J|}$,  for all $(i,j)\in J$ and  all $w\in \frac{1}{2} (\mathbb N\setminus\{1\})$,
	\begin{equation} \label{eq:sum_on_v}
		\sum_{v=0}^N \sum_{\emptyset \neq I\subset \mathcal C_{\bar J}} U_N(v,X, J,I) p^\star_{v+2w} \leq (1+\varepsilon_N^{\hyperlink{symbol: star}{\diamondsuit}}) c_0^{ \mathbf{1}_{|J|\geq 2}} {\frac{1}{\pi}}F(w).
	\end{equation}
\end{lem}
\begin{proof}
	We distinguish two cases. 	If $|J|=1$, then note that $X=(x^i,x^i)$ with $\{(i,j)\}=J$, so that $\sum_{\emptyset \neq I\subset \mathcal C_{\bar J}}  U_N(v,X,J,I) \leq  U_N(v)$ by Lemma~\ref{lem: comparison for general points}, where $U_N(n)$ is defined in \eqref{eq:defU_N}.
	The right-hand side of \eqref{eq:sum_on_v} thus writes $\sum_{v=0}^N U_N(v) p^\star_{v+2w}$, which we bound using Lemma \ref{eq:key_lemma}.
    
    Suppose that $|J|>1$. In this case, we first bound $p^\star_{v+2w}$ in \eqref{eq:sum_on_v} by $cF(w)$ where $c=c(\hat \beta)>0$ by using \eqref{eq:pnstar}.
	Then, we use \eqref{eq:sumOnU_N} and 
     Lemma \ref{lem:LastLemma}, 
     to obtain \eqref{eq:sum_on_v}.
\end{proof}

For the next lemma, recall $F$ in \eqref{eq-F}.
\begin{lem} \label{lem:F'}
Assume \eqref{eq:sub-criticality} or \eqref{eq:defNearCrit}.  
For all $1\leq u \leq N$,
	\begin{equation} \label{eq:F'}
		F'(u) = -(1+\varepsilon_N^{\diamond})F(u)^2 (1-\bar{\sigma}_N^2\log u),
	\end{equation}
	where the error $\varepsilon_N^{\diamond} = \mathcal O((\log N)^{-1})$ in \eqref{eq:sub-criticality} and $\varepsilon_N^{\diamond} = \mathcal O(\vartheta_N^{-1})$ in \eqref{eq:defNearCrit}, uniformly in $1\leq u\leq N$.
\end{lem}
\begin{proof} We find:
	\[
		F'(u) = \frac{-1}
		{u^2(1-\bar{\sigma}_N^2 \log u)}
		\left(
		1-\frac{\bar{\sigma}_N^2}
		{1-\bar{\sigma}_N^2 \log u}
		\right),
	\]
    where in quasi-criticality  $\frac{\bar{\sigma}_N^2}
		{1-\bar{\sigma}_N^2 \log N} = \mathcal{O}( \vartheta_N^{-1})$, while $\frac{\bar{\sigma}_N^2}
		{1-\bar{\sigma}_N^2 \log N} =\mathcal{O} ((\log N)^{-1})$ in sub-criticality. 
        This gives \eqref{eq:F'}.
\end{proof}

\section{Random walk estimates}
Let $p_n(x) := \rmP_0(S_n= x)$. The following lemma is proved in \cite[Appendix A]{Cosco2021MomentsUB}.
\begin{lem}
	We have:
	\begin{equation}\label{eq:pnstar}
		\forall n\geq 1: \quad  \sup_{x\in \mathbb Z^2} p_n(x) =:p_n^{\star} \leq \frac{2}{\pi n}.
	\end{equation}
\end{lem}
For $S^1,\dots,S^q$ independent copies of the simple random walk,
let $F_{n}^i := \{\exists j\in {\Iintv{1,q}} \setminus\{i\}: S_n^i=S_n^j\}$, $m_n^q := \sum_{i=1}^q \mathbf{1}_{F_n^i}$ and the event "there are no more than $p$ particles involved in an intersection at time $n$" 
$A_{n,p} := \{m_n^q \leq p\}$.
\begin{lem}
	There exists $c>0$ such that for all $p,q\in \mathbb N \setminus\{1\}$, for all $n\geq 1$, we have
	\begin{equation} \label{eq:boundPAc}
		\sup_{X\in (\mathbb Z^2)^q} \rmP_{X}^{\otimes q} \left(A_{n,p}^c\right) \leq q^{p}  n^{-p/3}.
	\end{equation}
\end{lem}
\begin{proof}
	By \eqref{eq:pnstar}, we have
\[
    \rmP_{x_1,x_2}^{\otimes 2}( S_n^1=S_n^2)= \rmP_{x_1-x_2}( S_{2n}=0) \leq \frac{1}{\pi n}.
\]
	Suppose that $m_n^q>p$. We define $p_0 := \lfloor  (p+1)/3\rfloor$. Given that \(3p_0 \leq m_n^q\), there exist disjoint pairs \((i_m, i'_m)_{m=1}^{p_0}\) such that \(i_m < i'_m\) and \(S_n^{i_m} = S_n^{i'_m}\). Indeed, we can find disjoint sets \((C_s)_{s=1}^\ell\) {included in $\Iintv{1,q}$}  such that  \(\sum_{s=1}^\ell |C_s| \geq m_n^q \geq 3p_0\) and:
\begin{enumerate}
    \item for any \(i, j \in C_s\), \(S_n^i = S_n^j\),
\item  for any \(i \in C_s\) and \(j \in C_t\) with \(s \neq t\), \(S_n^i \neq S_n^j\),
\item {$|C_s|\geq 2$ for all $s\leq \ell$.}
\end{enumerate}
For each \(s\), we select disjoint pairs \((i^s_m, i'^s_m)_{m=1}^{m_s}\) such that \(i^s_m, i'^s_m \in C_s\), $i^s_m< i'^s_m $, and \(m_s \geq \lfloor |C_s| / 2 \rfloor\). Then, \((i_m^s, i'^s_m)_{m,s}\) are the desired objects.

By the union bound, we estimate
\al{
    \sup_{X\in (\mathbb Z^2)^q} \rmP_{X}^{\otimes q} \left(A_{n,p}^c\right)\leq\sum_{(i_m,i_m')_{m=1}^{p_0}} \sup_{X\in (\mathbb Z^2)^q} \prod_{m=1}^{p_0} \rmP_X(S_n^{i_m}=S_n^{i'_m})\leq q^{2p_0}  \left(\frac{1}{\pi n}\right)^{p_0},
}
where the summation runs over all sequences of disjoint pairs $(i_m,i_m')$ in $\Iintv{1,q}$ such that $i_m<i_m'$. This ends the proof.
\end{proof}

\section{The a priori bound}

Recall that $\Lambda_p^{\mu_0}(\beta) := \binom{p}{2} \beta^2$ when $\mu_0=\mathcal N(0,1)$, see Remark \ref{rk:GaussianLambda}. 
The next lemma enables us to compare $\psi_{s,t,q}$ in \eqref{eq:identity} to $\psi^{\star}_{s,t,q}$.

\begin{lem} \label{lem:estimateLambda}
For all $q_0>0$, there exists $C=C(\mu,q_0)>0$, such that for all {$p\in \mathbb N, \beta>0$ with $p \beta \leq q_0$},
	\begin{equation} \label{eq:boundLambda}
		\Lambda^\mu_p(\beta) \leq (1+C {p} \beta) \Lambda_p^{\mu_0}(\beta).
	\end{equation}
\end{lem}
\begin{proof}
Let $\omega:=\omega(i,x)$ and $p\geq 1$. As by assumption $\E \omega=0$, $\E \omega^2=1$ and $\mathbb{E}[e^{t \omega}] < \infty$ for all $t>0$, we have for $p\beta$ small enough,
	\al{
		\Lambda^\mu_p(\beta)&=\log{\E e^{p\beta \mu}}-p\log{\E e^{\beta \mu}}\\
		&= \log{\left(1+\frac{1}{2}\E (p\beta \mu)^2+\mathcal O(p^3 \beta^3)\right)}-p\log{\left(1+\frac{1}{2}\E (\beta \mu)^2+\mathcal O(\beta^3)\right)}\\
		&=\log{\left(1+\frac{(p\beta)^2}{2}  +\mathcal O(p^3 \beta^3)\right)}-p\log{\left(1+\frac{\beta^2}{2}+\mathcal O(\beta^3)\right)}\\
		& =\beta^2 \binom p 2 +\mathcal O(p^3\beta^3),
	}
which implies $|\Lambda_p(\beta) -\Lambda_p^{\mu_0}(\beta)|= \mathcal O(p^3\beta^3)=\mathcal{O}(p\beta) \Lambda_p^{\mu_0}(\beta)$.   If $p \beta$ is of order $1$, then both $\Lambda^\mu_p(\beta_N)$ and $\Lambda_p^{\mu_0}(\beta_N)$ are also of order $1$. This leads to \eqref{eq:boundLambda}.
 \end{proof}
The following lemma is an extension of Lemma 2.3 in \cite{Cosco2021MomentsUB} to general environments. It allows us to control the moments of $W_k$ for small values of $k$. Recall the definitions of $\psi_{k,q}$ and $\psi_{k,q}^{\star}$ from \eqref{eq:defPsiGauss} and \eqref{eq:defPsikq} and Remark \ref{rk:GaussianLambda}.
\begin{lem}[An a priori bound]\label{lem:aprioriBound}
	Assume sub-criticality \eqref{eq:sub-criticality}.
	Let $\ell_N>0$ be a deterministic sequence such that $\ell_N = o(\sqrt{\log N})$ as $N\to\infty$.
	Then,
	for all {$q_0>0$}, $a>0$, $k\leq e^{\ell_N}$ and {$q\leq q_0 \sqrt{\log N}$},
	\begin{equation}\label{eq:AprioriLemma1}
		\sup_{X\in \mathbb (\mathbb{Z}^2)^q}{\rm E}_X^{\otimes q}\left[ e^{a\beta_N^2\psi_{k,q}^{\star}}\right] \leq e^{\frac{{a}}{\pi}(1+\varepsilon_N)  q^2 \beta_N^2 \log (k+1)},
	\end{equation}
	for $\varepsilon_N = \varepsilon_N(\hat \beta,q_0)\to 0$ as $N\to\infty$. 
    Moreover, there exists $c=  c(\mu,q_0)>0$ such that for all $a >0$, {$q_0>0$ and $k\leq e^{\ell_N}$ and $q\leq q_0 \sqrt{\log N}$},
	\begin{equation} \label{eq:AprioriLemma2}
		\sup_{X\in (\mathbb{Z}^2)^q}{\rm E}_X^{\otimes q}\left[ e^{{a} \psi_{k,q}(\beta_N)}\right] \leq e^{c {a} q^2 \beta_N^2 \log k}.
	\end{equation}
\end{lem}
\begin{proof}
Eq. \eqref{eq:AprioriLemma1} follows from the proof of Lemma 2.3 in \cite{Cosco2021MomentsUB}. To obtain \eqref{eq:AprioriLemma2}, {we use Lemma \ref{lem:estimateLambda} which entails that there exists $c=c(\mu,q_0)$ such that $\psi_{k,q}(\beta_N) \leq c \psi^{\mu_0}_{k,q} = c \beta_N^2 \psi^{\star}_{k,q}$ for all $k\geq 1$ (see Remark \ref{rk:GaussianLambda} for the equality)}, and \eqref{eq:AprioriLemma1}.
\end{proof}

\begin{lem} \label{lem:LastLemma} Assume \eqref{eq:sub-criticality}. For all (fixed) $p_0\in \mathbb N_{\geq 3}$,
\[
    \sup_{N\in \mathbb{N}} \sup_{X\in (\mathbb{Z}^2)^{p_0}} \rme^{\otimes p_0}_{X} \left[ e^{\Lambda_2(\beta_N)\psi^\star_{N,p_0}}\right] <\infty.
\]
\begin{proof}
     By \eqref{eq:boundLambda}, for all $a>1$ such that $\sqrt a\hat \beta <1$, we have for $N$ large enough, for all $X=(x_1,\dots,x_{p_0})$,
     \[
      \rme^{\otimes p_0}_{X} \left[ e^{\Lambda_2(\beta_N)\psi^\star_{N,p_0}}\right] \leq \rme^{\otimes p_0}_{X} \left[ e^{a^2 \beta_N^2\psi^\star_{N,p_0}}\right] = \Phi^{\mu_0}_{1,N,p_0}(X,\sqrt a\beta_N).
     \]
By Remark \ref{rk:GaussianLambda}, the last quantity equals $\mathbb{E}\left[\prod_{i=1}^q W^{\mu_0}_{1,N}(x_i,\sqrt a \beta_N)\right]$, which by H\"older's inequality is bounded by $\sup_N \mathbb{E}[W_N^{\mu_0}(\sqrt a \beta_N)^{p_0}]$, that is known to be finite under \eqref{eq:sub-criticality} for all $p_0>0$ by \cite{LD23,Cosco2021MomentsUB} 
(alternatively, for the paper to be self-contained, this follows from the \emph{lower $q$ case} of Theorem \ref{th:SCnost}, whose proof (Section \ref{sec:mainTh,lowq}) does not use Lemma \ref{lem:LastLemma}. Indeed, the only place where Lemma \ref{lem:LastLemma} is used is to prove Theorem \ref{thm:sumOverXYv}, which is then only used to prove point \ref{item:(i)} of Proposition \ref{prop:mainProp}. But Section \ref{sec:mainTh,lowq} relies rather on point \ref{item:(ii)} of Proposition \ref{prop:mainProp}, that is proved using instead Proposition
\ref{prop:UBFQcrit} and whose proof does not require Lemma \ref{lem:LastLemma}).
\end{proof}
\end{lem}

\bibliographystyle{plain}
\bibliography{biblio}

\end{document}